\documentclass[lang=en]{elegantpaper}

\title{Optimal convergence analysis of arbitrary Lagrangian-Eulerian finite element
methods for two-phase Stokes flow problems with moving interface}
\author{Yi Liang $^a$, Pengtao Sun $^{b,*}$, Cheng Wang $^{a,c}$, Yan Chen $^{d}$, Jiarui Han $^{d}$}
\institute{$^a$ School of Mathematical Sciences, Tongji University, 1239 Siping Road, Shanghai 200092, China\\
$^b$ Department of Mathematical Sciences, University of Nevada Las
Vegas, 4505 Maryland Parkway, Las Vegas, NV 89154, USA\\
$^c$ Key Laboratory of Intelligent Computing and Applications
(Ministry of Education), Tongji University, 1239 Siping Road, Shanghai 200092, China\\
$^d$ Shenzhen Raymind Biotechnology Co., Ltd., Shenzhen, Guangdong
518129, China}
\date{}

\usepackage{lineno}

\allowdisplaybreaks[4]
\newcommand{\bs}[1]{\boldsymbol{#1}}

\begin{document}

\maketitle


\let\thefootnote\relax
\makeatletter\def\Hy@Warning#1{}\makeatother
\footnotetext{$^{*}$Corresponding author} \footnotetext{Email
addresses:
\href{mailto:2310285@tongji.edu.cn}{2310285@tongji.edu.cn} (Yi
Liang), \href{mailto:pengtao.sun@unlv.edu}{pengtao.sun@unlv.edu}
(Pengtao Sun),
\href{mailto:wangcheng@tongji.edu.cn}{wangcheng@tongji.edu.cn}
(Cheng Wang), \href{mailto:yanc@raymind.com}{yanc@raymind.com} (Yan
Chen), \href{mailto:han@raymind.com}{han@raymind.com} (Jiarui Han)}

\begin{abstract}
In this paper, an arbitrary Lagrangian-Eulerian (ALE)-based finite
element method (FEM) is developed and studied in a monolithic
framework for a class of two-phase Stokes flow problems with moving
interfaces and jump coefficients, where the mixed finite element
approximation to Stokes moving interface problems is established and
analyzed in both semi- and fully discrete schemes based on the ALE
formulation. The key analytical technique involves a specific
$H^1$-projection associated with the ALE-induced mesh motion due to
the evolving interface. Optimal convergence properties of the
proposed $H^1$-projection and its ALE-time derivative are proved in
both $H^1$ and $L^2$ norms, with which optimal $H^1$ and $L^2$ error
estimates are obtained as well for both semi- and fully discrete
mixed finite element approximations to the studied Stokes moving
interface problem.
Numerical experiments are carried out to validate all derived
theoretical results. The \textcolor{black} {theoretical analysis,
which is developed for Taylor-Hood mixed elements in this paper, can
also be adapted to MINI mixed elements and extended to more general
two-phase flow problems.}

\end{abstract}

\section{Introduction}
Two-phase flow and the associated moving interface problems arise
widely in natural sciences and engineering applications, such as
bubble and droplet dynamics
\cite{mokhtarzadeh1985dynamics,anjos2020ale,chen2023two}, oil-water
mixing flows
\cite{bian2016numerical,pouraria2016numerical,li2018numerical},
contaminated soil remediation
\cite{lu2013mathematical,dong2025modeling}, biofluid transport
\cite{beg2013homotopy,beg2013differential}, and fluid-structure
interactions (FSI)
\cite{wang2018dynamic,gao40research,hu2025theoretical}. In these
numerous fields, the interaction between fluids and evolving
interfaces often determines the overall dynamical behavior of the
system. Due to the fact that the evolution of interfaces is usually
accompanied by discontinuities in physical parameters,
time-dependent domain deformation, and complex coupling mechanisms,
both theoretical analysis and numerical simulation of such problems
remain fundamental and challenging topics in computational fluid
dynamics.

The full mathematical models of practical two-phase flows problems
are typically governed by nonlinear Navier-Stokes equations coupled
with interface conditions and additional physical effects such as
surface tension. These models and their numerical approaches
associated with either body-fitted mesh methods or body-unfitted
mesh methods are often too complicated for rigorous analysis and
efficient computation, especially when the interface motion is taken
into considerations. Therefore, a simplified model plays an
important role for developing and analyzing appropriate numerical
approaches for moving interface problems of two-phase flow in both
theoretical and computational aspects for the first time. In this
paper, we consider a class of moving interface problems of two-phase
Stokes flow with jump coefficients, which can be regarded as a
canonical linearized model for two-phase flows with evolving
interfaces. In this model, the fluid motion is governed by the
time-dependent Stokes equations, while the interface evolution
induces time-dependent domains and discontinuous coefficients across
the interface. Although simplified, this model retains the essential
mathematical difficulties arising from moving interfaces and
provides a solid foundation for the development and analysis of
numerical methods for more general two-phase flow problems.

Among numerical approaches for moving boundary/interface problems,
body-fitted mesh methods are widely regarded as one of the most
accurate and reliable strategies, as they allow the computational
mesh to conform exactly to the evolving interface, thereby enforcing
interface conditions precisely. To effectively construct and update
such moving meshes, the arbitrary Lagrangian-Eulerian (ALE) method
has become one of the most popular frameworks. Originally introduced
by Noh \cite{noh1963cel} and Hirt \cite{hirt1974arbitrary} and later
developed by Hughes \cite{hughes1981lagrangian}, Huerta
\cite{huerta1988viscous}, Belytschko \cite{belytschko1982finite},
and others in the finite element context, the ALE method combines
the advantages of both Lagrangian and Eulerian descriptions while
avoiding their respective shortcomings when dealing with different
kinds of materials on either side of the moving interface. By
introducing an ALE mapping between the current and reference
domains, it provides a flexible framework for handling moving
geometries that depend on the motion of boundary/interface all the
time. Over the past decades, the ALE-based finite
difference/volume/element methods have been successfully applied to
free-surface flows \cite{ramaswamy1987arbitrary,souli2001arbitrary},
multiphase flows
\cite{potghan2021arbitrary,duan2022energy,di20263d}, fluid-structure
interactions
\cite{nitikitpaiboon1993arbitrary,wick2013coupling,hao2021multiscale},
and transport problems on moving domains
\cite{boiarkine2011positivity,mabuza2016modeling}, among other
related applications.

In view of numerical analysis, early studies on the ALE based finite
element method (ALE-FEM) primarily focused on parabolic problems in
moving domains \cite{formaggia1999stability,gastaldi2001priori},
where stability results, optimal error estimates in $H^1$ norm, and
suboptimal estimates in $L^2$ norm were established. Subsequently,
the ALE-FEM was extended to Stokes equations in moving domains
\cite{san2009convergence}, in which a classical $H^1$-projection was
employed. However, due to discretization errors arising from the
approximation of the ALE mapping, analyses based on standard
$H^1$-projection techniques are often insufficient to fully capture
the effects of mesh motion, leading to finite element error
estimates that fail to exhibit higher-order convergence in $L^2$
norm relative to $H^1$ norm. More recently, \textcolor{black}{an
ALE-dependent} $H^1$-projection tailored to moving interface
problems has been introduced in
\cite{lan2020monolithic,lan2020finite}, which provides a more
accurate treatment of geometric errors induced by ALE mappings and
has enabled systematic stability analysis and optimal $H^1$ error
estimates for both conservative and non-conservative ALE
formulations for Stokes/parabolic moving interface problems with
jump coefficients. Furthermore, optimal-order $H^1$ error estimate
for ALE-FEMs applied to two-phase Navier-Stokes flows
\textcolor{black}{with the interface transported by the fluid
velocity} was established in \cite{li2026optimal}, while optimal
$L^2$ error estimates for semi-discrete ALE-FEMs for the Stokes
equations in evolving domains with moving boundaries were obtained
in \cite{rao2025optimal}. Additional theoretical developments and
analytical results of ALE-FEMs can be found in
\cite{gawlik2015unified,elliott2021unified,li2023optimal,lin2025optimal}.
Nevertheless, despite these advances, a rigorous and systematic
optimal convergence theory for Stokes flow problems with evolving
interface and possibly high-contrast coefficients remains incomplete
within the ALE-FEM framework, particularly in establishing optimal
$L^2$ error estimates that incorporate both mesh-motion
approximation errors and projection-related discretization effects.

Motivated by these considerations, in this paper we develop and
analyze an ALE-based mixed finite element method for a class of
Stokes problems with moving interfaces and jump coefficients. Based
on a suitable ALE mapping, both semi-discrete and fully discrete
formulations are constructed within a monolithic ALE framework. A
key ingredient of the analysis is the \textcolor{black}{use of an}
$H^1$-projection associated with the interface motion, which enables
a precise treatment of the approximation effects induced by the ALE
mesh motion. Optimal approximation properties of this projection are
established in both $H^1$ and $L^2$ norms, together with
corresponding optimal error estimates for its ALE-time derivative,
which serve as an essential analytical tool in the subsequent error
analysis for the developed semi- and fully discrete ALE-FEMs. In
particular, a major difficulty in establishing optimal $L^2$ error
estimates for the velocity lies in deriving optimal-order bounds for
the ALE-time derivative of the error between the exact velocity and
its $H^1$-projection associated with the ALE-induced moving mesh.
Such an analysis is notably challenging since an adjoint problem
corresponding to the sophisticated form satisfied by the ALE-time
derivative is too complicated to be defined, as pointed out in
\cite{lan2020monolithic,lin2025optimal}. Another difficulty arises
from the fact that, when following the integration-by-parts strategy
in \cite{li2023optimal,rao2025optimal} to establish optimal-order
approximation properties for the ALE-time derivative of the
$H^1$-projection, one requires approximation estimates on the
boundaries of each subdomain between the exact solution and its
projection, which are generally unavailable for interface problems.
To overcome these difficulties, instead of introducing an adjoint
problem for the ALE-time derivative of $H^1$-projection, we employ
only the adjoint problem associated with the $H^1$-projection itself
to derive optimal $L^2$ error estimates for the ALE-time derivative
between the exact velocity and its $H^1$-projection. On the other
hand, by additionally establishing optimal $H^{-1}$-type error
estimates for $H^1$-projection of the pressure and for the gradient
of $H^1$-projection of the velocity, we are able to derive
optimal-order approximation properties for the ALE-time derivative
of $H^1$-projection using the Aubin-Nitsche duality argument,
without relying on integration-by-parts techniques, only. Based on
these results, we establish optimal-order error estimates for the
ALE-time derivative of the $H^1$-projection, which yields optimal
error estimates in both $H^1$ and $L^2$ norms for the semi- and
fully discrete ALE finite element schemes. \textcolor{black}{The
projection used here is inspired by the ALE-dependent
$H^1$-projections developed for closely related Stokes/parabolic
moving-interface problems in \cite{lan2020monolithic,lan2020finite},
while its form is tailored to the present Stokes interface problem.
The additional contribution is the derivation of auxiliary
negative-norm estimates for the pressure projection and the gradient
of the velocity projection. These estimates enable an Aubin-Nitsche
duality argument for the ALE-time derivative without introducing a
separate adjoint problem for the differentiated projection or
relying on subdomain boundary trace estimates.} Compared with
existing works, our analysis incorporates the effect of ALE mesh
motion and achieves optimal convergence rates in $L^2$ norm.
Finally, numerical experiments are presented to validate all
theoretical results. The proposed analytical framework can be
further extended to more general two-phase flow and even FSI
problems.

The structure of this paper is organized as follows. In Section
\ref{sec:model}, we present the model description of the Stokes
moving interface problem with jump coefficients, introduce the ALE
mapping and some standard definitions, and then present the weak
formulation in the ALE framework. In Section \ref{sec:semi}, we
construct the semi-discrete ALE-FEM \textcolor{black}{as well as the
ALE-dependent $H^1$-projection of primary variables}, for which we
establish optimal convergence properties \textcolor{black}{of the
velocity, pressure and ALE-time derivative of velocity}
in both $H^1$ and $L^2$ norms. 
In Section \ref{sec:full}, we develop the fully discrete ALE-FEM and
derive its optimal \textcolor{black}{$H^1$ and $L^2$ error estimates
based on the proposed ALE-dependent} $H^1$-projection technique.
Section \ref{sec:experiment} is devoted to numerical experiments
that are carried out to validate all theoretical results. Finally,
conclusions and perspectives for future work are given in Section
\ref{sec:conclusion}.

Throughout the paper, we use
$(\phi,\tilde\phi)_\Psi=\int_\Psi\phi\tilde\phi d\bs{x}$ and
$\langle\phi,\tilde\phi\rangle_{\partial\Psi}=\int_{\partial\Psi}\phi\tilde\phi
d\bs{s}$ to denote \textcolor{black}{the $L^2$ inner products for
scalar-valued functions over} a region $\Psi\subset\mathbb{R}^d\
(d=2,3)$ and \textcolor{black}{its $(d-1)$-dimensional boundary
$\partial\Psi$}, respectively. \textcolor{black}{For vector-valued
functions, we also write $(\bs{u},\bs{v})_{\Omega^t}$ and
$\langle\bs{u},\bs{v}\rangle_{\partial\Omega^t}$ for the
corresponding $L^2$ inner products over the current domain
$\Omega^t$ and its outer boundary $\partial\Omega^t$, respectively,
and use the numerator-layout notation for gradient of vector, i.e.,
$(\nabla\bs{\varphi})_{ij}=\partial_j\varphi_i$.} In addition, we
adopt $C$ to represent a generic positive constant that remains
independent of any discretization parameters, including the mesh
size $h$ and the time step size $\Delta t$\textcolor{black}{, while
allowing it to depend on the prescribed viscosity coefficients
$\mu_1$ and $\mu_2$, the regularity bounds of exact solutions, the
bounds on the ALE mapping, and the shape-regularity constant of the
mesh family}.

\section{Model description and ALE weak form}\label{sec:model}

\subsection{Model description}

\textcolor{black}{Let $\Omega^t$ be a family of open bounded domains
in $\mathbb{R}^d$ ($d=2,3$) with a convex polygonal/polyhedral
boundary $\partial\Omega^t$, and let $T>0$. For any $t\in[0,T]$, two
subdomains $\Omega_i^t:=\Omega_i(t)\subset\Omega^t$ ($i=1,2$),
satisfying
$\overline{\Omega_1^t}\cup\overline{\Omega_2^t}=\overline{\Omega^t}$
and $\Omega_1^t\cap\Omega_2^t=\varnothing$, are separated by a
moving interface:
$\Gamma^t:=\Gamma(t)=\partial\Omega_1^t\cap\partial\Omega_2^t$ that
moves/deforms according to a prescribed interface position function,
causing the subdomains $\Omega_i^t$ ($i=1,2$) and therefore the
entire domain $\Omega^t=\Omega_1^t\cup\Omega_2^t\cup\Gamma^t$ to
evolve in time. The corresponding initial (reference/Lagrangian)
domains are denoted by $\Omega_i^0$ (resp. $\Omega^0$) with respect
to the reference coordinates $\hat{\bs{x}}_i$ (resp.
$\hat{\bs{x}}$), $i=1,2$.} An example of this type of domain
configuration with an immersed subdomain is illustrated in Figure
\ref{fig:domain}.

\begin{figure}[htbp]
    \centering
    \includegraphics[width=0.5\textwidth]{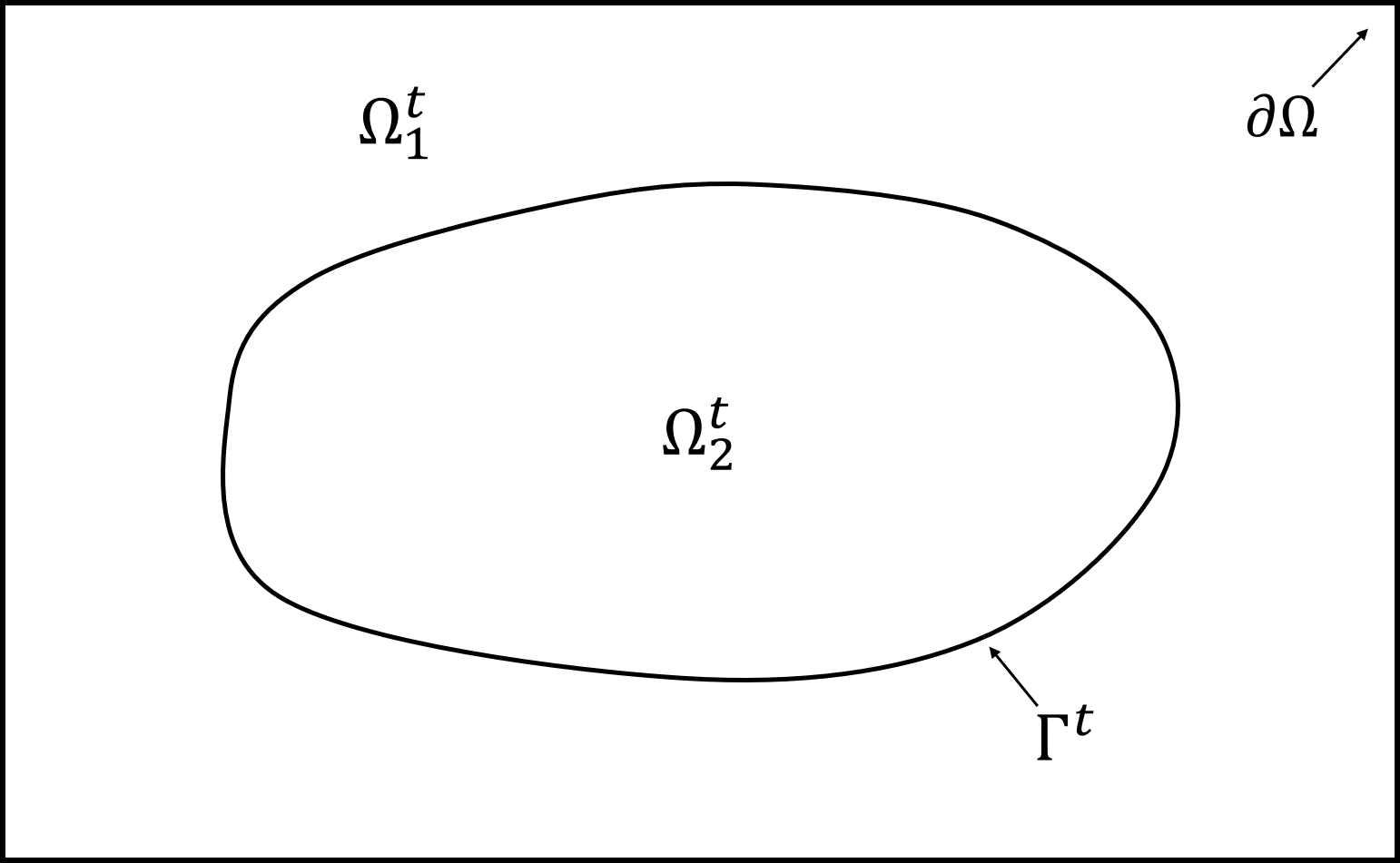}
    \caption{Schematic \textcolor{black}{current domain $\Omega^t=\Omega_1^t\cup\Omega_2^t\cup\Gamma^t$
    with $\Gamma^t$ as a moving interface}.}
    \label{fig:domain}
\end{figure}

\textcolor{black}{Within the current domain $\overline{\Omega^t}$},
we consider the following moving interface problem of two-phase
Stokes flow with jump coefficients:
\begin{alignat}{2}
    \label{prob1}
    \frac{\partial\bs{u}_1}{\partial t}-\nabla\cdot(\mu_1\nabla\bs{u}_1)+\nabla p_1&=\bs{f}_1,\qquad&&\text{in $\Omega_1^t\times(0,T]$},\\
    \label{prob2}
    \nabla\cdot\bs{u}_1&=0,&&\text{in $\Omega_1^t\times(0,T]$},\\
    \label{prob3}
    \bs{u}_1&=\bs{0},&&\text{on $\partial\Omega_1^t\setminus\Gamma^t\times(0,T]$},\\
    \label{prob4}
    \bs{u}_1(\hat{\bs{x}}_1,0)&=\bs{u}_1^0,&&\text{in $\Omega_1^0$},\\
    \label{prob5}
    \frac{\partial\bs{u}_2}{\partial t}-\nabla\cdot(\mu_2\nabla\bs{u}_2)+\nabla p_2&=\bs{f}_2,&&\text{in $\Omega_2^t\times(0,T]$},\\
    \label{prob6}
    \nabla\cdot\bs{u}_2&=0,&&\text{in $\Omega_2^t\times(0,T]$},\\
    \label{prob7}
    \bs{u}_2&=\bs{0},&&\text{on $\partial\Omega_2^t\setminus\Gamma^t\times(0,T]$},\\
    \label{prob8}
    \bs{u}_2(\hat{\bs{x}}_2,0)&=\bs{u}_2^0,&&\text{in $\Omega_2^0$},\\
    \label{prob9}
    \bs{u}_1&=\bs{u}_2,&&\text{on $\Gamma^t\times[0,T]$},\\
    \label{prob10}
    (-p_1\bs{I}+\mu_1\nabla\bs{u}_1)\bs{n}_1+(-p_2\bs{I}+\mu_2\nabla\bs{u}_2)\bs{n}_2&=\bs{g},&&\text{on
    $\Gamma^t\times[0,T]$},
    \\ \label{prob11}
\textcolor{black}{\int_{\Omega_1^t}p_1d\bs{x}_1+\int_{\Omega_2^t}p_2d\bs{x}_2}&
\textcolor{black}{=0,}&&\textcolor{black}{\text{in $[0,T]$},}
\end{alignat}
where $\mu_1>0$ and $\mu_2>0$ are two distinct jump coefficients
with $\mu_1\neq\mu_2$. In this paper, both $\mu_1$ and $\mu_2$ are
assumed to be constants. \textcolor{black}{The present model uses
the nonsymmetric stress tensor
$-p_i\bs{I}+\mu_i\nabla\bs{u}_i\,,i=1,2$, with the bulk diffusion
operators in \eqref{prob1} and \eqref{prob5} and the interface
traction in \eqref{prob10} all defined, accordingly.} Moreover,
$\bs{n}_1$ and $\bs{n}_2$ denote the outward unit normal vectors on
the boundary of $\Omega_i^t$, respectively. In addition,
$(\bs{f}_1,\bs{f}_2)\in(L^2(\Omega_1^t))^d\times(L^2(\Omega_2^t))^d$
and $\bs{g}\in(H^{1/2}(\Gamma^t))^d$. If $\bs{g}\neq 0$, it denotes
the jump of normal stress of two-phase Stokes flow across the
interface $\Gamma^t$.

\subsection{ALE mapping and ALE weak form}

With the model problem in place, we now introduce a family of
time-dependent and bijective mappings $\bs{X}_i^t\in
H^1(0,T;(W^{2,\infty}(\Omega_i^0))^d)$ ($i=1,2$), defined in each
subdomain at $t=0$ and mapping the initial (Lagrangian) domain
$\Omega_i^0$ ($i=1,2$) to the current (Eulerian) domain $\Omega_i^t$
($i=1,2$), for any $t\in(0,T]$, such that
\begin{align*}
    \bs{X}_i^t:\,\Omega_i^0&\rightarrow\Omega_i^t,\\
    \hat{\bs{x}}_i&\mapsto\bs{x}_i(\hat{\bs{x}}_i,t),
\end{align*}
is invertible, and
$\left(\bs{X}_i^t\right)^{-1}\in(W^{1,\infty}(\Omega_i^t))^d$.
\textcolor{black}{The ALE mapping $\bs{X}_i^t$ ($i=1,2$ can be
arbitrarily defined. One popular choice is through the so-called
harmonic extension technique 
\cite{formaggia1999stability,gastaldi2001priori,lan2020monolithic,lan2020finite}:
find $\bs{X}_i^t$ such that
\begin{equation}
    \label{ALE-mapping}
    \left\{\begin{aligned}
        -\Delta_{\hat{\bs{x}}_i}\bs{X}_i^t&=0,&&\text{in }\Omega_i^0,\\
        \bs{X}_i^t&=\hat{\bs{x}}_i,&&\text{on }\partial\Omega_i^0\setminus\Gamma^0,\\
        \bs{X}_i^t&=\bs{x}_\Gamma(\bs{x}(\hat{\bs{x}}_i,t),t),&&\text{on
        }\Gamma^0,\qquad\qquad i=1,2,\quad \forall t\in(0,T],
    \end{aligned}\right.
\end{equation}
where $\bs{x}_\Gamma$ denotes the prescribed interface position
function as a representation of $\Gamma^t$ at any time $t\in[0,T]$.}
Then, the ALE domain velocity can be defined as
\begin{align*}
    \bs{w}_i:\,\Omega_i^t\times(0,T]&\rightarrow\mathbb{R}^d,\\
    (\bs{x}_i,t)&\mapsto\bs{w}_i(\bs{x}_i,t)=\frac{\partial\bs{X}_i^t}{\partial t}\circ\left(\bs{X}_i^t\right)^{-1}.
\end{align*}
Further, we can define the ALE-time derivative, which takes the ALE
domain velocity into account, as follows
\begin{align*}
    \left.\frac{\partial\bs{u}_i}{\partial t}\right|_{\hat{\bs{x}}}:\,\Omega_i^t\times(0,T]&\rightarrow\mathbb{R}^d,\\
    (\bs{x}_i,t)&\mapsto\left.\frac{\partial\bs{u}_i}{\partial t}\right|_{\hat{\bs{x}}}(\bs{x}_i,t)=\frac{\partial\bs{u}_i}{\partial t}(\bs{x}_i,t)+(\bs{w}_i(\bs{x}_i,t)\cdot\nabla)\bs{u}_i(\bs{x}_i,t).
\end{align*}

To define the weak form of \eqref{prob1}-\eqref{prob11}, we first
introduce the following Sobolev spaces for $t\in[0,T]$ in the ALE
frame:
\begin{align*}
\bs{U}^t=&\left\{(\bs{v}_1,\bs{v}_2)\in(H^1(\Omega_1^t))^d\times(H^1(\Omega_2^t))^d\,\big|\,\bs{v}_i=\hat{\bs{v}}_i\circ\left(\bs{X}_i^t\right)^{-1},\ \forall\,\hat{\bs{v}}_i\in(H^1(\Omega_i^0))^d,\right.\\
    &\quad\bs{v}_i=\bs{0}\ \text{on $\partial\Omega_i^t\setminus\Gamma^t$},\ i=1,2,\ \bs{v}_1=\bs{v}_2\ \text{on $\Gamma^t$}\bigg\},\\
    \textcolor{black}{Q^t}=&\textcolor{black}{\left\{(q_1,q_2)\in L^2(\Omega_1^t)\times L^2(\Omega_2^t)\,\bigg|\,\sum_{i=1}^2\int_{\Omega_i^t}q_i\,d\bs{x}_i=0\right\}.}
\end{align*}
Define the norm $\|(\bs{u}_1,\bs{u}_2)\|_k=\left(\|\bs{u}_1\|_{k,\Omega_1^t}^2+\|\bs{u}_2\|_{k,\Omega_2^t}^2\right)^{1/2}$ for $k\ge0$. It is readily seen that the following inequalities hold
\begin{equation*}
    \frac{1}{\sqrt{2}}\sum_{i=1}^2\|\bs{u}_i\|_{k,\Omega_i^t}\le\|(\bs{u}_1,\bs{u}_2)\|_k\le\sum_{i=1}^2\|\bs{u}_i\|_{k,\Omega_i^t},
\end{equation*}
which will be used interchangeably in the subsequent analysis
whenever convenient. With these notations, the ALE weak
form of \eqref{prob1}-\eqref{prob11} is stated as follows: find
$(\bs{u}_1,\bs{u}_2)\in H^1(0,T;\bs{U}^t)$ and $(p_1,p_2)\in
L^2(0,T;Q^t)$ such that
\begin{equation}\label{weak1}
\left\{
    \begin{array}{l}
        \sum\limits_{i=1}^2\left[\left(\frac{\partial\bs{u}_i}{\partial t}\big|_{\hat{\bs{x}}},
        \bs{v}_i\right)_{\Omega_i^t}+(\mu_i\nabla\bs{u}_i,\nabla\bs{v}_i)_{\Omega_i^t}-((\bs{w}_i\cdot\nabla)\bs{u}_i,\bs{v}_i)_{\Omega_i^t}
        -(p_i,\nabla\cdot\bs{v}_i)_{\Omega_i^t}\right]=\sum\limits_{i=1}^2(\bs{f}_i,\bs{v}_i)_{\Omega_i^t}
        +\langle\bs{g},\bs{v}_1\rangle_{\Gamma^t},\\
    \sum\limits_{i=1}^2(\nabla\cdot\bs{u}_i,q_i)_{\Omega_i^t}=0,
\qquad   \forall\,(\bs{v}_1,\bs{v}_2)\in\bs{U}^t,\, (q_1,q_2)\in
Q^t.
    \end{array}
    \right.
\end{equation}

For the sake of the subsequent error analysis, we recall some useful
formulas for time derivatives of integrals over a moving domain
{\color{black}$\Phi^t\subset\Omega^t$, $t\in[0,T]$}. In particular,
we first introduce the Reynolds transport theorem
\cite{gastaldi2001priori}, which states that for any smooth function
$\psi:\,{\color{black}\Phi}^t\rightarrow\mathbb{R}$, it holds that
\begin{equation}
    \frac{d}{dt}\int_{{\color{black}\Phi}^t}\psi\,d\bs{x}=\int_{{\color{black}\Phi}^t}\left[\left.\frac{\partial\psi}{\partial t}\right|_{\hat{\bs{x}}}+\psi(\nabla\cdot\bs{w})\right]\,d\bs{x}.
\end{equation}
Furthermore, noting that for any function
$\chi:\,{\color{black}\Phi}^0\rightarrow\mathbb{R}$, it holds that
$\left.\dfrac{\partial(\chi\circ(\bs{X}^t)^{-1})}{\partial
t}\right|_{\hat{\bs{x}}}=0$, with
$\bs{X}^t:{\color{black}\Phi}^0\rightarrow{\color{black}\Phi}^t$
being the flow map, we obtain several commonly used identities, as
stated in the lemma below.

\begin{lemma}[\cite{san2009convergence,lan2020monolithic,lin2025optimal}]
\label{lem:semi0} Assume that
$\bs{\varphi}:\,{\color{black}\Phi}^t\rightarrow\mathbb{R}^d$,
$\bs{\chi}:\,{\color{black}\Phi}^0\rightarrow\mathbb{R}^d$,
$\psi:\,{\color{black}\Phi}^t\rightarrow\mathbb{R}$ and
$\bs{\omega}:\,{\color{black}\Phi}^t\rightarrow\mathbb{R}^d$ are
smooth functions. Then the following relations hold
\begin{equation}
    \label{eq:semi0-1}
    \begin{aligned}
        \frac{d}{dt}\int_{{\color{black}\Phi}^t}\nabla\bs{\varphi}:\nabla(\bs{\chi}\circ(\bs{X}^t)^{-1})\,d\bs{x}=&\int_{{\color{black}\Phi}^t}\left[\nabla\left(\left.\frac{\partial\bs{\varphi}}{\partial t}\right|_{\hat{\bs{x}}}\right):\nabla(\bs{\chi}\circ(\bs{X}^t)^{-1})+\nabla\bs{\varphi}:\nabla(\bs{\chi}\circ(\bs{X}^t)^{-1})(\nabla\cdot\bs{w})\right.\\
        &\qquad-\nabla\bs{\varphi}\left(\nabla\bs{w}+\nabla\bs{w}^\mathrm{T}\right):\nabla(\bs{\chi}\circ(\bs{X}^t)^{-1})\bigg]\,d\bs{x},
    \end{aligned}
\end{equation}
\begin{equation}
    \label{eq:semi0-2}
    \begin{aligned}
        \frac{d}{dt}\int_{{\color{black}\Phi}^t}\psi\nabla\cdot(\bs{\chi}\circ(\bs{X}^t)^{-1})\,d\bs{x}=&\int_{{\color{black}\Phi}^t}\left[\left.\frac{\partial\psi}{\partial t}\right|_{\hat{\bs{x}}}\nabla\cdot(\bs{\chi}\circ(\bs{X}^t)^{-1})+\psi\nabla\cdot(\bs{\chi}\circ(\bs{X}^t)^{-1})(\nabla\cdot\bs{w})\right.\\
        &\qquad-\psi\nabla\bs{w}:\nabla(\bs{\chi}\circ(\bs{X}^t)^{-1})^\mathrm{T}\bigg]\,d\bs{x},
    \end{aligned}
\end{equation}
\begin{equation}
    \label{eq:semi0-3}
    \begin{aligned}
        \frac{d}{dt}\int_{{\color{black}\Phi}^t}(\bs{\omega}\cdot\nabla)\bs{\varphi}\cdot(\bs{\chi}\circ(\bs{X}^t)^{-1})\,d\bs{x}=&\int_{{\color{black}\Phi}^t}\left[\left(\left.\frac{\partial\bs{\omega}}{\partial t}\right|_{\hat{\bs{x}}}\cdot\nabla\right)\bs{\varphi}\cdot(\bs{\chi}\circ(\bs{X}^t)^{-1})+(\bs{\omega}\cdot\nabla)\left(\left.\frac{\partial\bs{\varphi}}{\partial t}\right|_{\hat{\bs{x}}}\right)\cdot(\bs{\chi}\circ(\bs{X}^t)^{-1})\right.\\
        &\qquad+(\bs{\omega}\cdot\nabla)\bs{\varphi}\cdot(\bs{\chi}\circ(\bs{X}^t)^{-1})(\nabla\cdot\bs{w})\\
        &\qquad-\left(\bs{\omega}\cdot\left(\nabla\bs{w}^\mathrm{T}\nabla\right)\right)\bs{\varphi}\cdot(\bs{\chi}\circ(\bs{X}^t)^{-1})\bigg]\,d\bs{x}.
    \end{aligned}
\end{equation}
\end{lemma}
\begin{proof}
For conciseness, we only present the re-derivation of
\eqref{eq:semi0-1}, as the remaining two identities can be obtained
in a similar manner. Here, $\bs{F}^t$ denotes the Jacobian matrix of
the flow map $\bs{X}^t$, and $J^t=\det(\bs{F}^t)$ denotes its
Jacobian. To derive \eqref{eq:semi0-1}, by invoking the identities
\begin{equation}\label{identity-F}
    \frac{\partial J}{\partial t}=J(\nabla\cdot\bs{w}),\qquad\frac{\partial(\bs{F}^t)^{-1}}{\partial t}=-(\bs{F}^t)^{-1}\nabla_{\hat{\bs{x}}}\bs{w}(\bs{F}^t)^{-1},
\end{equation}
we obtain
\begin{align*}
    &\ \frac{d}{dt}\int_{{\color{black}\Phi}^t}\nabla\bs{\varphi}:\nabla(\bs{\chi}\circ(\bs{X}^t)^{-1})\,d\bs{x}=\frac{d}{dt}\int_{{\color{black}\Phi}^0}\nabla_{\hat{\bs{x}}}(\bs{\varphi}\circ\bs{X}^t)(\bs{F}^t)^{-1}:\nabla_{\hat{\bs{x}}}\bs{\chi}(\bs{F}^t)^{-1}J\,d\hat{\bs{x}}\\
    =&\ \int_{{\color{black}\Phi}^0}\left[\nabla_{\hat{\bs{x}}}\left(\frac{\partial(\bs{\varphi}\circ\bs{X}^t)}{\partial t}\right)(\bs{F}^t)^{-1}:\nabla_{\hat{\bs{x}}}\bs{\chi}(\bs{F}^t)^{-1}J+\nabla_{\hat{\bs{x}}}(\bs{\varphi}\circ\bs{X}^t)\frac{\partial(\bs{F}^t)^{-1}}{\partial t}:\nabla_{\hat{\bs{x}}}\bs{\chi}(\bs{F}^t)^{-1}J\right.\\
    &\qquad\ +\nabla_{\hat{\bs{x}}}(\bs{\varphi}\circ\bs{X}^t)(\bs{F}^t)^{-1}:\nabla_{\hat{\bs{x}}}\bs{\chi}\frac{\partial(\bs{F}^t)^{-1}}{\partial t}J+\nabla_{\hat{\bs{x}}}(\bs{\varphi}\circ\bs{X}^t)(\bs{F}^t)^{-1}:\nabla_{\hat{\bs{x}}}\bs{\chi}(\bs{F}^t)^{-1}\frac{\partial J}{\partial t}\bigg]\,d\hat{\bs{x}}\\
    =&\ \int_{{\color{black}\Phi}^t}\left[\nabla\left(\left.\frac{\partial\bs{\varphi}}{\partial t}\right|_{\hat{\bs{x}}}\right):\nabla(\bs{\chi}\circ(\bs{X}^t)^{-1})-\nabla\bs{\varphi}\nabla\bs{w}:\nabla(\bs{\chi}\circ(\bs{X}^t)^{-1})\right.\\
    &\qquad\ -\nabla\bs{\varphi}:\nabla(\bs{\chi}\circ(\bs{X}^t)^{-1})\nabla\bs{w}+\nabla\bs{\varphi}:\nabla(\bs{\chi}\circ(\bs{X}^t)^{-1})(\nabla\cdot\bs{w})\bigg]\,d\bs{x}\\
    =&\ \int_{{\color{black}\Phi}^t}\left[\nabla\left(\left.\frac{\partial\bs{\varphi}}{\partial t}\right|_{\hat{\bs{x}}}\right):\nabla(\bs{\chi}\circ(\bs{X}^t)^{-1})-\nabla\bs{\varphi}\left(\nabla\bs{w}+\nabla\bs{w}^\mathrm{T}\right):\nabla(\bs{\chi}\circ(\bs{X}^t)^{-1})\right.\\
    &\qquad\ +\nabla\bs{\varphi}:\nabla(\bs{\chi}\circ(\bs{X}^t)^{-1})(\nabla\cdot\bs{w})\bigg]\,d\bs{x}.
\end{align*}
\end{proof}

\section{Semi-discrete ALE-finite element approximation}\label{sec:semi}

\subsection{Discrete ALE mapping and semi-discrete scheme}

Let $h$ ($0<h<1$) denote the mesh size, and let
$\mathcal{T}_{h,i}^0$ be a quasi-uniform triangulation of
$\Omega_i^0$ ($i=1,2$). We further assume that no element of
$\mathcal{T}_{h,i}^0$ has two edges lying on $\partial\Omega_i^0$,
and that no element intersects the interface $\Gamma^0$. Moreover,
$\mathcal{T}_h^0=\mathcal{T}_{h,1}^0\cup\mathcal{T}_{h,2}^0$ is
conforming across the interface $\Gamma^0$.


Then, for any $t\in(0,T]$, we \textcolor{black}{define the discrete
ALE mapping, $\bs{X}_{h,i}^t:
\mathcal{T}_{h,i}^0\rightarrow\mathcal{T}_{h,i}^t$, which
approximates to $\bs{X}_i^t: \Omega_i^0\rightarrow\Omega_i^t\
(i=1,2)$ as a standard $P^k$ finite element solution to the elliptic
boundary-value problem (\ref{ALE-mapping}), i.e.,
find $\bs{X}_{h,i}^t\in \bs{S}_{h,i}:=\big\{\tilde{\bs{v}}_i\in
H^1(\Omega_i^0)\big|\tilde{\bs{v}}_i|_K\in(P^k(K))^d$, $\forall
K\in\mathcal{T}_{h,i}^0\big\}$, satisfying
$\bs{X}_{h,i}^t|_{\partial\Omega_i^0\setminus\Gamma^0}=\hat{\bs{x}}_i$
and $\bs{X}_{h,i}^t|_{\Gamma^0}=\bs{x}_\Gamma\ (i=1,2)$, such that
\begin{equation}
    \label{discreteALE}
    (\nabla_{\hat{\bs{x}}_i}\bs{X}_{h,i}^t,\nabla_{\hat{\bs{x}}_i}\tilde{\bs{v}}_i)_{\Omega_i^0}=0,
    \quad\forall\tilde{\bs{v}}_i\in\mathcal{S}_{h,i}
    \text{ with }\tilde{\bs{v}}_i|_{\partial\Omega_i^0}=0,
    \quad i=1,2.
\end{equation}
Therefore, $\mathcal{T}_{h,i}^t$ is just the consequence of the
discrete ALE mapping $\bs{X}_{h,i}^t$ mapped from
$\mathcal{T}_{h,i}^0$ onward for $t\in (0,T]$, $i=1,2$, while
$\mathcal{T}_h^t=\mathcal{T}_{h,1}^t\cup\mathcal{T}_{h,2}^t$ remains
conforming through $\Gamma^t$ that is described by $\bs{x}_\Gamma$.
The invertibility of $\bs{X}_{h,i}^t$ guarantees a one-to-one
correspondence between $\Omega_i^0$ and $\Omega_i^t$, including all
mesh nodes.} Likewise, the discrete mesh velocity is defined as
\begin{align}
    \bs{w}_{h,i}:\,\Omega_i^t\times(0,T]&\rightarrow\mathbb{R}^d,\notag\\
    (\bs{x}_i,t)&\mapsto\bs{w}_{h,i}(\bs{x}_i,t)=\frac{\partial\bs{X}_{h,i}^t}{\partial
    t}\circ\left(\bs{X}_{h,i}^t\right)^{-1},\label{dis_meshvelocity}
\end{align}
which leads to the discrete ALE-time derivative
\begin{align*}
    \left.\frac{\partial\bs{u}_i}{\partial t}\right|_{\hat{\bs{x}}}^h:\,\Omega_i^t\times(0,T]&\rightarrow\mathbb{R}^d,\\
    (\bs{x}_i,t)&\mapsto\left.\frac{\partial\bs{u}_i}{\partial t}\right|_{\hat{\bs{x}}}^h(\bs{x}_i,t)=\frac{\partial\bs{u}_i}{\partial t}(\bs{x}_i,t)+(\bs{w}_{h,i}(\bs{x}_i,t)\cdot\nabla)\bs{u}_i(\bs{x}_i,t).
\end{align*}


Similar to the classical elliptic finite element analysis, the
following error estimates hold for the ALE mapping
\textcolor{black}{$\bs{X}_i^t\in H^1(0,T;[(W^{2,\infty}\cap
W^{k+1,p})(\Omega_i^0)]^d)$} and the ALE domain velocity
\textcolor{black}{$\bs{w}_i\in H^1(0,T;[(W^{2,\infty}\cap
W^{k+1,p})(\Omega_i^t)]^d)$} ($i=1,2$), which can be found in
\cite{gastaldi2001priori,san2009convergence,ciarlet2002finite,brenner2008mathematical}
and recalled below
\begin{align}
    \label{X-conv1}
    &\|\bs{X}_i^t-\bs{X}_{h,i}^t\|_{(L^p(\Omega_i^0))^d}+h\|(\bs{X}_i^t-\bs{X}_{h,i}^t)\|_{(W^{1,p}(\Omega_i^0))^d}\le Ch^{k+1}\|\bs{X}_i^t\|_{(W^{k+1,p}(\Omega_i^0))^d},\quad\forall\,1<p<\infty,\\
    \label{w-conv1}
    &\|\bs{w}_i-\bs{w}_{h,i}\|_{(L^p(\Omega_i^t))^d}+h\|(\bs{w}_i-\bs{w}_{h,i})\|_{(W^{1,p}(\Omega_i^t))^d}\le Ch^{k+1}\|\bs{w}_i\|_{(W^{k+1,p}(\Omega_i^t))^d},\quad\forall\,1<p<\infty,\\
    \label{X-conv2}
    &\|\bs{X}_i^t-\bs{X}_{h,i}^t\|_{(L^\infty(\Omega_i^0))^d}+h\|(\bs{X}_i^t-\bs{X}_{h,i}^t)\|_{(W^{1,\infty}(\Omega_i^0))^d}\le Ch^2|\ln{h}|\|\bs{X}_i^t\|_{(W^{2,\infty}(\Omega_i^0))^d},\\
    \label{w-conv2}
    &\|\bs{w}_i-\bs{w}_{h,i}\|_{(L^\infty(\Omega_i^t))^d}+h\|(\bs{w}_i-\bs{w}_{h,i})\|_{(W^{1,\infty}(\Omega_i^t))^d}\le Ch^2|\ln{h}|\|\bs{w}_i\|_{(W^{2,\infty}(\Omega_i^t))^d}.
\end{align}
Therefore, we readily obtain the following boundedness property for the ALE mesh velocity
\begin{equation}
    \label{w-bound}
    \|\bs{w}_{h,i}\|_{(W^{1,\infty}(\Omega_i^t))^d}\le C.
\end{equation}
\textcolor{black}{Furthermore, we assume that the second time
derivative of the continuous ALE mapping satisfies
\begin{equation*}
    \frac{\partial^2\bs{X}_i^t}{\partial t^2}\in L^\infty(0,T;[W^{2,\infty}(\Omega_i^0)]^d),\quad i=1,2.
\end{equation*}
Since the finite element equation of discrete ALE mapping
(\ref{discreteALE}) is steady state defined in a time-independent
reference domain, differentiating it twice with respect to time and
applying the error estimate in \eqref{X-conv2} can similarly yield
\begin{equation*}
    \left\|\left.\frac{\partial\bs{w}_{h,i}}{\partial t}\right|_{\hat{\bs{x}}}^h\right\|_{(L^\infty(\Omega_i^t))^d}\le C.
\end{equation*}
These bounds ensure that} the ALE-induced moving mesh
$\mathcal{T}_{h,i}^t$ ($i=1,2$) remains shape-regular without mesh
distortion along the time, provided that the prescribed interface
\textcolor{black}{position function} $\bs{x}_\Gamma$ does not cause
excessively large displacement or deformation of $\Omega_i^t$
($i=1,2$).

We introduce the following discrete ALE finite element spaces using
Taylor-Hood finite elements \textcolor{black}{
\begin{align*}
    \bs{U}_h^t&=\left\{(\bs{v}_{h,1},\bs{v}_{h,2})\in\bs{U}^t\,\bigg|\,\bs{v}_{h,i}\circ\bs{X}_{h,i}^t\in(H^1(\Omega_i^0))^d,\ (\bs{v}_{h,i}\circ\bs{X}_{h,i}^t)|_K\in(P^k(K))^d,\ \forall K\in\mathcal{T}_{h,i}^0,\ i=1,2\right\},\\
    Q_h^t&=\left\{(q_{h,1},q_{h,2})\in Q^t\,\bigg|\,q_{h,i}\circ\bs{X}_{h,i}^t\in H^1(\Omega_i^0),\ (q_{h,i}\circ\bs{X}_{h,i}^t)|_K\in P^{k-1}(K),\ \forall K\in\mathcal{T}_{h,i}^0,\ i=1,2\right\},
\end{align*}}
for $k\ge2$. Then, the semi-discrete ALE finite element
approximation to (\ref{weak1}) is defined as follows: find
$(\bs{u}_{h,1},\bs{u}_{h,2})\in\bs{U}_h^t$ and $(p_{h,1},p_{h,2})\in
Q_h^t$ such that
\begin{eqnarray}
    &&\sum\limits_{i=1}^2\left[\left(\frac{\partial\bs{u}_{h,i}}{\partial t}\big|_{\hat{\bs{x}}}^h,\bs{v}_{h,i}\right)_{\Omega_i^t}+(\mu_i\nabla\bs{u}_{h,i},\nabla\bs{v}_{h,i})_{\Omega_i^t}-((\bs{w}_{h,i}\cdot\nabla)\bs{u}_{h,i},\bs{v}_{h,i})_{\Omega_i^t}\right.\notag\\
    &&\qquad-(p_{h,i},\nabla\cdot\bs{v}_{h,i})_{\Omega_i^t}\Bigg]=\sum\limits_{i=1}^2(\bs{f}_i,\bs{v}_{h,i})_{\Omega_i^t}+\langle\bs{g},\bs{v}_{h,1}\rangle_{\Gamma^t},\quad\forall\,(\bs{v}_{h,1},\bs{v}_{h,2})\in\bs{U}_h^t,\label{semi1}\\
    &&\sum\limits_{i=1}^2(\nabla\cdot\bs{u}_{h,i},q_{h,i})_{\Omega_i^t}=0,\quad\forall\,(q_{h,1},q_{h,2})\in Q_h^t.\label{semi2}
\end{eqnarray}

For the purpose of error analysis carried out in the sequel, we
assume that the following regularity properties are held for the
primary variables $\bs{u}_i,p_i$ and their ALE-time derivatives in
$\Omega_i^t\times[0,T]$ ($i=1,2$):
\begin{align}
    \label{reg1}
    &\ \,\bs{u}_i\in L^\infty(0,T;(H^{k+1}(\Omega_i^t))^d)\cap
    L^2(0,T;(W^{2,\infty}(\Omega_i^t))^d),\\
    \label{reg2}
    &\left.\frac{\partial\bs{u}_i}{\partial t}\right|_{\hat{\bs{x}}}
    \in L^2(0,T;(H^{k+1}(\Omega_i^t))^d),\\
    \label{reg3}
    &\left.\frac{\partial^2\bs{u}_i}{\partial t^2}\right|_{\hat{\bs{x}}}\in L^2(0,T;(L^2(\Omega_i^t))^d),\\
    \label{reg4}
    &\ \,p_i\in L^\infty(0,T;H^k(\Omega_i^t))\cap
    L^2(0,T;W^{1,\infty}(\Omega_i^t)),\\
    \label{reg5}
    &\left.\frac{\partial p_i}{\partial t}\right|_{\hat{\bs{x}}}\in
    L^2(0,T;H^k(\Omega_i^t)).
\end{align}

\subsection{\texorpdfstring{$H^1$}{H\^{}1}-projection of Stokes interface problems}
To obtain optimal error estimates in either $H^1$ or $L^2$ norm for
finite element approximations to time-dependent problems, an
$H^1$-projection is generally required. To this end, we introduce
\textcolor{black}{an ALE-dependent} $H^1$-projection for the Stokes
moving interface problem \eqref{prob1}-\eqref{prob11}, which
accounts for the effects of the ALE mesh motion, and is defined as
follows: find $(\tilde{\bs{u}}_1,\tilde{\bs{u}}_2)\in\bs{U}_h^t$ and
$(\tilde{p}_1,\tilde{p}_2)\in Q_h^t$ \textcolor{black}{such that}
\begin{align}
    \label{proj1}
    a(\bs{u}_1-\tilde{\bs{u}}_1,\bs{u}_2-\tilde{\bs{u}}_2;\bs{v}_{h,1},\bs{v}_{h,2})-b(\bs{v}_{h,1},\bs{v}_{h,2};p_1-\tilde{p}_1,p_2-\tilde{p}_2)&=0,\qquad\forall\,(\bs{v}_{h,1},\bs{v}_{h,2})\in\bs{U}_h^t,\\
    \label{proj2}
    b(\bs{u}_1-\tilde{\bs{u}}_1,\bs{u}_2-\tilde{\bs{u}}_2;q_{h,1},q_{h,2})&=0,\qquad\forall\,(q_{h,1},q_{h,2})\in Q_h^t,
\end{align}
where
\begin{align*}
    a(\bs{\varphi}_1,\bs{\varphi}_2;\bs{\psi}_1,\bs{\psi}_2)&=\sum_{i=1}^2\left[(\mu_i\nabla\bs{\varphi}_i,\nabla\bs{\psi}_i)_{\Omega_i^t}-((\bs{w}_{h,i}\cdot\nabla)\bs{\varphi}_i,\bs{\psi}_i)_{\Omega_i^t}+\kappa(\bs{\varphi}_i,\bs{\psi}_i)_{\Omega_i^t}\right],\\
    b(\bs{\varphi}_1,\bs{\varphi}_2;\chi_1,\chi_2)&=\sum_{i=1}^2(\nabla\cdot\bs{\varphi}_i,\chi_i)_{\Omega_i^t},
\end{align*}
provided that $\bs{w}_{h,i}$ ($i=1,2$) is given \textcolor{black}{by
(\ref{dis_meshvelocity}) and (\ref{discreteALE})}, with
$\|\bs{w}_{h,i}\|_{(L^\infty(\Omega_i^t))^d}\le M_i$ ($i=1,2$) due
to \eqref{w-bound}, and
$\kappa=\max\bigg\{\dfrac{M_1^2}{2\mu_1}+\dfrac{\mu_1}{2}+M_1,\,\dfrac{M_2^2}{2\mu_2}+\dfrac{\mu_2}{2}+M_2\bigg\}$.

Then, we have the following error estimates for this
\textcolor{black}{ALE-dependent} $H^1$-projection using Taylor-Hood
finite elements.

\begin{lemma}
\label{lem:semi1} With the regularity assumptions
\eqref{reg1}-\eqref{reg5} holding for
$((\bs{u}_1,\bs{u}_2),(p_1,p_2))$ to \eqref{weak1}, there exists a 
solution
$((\tilde{\bs{u}}_1,\tilde{\bs{u}}_2),(\tilde{p}_1,\tilde{p}_2))\in\bs{U}_h^t\times
Q_h^t$ to \eqref{proj1}-\eqref{proj2} for any $t\in[0,T]$ such that
\begin{equation}
    \label{eq:semi1}
    \sum_{i=1}^2\|\bs{u}_i-\tilde{\bs{u}}_i\|_{0,\Omega_i^t}+h\sum_{i=1}^2\left[\|\bs{u}_i-\tilde{\bs{u}}_i\|_{1,\Omega_i^t}+\|p_i-\tilde{p}_i\|_{0,\Omega_i^t}\right]\le Ch^{k+1}\sum_{i=1}^2\left[\|\bs{u}_i\|_{k+1,\Omega_i^t}+\|p_i\|_{k,\Omega_i^t}\right].
\end{equation}
\end{lemma}
\begin{proof}
Continuity of $a(\cdot,\cdot;\cdot,\cdot)$ and $b(\cdot,\cdot;\cdot,\cdot)$ follows directly from the Cauchy-Schwarz inequality, yielding
\begin{align}
    \label{a-cont}
    |a(\bs{\varphi}_1,\bs{\varphi}_2;\bs{\psi}_1,\bs{\psi}_2)|&\le C\sum_{i=1}^2\|\bs{\varphi}_i\|_{1,\Omega_i^t}\|\bs{\psi}_i\|_{1,\Omega_i^t}\le C\|(\bs{\varphi}_1,\bs{\varphi}_2)\|_1\|(\bs{\psi}_1,\bs{\psi}_2)\|_1,\\
    \label{b-cont}
    |b(\bs{\varphi}_1,\bs{\varphi}_2;\chi_1,\chi_2)|&\le C\sum_{i=1}^2\|\bs{\varphi}_i\|_{1,\Omega_i^t}\|\chi_i\|_{0,\Omega_i^t}\le C\|(\bs{\varphi}_1,\bs{\varphi}_2)\|_1\|(\chi_1,\chi_2)\|_0.
\end{align}
The coercivity of $a(\cdot,\cdot;\cdot,\cdot)$ is obtained by applying the $\varepsilon$-Young's inequality with $\varepsilon=\mu_i$, which gives
\begin{equation}
    \label{a-coer}
    \begin{aligned}
        a(\bs{\varphi}_1,\bs{\varphi}_2;\bs{\varphi}_1,\bs{\varphi}_2)&=\sum_{i=1}^2\left[(\mu_i\nabla\bs{\varphi}_i,\nabla\bs{\varphi}_i)_{\Omega_i^t}-((\bs{w}_{h,i}\cdot\nabla)\bs{\varphi}_i,\bs{\varphi}_i)_{\Omega_i^t}+\kappa(\bs{\varphi}_i,\bs{\varphi}_i)_{\Omega_i^t}\right]\\
        &\ge\sum_{i=1}^2\left[\mu_i\|\nabla\bs{\varphi}_i\|_{0,\Omega_i^t}^2-\left(\frac{\mu_i}{2}\|\nabla\bs{\varphi}_i\|_{0,\Omega_i^t}^2+\frac{M_i^2}{2\mu_i}\|\bs{\varphi}_i\|_{0,\Omega_i^t}^2\right)+\kappa\|\bs{\varphi}_i\|_{0,\Omega_i^t}^2\right]\\
        &=\sum_{i=1}^2\left[\frac{\mu_i}{2}\|\nabla\bs{\varphi}_i\|_{0,\Omega_i^t}^2+\left(\kappa-\frac{M_i^2}{2\mu_i}\right)\|\bs{\varphi}_i\|_{0,\Omega_i^t}^2\right]\ge\min_{i=1,2}\left\{\frac{\mu_i}{2}\right\}\|(\bs{\varphi}_1,\bs{\varphi}_2)\|_1^2.
    \end{aligned}
\end{equation}

\textcolor{black}{For Taylor-Hood mixed finite element spaces
$\bs{U}_h^t\times Q_h^t$ in the scenario of two-phase flow, the
coupled discrete inf-sup condition is supported primarily by the
results of ALE finite element methods for relevant two-phase flow
problems in
\cite{li2026optimal,duan2022energy,hahn2017ale,baumler2014simulation},
while the domain-decomposition and macro-element techniques
underlying such stability arguments are described in
\cite{boffi2013mixed,stenberg1984analysis,boland1983stability}.}
Combining the \textcolor{black}{well-supported} inf-sup condition
with the above continuity and coercivity properties of
$a(\cdot,\cdot;\cdot,\cdot)$ and $b(\cdot,\cdot;\cdot,\cdot)$, and
applying the classical Brezzi's theory
\cite{girault1986finite,boffi2013mixed} together with the
convergence properties of Taylor-Hood mixed elements
\cite{brezzi1991stability}, we can readily obtain the well-posedness
property as well as the following energy error estimate for the
proposed $H^1$-projection,
\begin{equation}
    \label{eq:semi1-errest1}
    \begin{aligned}
        &\ \sum_{i=1}^2\left[\|\bs{u}_i-\tilde{\bs{u}}_i\|_{1,\Omega_i^t}+\|p_i-\tilde{p}_i\|_{0,\Omega_i^t}\right]\le C(\|(\bs{u}_1-\tilde{\bs{u}}_1,\bs{u}_2-\tilde{\bs{u}}_2)\|_1+\|(p_1-\tilde{p}_1,p_2-\tilde{p}_2)\|_0)\\
        \le&\ C\left(\inf_{(\bs{v}_{h,1},\bs{v}_{h,2})\in\bs{U}_h^t}\|(\bs{u}_1-\bs{v}_{h,1},\bs{u}_2-\bs{v}_{h,2})\|_1+\inf_{(q_{h,1},q_{h,2})\in Q_h^t}\|(p_1-q_{h,1},p_2-q_{h,2})\|_0\right)\\
        \le&\ C\left(\inf_{(\bs{v}_{h,1},\bs{v}_{h,2})\in\bs{U}_h^t}\sum_{i=1}^2\|\bs{u}_i-\bs{v}_{h,i}\|_{1,\Omega_i^t}+\inf_{(q_{h,1},q_{h,2})\in Q_h^t}\sum_{i=1}^2\|p_i-q_{h,i}\|_{0,\Omega_i^t}\right)\\
        \le&\ Ch^k\sum_{i=1}^2\left[\|\bs{u}_i\|_{k+1,\Omega_i^t}+\|p_i\|_{k,\Omega_i^t}\right].
    \end{aligned}
\end{equation}

To prove an optimal $L^2$ error estimate for the $H^1$ projection of
velocity, we first define an adjoint problem of
\eqref{proj1}-\eqref{proj2} as follows: given
$(\tilde{\bs{f}}_1,\tilde{\bs{f}}_2)\in(L^2(\Omega_1^t))^d\times(L^2(\Omega_2^t))^d$,
find $(\bs{v}_1,\bs{v}_2)\in\bs{U}^t$ and $(q_1,q_2)\in Q^t$, such
that
\begin{align}
    \label{projadj1}
    a^\star(\bs{v}_1,\bs{v}_2;\bs{\psi}_1,\bs{\psi}_2)-b(\bs{\psi}_1,\bs{\psi}_2;q_1,q_2)&=\sum_{i=1}^2(\tilde{\bs{f}}_i,\bs{\psi}_i)_{\Omega_i^t},\qquad\forall\,(\bs{\psi}_1,\bs{\psi}_2)\in\bs{U}^t,\\
    \label{projadj2}
    b(\bs{v}_1,\bs{v}_2;\chi_1,\chi_2)&=0,\qquad\forall\,(\chi_1,\chi_2)\in Q^t,
\end{align}
where
\begin{equation*}
    a^\star(\bs{\varphi}_1,\bs{\varphi}_2;\bs{\psi}_1,\bs{\psi}_2)=\sum_{i=1}^2\left[(\mu_i\nabla\bs{\varphi}_i,\nabla\bs{\psi}_i)_{\Omega_i^t}+((\bs{w}_{h,i}\cdot\nabla)\bs{\varphi}_i,\bs{\psi}_i)_{\Omega_i^t}+((\nabla\cdot\bs{w}_{h,i})\bs{\varphi}_i,\bs{\psi}_i)_{\Omega_i^t}+\kappa(\bs{\varphi}_i,\bs{\psi}_i)_{\Omega_i^t}\right].
\end{equation*}
The adjoint property of \eqref{projadj1}-\eqref{projadj2} can be easily verified by noting that
\begin{equation}
    \label{propadj}
    a^\star(\bs{\varphi}_1,\bs{\varphi}_2;\bs{\psi}_1,\bs{\psi}_2)=a(\bs{\psi}_1,\bs{\psi}_2;\bs{\varphi}_1,\bs{\varphi}_2).
\end{equation}
The equivalent strong form of \eqref{projadj1}-\eqref{projadj2} can be defined as: find $(\bs{v}_1,\bs{v}_2)\in(H^2(\Omega_1^t))^d\times(H^2(\Omega_2^t))^d$ and $(q_1,q_2)\in H^1(\Omega_1^t)\times H^1(\Omega_2^t)$, such that
\begin{equation}\label{strong-adjoint}
    \left\{\begin{alignedat}{2}
        -\nabla\cdot(\mu_1\nabla\bs{v}_1)+(\bs{w}_{h,1}\cdot\nabla)\bs{v}_1+(\nabla\cdot\bs{w}_{h,1})\bs{v}_1+\kappa\bs{v}_1+\nabla q_1&=\tilde{\bs{f}}_1,\qquad&&\text{in $\Omega_1^t$},\\
        \nabla\cdot\bs{v}_1&=0,&&\text{in $\Omega_1^t$},\\
        \bs{v}_1&=\bs{0},&&\text{on $\partial\Omega_1^t\setminus\Gamma^t$},\\
        -\nabla\cdot(\mu_2\nabla\bs{v}_2)+(\bs{w}_{h,2}\cdot\nabla)\bs{v}_2+(\nabla\cdot\bs{w}_{h,2})\bs{v}_2+\kappa\bs{v}_2+\nabla q_2&=\tilde{\bs{f}}_2,&&\text{in $\Omega_2^t$},\\
        \nabla\cdot\bs{v}_2&=0,&&\text{in $\Omega_2^t$},\\
        \bs{v}_2&=\bs{0},&&\text{on $\partial\Omega_2^t\setminus\Gamma^t$},\\
        \bs{v}_1&=\bs{v}_2,&&\text{on $\Gamma^t$},\\
        (-q_1\bs{I}+\mu_1\nabla\bs{v}_1)\bs{n}_1+(-q_2\bs{I}+\mu_2\nabla\bs{v}_2)\bs{n}_2&=\bs{0},&&\text{on
        $\Gamma^t$},\\
\textcolor{black}{\int_{\Omega_1^t}q_1d\bs{x}_1+\int_{\Omega_2^t}q_2d\bs{x}_2}&
\textcolor{black}{=0,}&&
    \end{alignedat}\right.
\end{equation}
\textcolor{black}{which is essentially a steady-state Stokes problem
equipped with transport and reaction terms as well as a moving
interface. 
We assume the following regularity estimate guided by regularity
assumptions that are adopted in related studies of Stokes interface
problems in
\cite{shibata2003resolvent,wang2015new,naraparaju2025higher}:}
\begin{equation}
    \label{regadj1}
    \sum_{i=1}^2\left[\|\bs{v}_i\|_{2,\Omega_i^t}+\|q_i\|_{1,\Omega_i^t}\right]\le C\sum_{i=1}^2\|\tilde{\bs{f}}_i\|_{0,\Omega_i^t}.
\end{equation}

Let $\bs{\psi}_i=\bs{u}_i-\tilde{\bs{u}}_i$ and $\chi_i=p_i-\tilde{p}_i$ in \eqref{projadj1}-\eqref{projadj2}, take any $(\bs{v}_{h,1},\bs{v}_{h,2})\in\bs{U}_h^t$ and $(q_{h,1},q_{h,2})\in Q_h^t$, and apply the $H^1$-projection \eqref{proj1}-\eqref{proj2} together with the adjoint property \eqref{propadj}, which yields
\begin{alignat*}{2}
    &\ \sum_{i=1}^2(\tilde{\bs{f}}_i,\bs{u}_i-\tilde{\bs{u}}_i)_{\Omega_i^t}=a^\star(\bs{v}_1,\bs{v}_2;\bs{u}_1-\tilde{\bs{u}}_1,\bs{u}_2-\tilde{\bs{u}}_2)-b(\bs{u}_1-\tilde{\bs{u}}_1,\bs{u}_2-\tilde{\bs{u}}_2;q_1,q_2)&\text{(using \eqref{projadj1})}\\
    =&\ a^\star(\bs{v}_1-\bs{v}_{h,1},\bs{v}_2-\bs{v}_{h,2};\bs{u}_1-\tilde{\bs{u}}_1,\bs{u}_2-\tilde{\bs{u}}_2)-b(\bs{u}_1-\tilde{\bs{u}}_1,\bs{u}_2-\tilde{\bs{u}}_2;q_1-q_{h,1},q_2-q_{h,2})\\
    &+a^\star(\bs{v}_{h,1},\bs{v}_{h,2};\bs{u}_1-\tilde{\bs{u}}_1,\bs{u}_2-\tilde{\bs{u}}_2)-b(\bs{u}_1-\tilde{\bs{u}}_1,\bs{u}_2-\tilde{\bs{u}}_2;q_{h,1},q_{h,2})\\
    =&\ a(\bs{u}_1-\tilde{\bs{u}}_1,\bs{u}_2-\tilde{\bs{u}}_2;\bs{v}_1-\bs{v}_{h,1},\bs{v}_2-\bs{v}_{h,2})-b(\bs{u}_1-\tilde{\bs{u}}_1,\bs{u}_2-\tilde{\bs{u}}_2;q_1-q_{h,1},q_2-q_{h,2})\qquad&\text{(using \eqref{propadj})}\\
    &+a(\bs{u}_1-\tilde{\bs{u}}_1,\bs{u}_2-\tilde{\bs{u}}_2;\bs{v}_{h,1},\bs{v}_{h,2})-b(\bs{u}_1-\tilde{\bs{u}}_1,\bs{u}_2-\tilde{\bs{u}}_2;q_{h,1},q_{h,2})\\
    &-b(\bs{v}_{h,1},\bs{v}_{h,2};p_1-\tilde{p}_1,p_2-\tilde{p}_2)+b(\bs{v}_1,\bs{v}_2;p_1-\tilde{p}_1,p_2-\tilde{p}_2)\\
    &-b(\bs{v}_1-\bs{v}_{h,1},\bs{v}_2-\bs{v}_{h,2};p_1-\tilde{p}_1,p_2-\tilde{p}_2)\\
    =&\ a(\bs{u}_1-\tilde{\bs{u}}_1,\bs{u}_2-\tilde{\bs{u}}_2;\bs{v}_1-\bs{v}_{h,1},\bs{v}_2-\bs{v}_{h,2})-b(\bs{u}_1-\tilde{\bs{u}}_1,\bs{u}_2-\tilde{\bs{u}}_2;q_1-q_{h,1},q_2-q_{h,2})\\
    &-b(\bs{v}_1-\bs{v}_{h,1},\bs{v}_2-\bs{v}_{h,2};p_1-\tilde{p}_1,p_2-\tilde{p}_2)&\hspace*{-8em}\text{(using \eqref{proj1}-\eqref{proj2} and \eqref{projadj2})}\\
    \le&\ C\left(\sum_{i=1}^2\left[\|\bs{u}_i-\tilde{\bs{u}}_i\|_{1,\Omega_i^t}+\|p_i-\tilde{p}_i\|_{0,\Omega_i^t}\right]\right)\left(\sum_{i=1}^2\left[\|\bs{v}_i-\bs{v}_{h,i}\|_{1,\Omega_i^t}+\|q_i-q_{h,i}\|_{0,\Omega_i^t}\right]\right).
\end{alignat*}
Then, we have
\begin{equation}
    \label{eq:semi1-errest2}
    \begin{aligned}
        &\ \|(\bs{u}_1-\tilde{\bs{u}}_1,\bs{u}_2-\tilde{\bs{u}}_2)\|_0=\sup_{(\tilde{\bs{f}}_1,\tilde{\bs{f}}_2)\in(L^2(\Omega_1^t))^d\times(L^2(\Omega_2^t))^d}\frac{\displaystyle\sum_{i=1}^2(\tilde{\bs{f}}_i,\bs{u}_i-\tilde{\bs{u}}_i)_{\Omega_i^t}}{\|(\tilde{\bs{f}}_1,\tilde{\bs{f}}_2)\|_0}\\
        \le&\ \sup_{(\tilde{\bs{f}}_1,\tilde{\bs{f}}_2)\in(L^2(\Omega_1^t))^d\times(L^2(\Omega_2^t))^d}\frac{\displaystyle C\left(\sum_{i=1}^2\left[\|\bs{u}_i-\tilde{\bs{u}}_i\|_{1,\Omega_i^t}+\|p_i-\tilde{p}_i\|_{0,\Omega_i^t}\right]\right)\left(\sum_{i=1}^2\left[\|\bs{v}_i-\bs{v}_{h,i}\|_{1,\Omega_i^t}+\|q_i-q_{h,i}\|_{0,\Omega_i^t}\right]\right)}{\displaystyle\sum_{i=1}^2\|\tilde{\bs{f}}_i\|_{0,\Omega_i^t}}\\
        \le&\ \sup_{(\tilde{\bs{f}}_1,\tilde{\bs{f}}_2)\in(L^2(\Omega_1^t))^d\times(L^2(\Omega_2^t))^d}\frac{\displaystyle Ch^{k+1}\left(\sum_{i=1}^2\left[\|\bs{u}_i\|_{k+1,\Omega_i^t}+\|p_i\|_{k,\Omega_i^t}\right]\right)\left(\sum_{i=1}^2\left[\|\bs{v}_i\|_{2,\Omega_i^t}+\|q_i\|_{1,\Omega_i^t}\right]\right)}{\displaystyle\sum_{i=1}^2\|\tilde{\bs{f}}_i\|_{0,\Omega_i^t}}\\
        \le&\ Ch^{k+1}\sum_{i=1}^2\left[\|\bs{u}_i\|_{k+1,\Omega_i^t}+\|p_i\|_{k,\Omega_i^t}\right],
    \end{aligned}
\end{equation}
where the convergence properties of Taylor-Hood elements and the
regularity assumption \eqref{regadj1} are applied,
\textcolor{black}{with optimal orders of the standard finite element
approximations to the adjoint velocity
$\textcolor{black}{\bs{v}}_i\in (H^2(\Omega_i^t))^d$ in $H^1$ norm
and to the adjoint pressure $q_i\in H^1(\Omega_i^t)$ in $L^2$ norm.}
Combining (\ref{eq:semi1-errest2}) with \eqref{eq:semi1-errest1}, we
obtain the desired error estimate in \eqref{eq:semi1}.
\end{proof}

In the following two lemmas, we further employ the Aubin-Nitsche
duality argument to establish optimal $H^{-1}$ error estimates for
the $H^1$-projection of pressure and the gradient of
$H^1$-projection of the velocity, which play a crucial role in
deriving optimal-order convergence properties for the ALE-time
derivative of the proposed $H^1$ projection.

\begin{lemma}
\label{lem:semi2} Let $((\bs{u}_1,\bs{u}_2),(p_1,p_2))$ be the solution to \eqref{weak1} satisfying the regularity properties \eqref{reg1}-\eqref{reg5}, and let
$((\tilde{\bs{u}}_1,\tilde{\bs{u}}_2),(\tilde{p}_1,\tilde{p}_2))$ be
the solution to \eqref{proj1}-\eqref{proj2}. Then, for any
$(\tilde{g}_1,\tilde{g}_2)\in H^1(\Omega_1^t)\times
H^1(\Omega_2^t)$, the following error estimate holds
\begin{equation}
    \label{eq:semi2}
    \sum_{i=1}^2(p_i-\tilde{p}_i,\tilde{g}_i)_{\Omega_i^t}\le Ch^{k+1}\left(\sum_{i=1}^2\left[\|\bs{u}_i\|_{k+1,\Omega_i^t}+\|p_i\|_{k,\Omega_i^t}\right]\right)\left(\sum_{i=1}^2\|\tilde{g}_i\|_{1,\Omega_i^t}\right).
\end{equation}
\end{lemma}
\begin{proof}
We introduce a new adjoint problem of \eqref{proj1}-\eqref{proj2} as follows: given $(\tilde{g}_1,\tilde{g}_2)\in H^1(\Omega_1^t)\times H^1(\Omega_2^t)$, find $(\bs{v}_1,\bs{v}_2)\in\bs{U}^t$ and $(q_1,q_2)\in Q^t$, such that
\begin{align}
    \label{projadj3}
    a^\star(\bs{v}_1,\bs{v}_2;\bs{\psi}_1,\bs{\psi}_2)-b(\bs{\psi}_1,\bs{\psi}_2;q_1,q_2)&=0,\qquad\forall\,(\bs{\psi}_1,\bs{\psi}_2)\in\bs{U}^t,\\
    \label{projadj4}
    b(\bs{v}_1,\bs{v}_2;\chi_1,\chi_2)&=\sum_{i=1}^2(\tilde{g}_i-\textcolor{black}{\bar{\tilde{g}}},\chi_i)_{\Omega_i^t},\qquad\forall\,(\chi_1,\chi_2)\in Q^t,
\end{align}
with
\textcolor{black}{$\displaystyle\bar{\tilde{g}}=\frac{1}{|\Omega^t|}\sum_{i=1}^2\int_{\Omega_i^t}\tilde{g}_i\,d\bs{x}_i$},
where the following compatibility condition is satisfied:
\begin{equation*}
    \sum_{i=1}^2\int_{\Omega_i^t}(\tilde{g}_i-\textcolor{black}{\bar{\tilde{g}}})\,d\bs{x}_i
    =\sum_{i=1}^2\int_{\Omega_i^t}(\nabla\cdot\bs{v}_i)\,d\bs{x}_i
    =\sum_{i=1}^2\int_{\partial\Omega_i^t\setminus\Gamma^t}\bs{v}_i\cdot\bs{n}_i\,d\bs{s}
    +\int_{\Gamma^t}(\bs{v}_1-\bs{v}_2)\cdot\bs{n}_1d\bs{s}=0.
\end{equation*}
The equivalent strong form of \eqref{projadj3}-\eqref{projadj4} can
be defined as: find
$(\bs{v}_1,\bs{v}_2)\in(H^2(\Omega_1^t))^d\times(H^2(\Omega_2^t))^d$
and $(q_1,q_2)\in H^1(\Omega_1^t)\times H^1(\Omega_2^t)$ such that
\begin{equation}\label{strong-adjoint1}
    \left\{\begin{alignedat}{2}
        -\nabla\cdot(\mu_1\nabla\bs{v}_1)+(\bs{w}_{h,1}\cdot\nabla)\bs{v}_1+(\nabla\cdot\bs{w}_{h,1})\bs{v}_1+\kappa\bs{v}_1+\nabla q_1&=\bs{0},&&\text{in $\Omega_1^t$},\\
        \nabla\cdot\bs{v}_1&=\tilde{g}_1-\textcolor{black}{\bar{\tilde{g}}},\qquad&&\text{in $\Omega_1^t$},\\
        \bs{v}_1&=\bs{0},&&\text{on $\partial\Omega_1^t\setminus\Gamma^t$},\\
        -\nabla\cdot(\mu_2\nabla\bs{v}_2)+(\bs{w}_{h,2}\cdot\nabla)\bs{v}_2+(\nabla\cdot\bs{w}_{h,2})\bs{v}_2+\kappa\bs{v}_2+\nabla q_2&=\bs{0},&&\text{in $\Omega_2^t$},\\
        \nabla\cdot\bs{v}_2&=\tilde{g}_2-\textcolor{black}{\bar{\tilde{g}}},\qquad&&\text{in $\Omega_2^t$},\\
        \bs{v}_2&=\bs{0},&&\text{on $\partial\Omega_2^t\setminus\Gamma^t$},\\
        \bs{v}_1&=\bs{v}_2,&&\text{on $\Gamma^t$},\\
        (-q_1\bs{I}+\mu_1\nabla\bs{v}_1)\bs{n}_1+(-q_2\bs{I}+\mu_2\nabla\bs{v}_2)\bs{n}_2&=\bs{0},&&\text{on
        $\Gamma^t$},\\
\textcolor{black}{\int_{\Omega_1^t}q_1d\bs{x}_1+\int_{\Omega_2^t}q_2d\bs{x}_2}&
\textcolor{black}{=0,}&&
    \end{alignedat}\right.
\end{equation}
\textcolor{black}{which is also a steady-state Stokes problem with
transport and reaction terms as well as a moving 
interface. Similarly, we assume the following regularity estimate
guided by regularity assumptions as indicated in relevant studies
\cite{shibata2003resolvent,wang2015new,naraparaju2025higher}:}
\begin{equation}
    \label{regadj2}
    \sum_{i=1}^2\left[\|\bs{v}_i\|_{2,\Omega_i^t}+\|q_i\|_{1,\Omega_i^t}\right]\le C\sum_{i=1}^2\|\tilde{g}_i-\textcolor{black}{\bar{\tilde{g}}}\|_{1,\Omega_i^t}\le C\sum_{i=1}^2\|\tilde{g}_i\|_{1,\Omega_i^t}.
\end{equation}

Let $\bs{\psi}_i=\bs{u}_i-\tilde{\bs{u}}_i$ and
$\chi_i=p_i-\tilde{p}_i$ in \eqref{projadj3}-\eqref{projadj4}, take
any $(\bs{v}_{h,1},\bs{v}_{h,2})\in\bs{U}_h^t$ and
$(q_{h,1},q_{h,2})\in Q_h^t$, and apply the $H^1$-projection
\eqref{proj1}-\eqref{proj2} together with the adjoint property
\eqref{propadj}, the regularity assumption \eqref{regadj2}, Lemma
\ref{lem:semi1} and standard interpolation error estimates, which
yields
\begin{alignat*}{2}
    &\ \sum_{i=1}^2(p_i-\tilde{p}_i,\tilde{g}_i)_{\Omega_i^t}=\sum_{i=1}^2(p_i-\tilde{p}_i,\tilde{g}_i-\textcolor{black}{\bar{\tilde{g}}})_{\Omega_i^t}=b(\bs{v}_1,\bs{v}_2;p_1-\tilde{p}_1,p_2-\tilde{p}_2)&\text{(using \eqref{projadj4})}\\
    =&\ b(\bs{v}_1-\bs{v}_{h,1},\bs{v}_2-\bs{v}_{h,2};p_1-\tilde{p}_1,p_2-\tilde{p}_2)+b(\bs{v}_{h,1},\bs{v}_{h,2};p_1-\tilde{p}_1,p_2-\tilde{p}_2)\\
    =&\ b(\bs{v}_1-\bs{v}_{h,1},\bs{v}_2-\bs{v}_{h,2};p_1-\tilde{p}_1,p_2-\tilde{p}_2)+a(\bs{u}_1-\tilde{\bs{u}}_1,\bs{u}_2-\tilde{\bs{u}}_2;\bs{v}_{h,1},\bs{v}_{h,2})&\text{(using \eqref{proj1})}\\
    =&\ b(\bs{v}_1-\bs{v}_{h,1},\bs{v}_2-\bs{v}_{h,2};p_1-\tilde{p}_1,p_2-\tilde{p}_2)+a(\bs{u}_1-\tilde{\bs{u}}_1,\bs{u}_2-\tilde{\bs{u}}_2;\bs{v}_{h,1}-\bs{v}_1,\bs{v}_{h,2}-\bs{v}_2)\\
    &+a^\star(\bs{v}_1,\bs{v}_2;\bs{u}_1-\tilde{\bs{u}}_1,\bs{u}_2-\tilde{\bs{u}}_2)&\text{(using \eqref{propadj})}\\
    =&\ b(\bs{v}_1-\bs{v}_{h,1},\bs{v}_2-\bs{v}_{h,2};p_1-\tilde{p}_1,p_2-\tilde{p}_2)+a(\bs{u}_1-\tilde{\bs{u}}_1,\bs{u}_2-\tilde{\bs{u}}_2;\bs{v}_{h,1}-\bs{v}_1,\bs{v}_{h,2}-\bs{v}_2)\\
    &+b(\bs{u}_1-\tilde{\bs{u}}_1,\bs{u}_2-\tilde{\bs{u}}_2;q_1-q_{h,1},q_2-q_{h,2})&\text{(using \eqref{projadj3} and \eqref{proj2})}\\
    \le&\ C\left(\sum_{i=1}^2\left[\|\bs{u}_i-\tilde{\bs{u}}_i\|_{1,\Omega_i^t}+\|p_i-\tilde{p}_i\|_{0,\Omega_i^t}\right]\right)\left(\sum_{i=1}^2\left[\|\bs{v}_i-\bs{v}_{h,i}\|_{1,\Omega_i^t}+\|q_i-q_{h,i}\|_{0,\Omega_i^t}\right]\right)\\
    \le&\ Ch^{k+1}\left(\sum_{i=1}^2\left[\|\bs{u}_i\|_{k+1,\Omega_i^t}+\|p_i\|_{k,\Omega_i^t}\right]\right)\left(\sum_{i=1}^2\left[\|\bs{v}_i\|_{2,\Omega_i^t}+\|q_i\|_{1,\Omega_i^t}\right]\right)\\
    \le&\ Ch^{k+1}\left(\sum_{i=1}^2\left[\|\bs{u}_i\|_{k+1,\Omega_i^t}+\|p_i\|_{k,\Omega_i^t}\right]\right)\left(\sum_{i=1}^2\|\tilde{g}_i\|_{1,\Omega_i^t}\right),&\text{(using \eqref{regadj2})}
\end{alignat*}
\textcolor{black}{where we apply the standard finite element
approximations to the adjoint velocity
$\textcolor{black}{\bs{v}}_i\in (H^2(\Omega_i^t))^d$ in $H^1$ norm
and to the adjoint pressure $q_i\in H^1(\Omega_i^t)$ in $L^2$ norm,
both hold the optimal convergence orders. The desired error estimate
\eqref{eq:semi2} is thus obtained.}
\end{proof}

\begin{lemma}
\label{lem:semi3} \textcolor{black}{Let
$((\bs{u}_1,\bs{u}_2),(p_1,p_2))$ be the solution to \eqref{weak1}
satisfying the regularity properties \eqref{reg1}-\eqref{reg5}, and
let
$((\tilde{\bs{u}}_1,\tilde{\bs{u}}_2),(\tilde{p}_1,\tilde{p}_2))$ be
the solution to \eqref{proj1}-\eqref{proj2}. For} any
$(\tilde{\bs{g}}_1,\tilde{\bs{g}}_2)\in(H^1(\Omega_1^t))^{d\times
d}\times(H^1(\Omega_2^t))^{d\times d}$, the following error estimate
holds
\begin{equation}
    \label{eq:semi3}
    \sum_{i=1}^2(\nabla(\bs{u}_i-\tilde{\bs{u}}_i),\tilde{\bs{g}}_i)_{\Omega_i^t}\le Ch^{k+1}\left(\sum_{i=1}^2\left[\|\bs{u}_i\|_{k+1,\Omega_i^t}+\|p_i\|_{k,\Omega_i^t}\right]\right)\left(\sum_{i=1}^2\|\tilde{\bs{g}}_i\|_{1,\Omega_i^t}\right).
\end{equation}
\end{lemma}
\begin{proof}
We further introduce another adjoint problem of
\eqref{proj1}-\eqref{proj2} as follows: given
$(\tilde{\bs{g}}_1,\tilde{\bs{g}}_2)\in(H^1(\Omega_1^t))^{d\times
d}\times(H^1(\Omega_2^t))^{d\times d}$, find
$(\bs{v}_1,\bs{v}_2)\in\bs{U}^t$ and $(q_1,q_2)\in Q^t$ such that
\begin{align}
    \label{projadj5}
    a^\star(\bs{v}_1,\bs{v}_2;\bs{\psi}_1,\bs{\psi}_2)-b(\bs{\psi}_1,\bs{\psi}_2;q_1,q_2)&=\sum_{i=1}^2(\tilde{\bs{g}}_i,\nabla\bs{\psi}_i)_{\Omega_i^t},\qquad\forall\,(\bs{\psi}_1,\bs{\psi}_2)\in\bs{U}^t,\\
    \label{projadj6}
    b(\bs{v}_1,\bs{v}_2;\chi_1,\chi_2)&=0,\qquad\forall\,(\chi_1,\chi_2)\in Q^t.
\end{align}
The equivalent strong form of \eqref{projadj5}-\eqref{projadj6} can
be defined as: find
$(\bs{v}_1,\bs{v}_2)\in(H^2(\Omega_1^t))^d\times(H^2(\Omega_2^t))^d$
and $(q_1,q_2)\in H^1(\Omega_1^t)\times H^1(\Omega_2^t)$ such that
\begin{equation}\label{strong-adjoint2}
    \left\{\begin{alignedat}{2}
        -\nabla\cdot(\mu_1\nabla\bs{v}_1)+(\bs{w}_{h,1}\cdot\nabla)\bs{v}_1+(\nabla\cdot\bs{w}_{h,1})\bs{v}_1+\kappa\bs{v}_1+\nabla q_1&=-\nabla\cdot\tilde{\bs{g}}_1,&&\text{in $\Omega_1^t$},\\
        \nabla\cdot\bs{v}_1&=0,&&\text{in $\Omega_1^t$},\\
        \bs{v}_1&=\bs{0},&&\text{on $\partial\Omega_1^t\setminus\Gamma^t$},\\
        -\nabla\cdot(\mu_2\nabla\bs{v}_2)+(\bs{w}_{h,2}\cdot\nabla)\bs{v}_2+(\nabla\cdot\bs{w}_{h,2})\bs{v}_2+\kappa\bs{v}_2+\nabla q_2&=-\nabla\cdot\tilde{\bs{g}}_2,&&\text{in $\Omega_2^t$},\\
        \nabla\cdot\bs{v}_2&=0,&&\text{in $\Omega_2^t$},\\
        \bs{v}_2&=\bs{0},&&\text{on $\partial\Omega_2^t\setminus\Gamma^t$},\\
        \bs{v}_1&=\bs{v}_2,&&\text{on $\Gamma^t$},\\
        (-q_1\bs{I}+\mu_1\nabla\bs{v}_1)\bs{n}_1+(-q_2\bs{I}+\mu_2\nabla\bs{v}_2)\bs{n}_2&=\tilde{\bs{g}}_1\bs{n}_1+\tilde{\bs{g}}_2\bs{n}_2,\qquad&&\text{on
        $\Gamma^t$},\\
\textcolor{black}{\int_{\Omega_1^t}q_1d\bs{x}_1+\int_{\Omega_2^t}q_2d\bs{x}_2}&
\textcolor{black}{=0,}&&
    \end{alignedat}\right.
\end{equation}
\textcolor{black}{which again is likewise a steady-state Stokes
problem with transport and reaction terms, and a moving 
interface as well. The following regularity estimate, which is
guided by regularity assumptions adopted in related studies
\cite{shibata2003resolvent,wang2015new,naraparaju2025higher}, can be
assumed analogously:}
\begin{equation}
    \label{regadj3}
    \sum_{i=1}^2\left[\|\bs{v}_i\|_{2,\Omega_i^t}+\|q_i\|_{1,\Omega_i^t}\right]\le C\sum_{i=1}^2\left[\|\nabla\cdot\tilde{\bs{g}}_i\|_{0,\Omega_i^t}+\|\tilde{\bs{g}}_i\|_{\frac{1}{2},\Gamma^t}\right]\le C\sum_{i=1}^2\|\tilde{\bs{g}}_i\|_{1,\Omega_i^t}.
\end{equation}

Let $\bs{\psi}_i=\bs{u}_i-\tilde{\bs{u}}_i$ and
$\chi_i=p_i-\tilde{p}_i$ in \eqref{projadj5}-\eqref{projadj6}, take
any $(\bs{v}_{h,1},\bs{v}_{h,2})\in\bs{U}_h^t$ and
$(q_{h,1},q_{h,2})\in Q_h^t$, and apply the $H^1$-projection
\eqref{proj1}-\eqref{proj2} together with the adjoint property
\eqref{propadj}, the regularity assumption \eqref{regadj3}, Lemma
\ref{lem:semi1} and standard interpolation error estimates, which
yields
\begin{alignat*}{2}
    &\ \sum_{i=1}^2(\nabla(\bs{u}_i-\tilde{\bs{u}}_i),\tilde{\bs{g}}_i)_{\Omega_i^t}=a^\star(\bs{v}_1,\bs{v}_2;\bs{u}_1-\tilde{\bs{u}}_1,\bs{u}_2-\tilde{\bs{u}}_2)-b(\bs{u}_1-\tilde{\bs{u}}_1,\bs{u}_2-\tilde{\bs{u}}_2;q_1,q_2)&\text{(using \eqref{projadj5})}\\
    =&\ a^\star(\bs{v}_1-\bs{v}_{h,1},\bs{v}_2-\bs{v}_{h,2};\bs{u}_1-\tilde{\bs{u}}_1,\bs{u}_2-\tilde{\bs{u}}_2)-b(\bs{u}_1-\tilde{\bs{u}}_1,\bs{u}_2-\tilde{\bs{u}}_2;q_1-q_{h,1},q_2-q_{h,2})\\
    &+a^\star(\bs{v}_{h,1},\bs{v}_{h,2};\bs{u}_1-\tilde{\bs{u}}_1,\bs{u}_2-\tilde{\bs{u}}_2)-b(\bs{u}_1-\tilde{\bs{u}}_1,\bs{u}_2-\tilde{\bs{u}}_2;q_{h,1},q_{h,2})\\
    =&\ a(\bs{u}_1-\tilde{\bs{u}}_1,\bs{u}_2-\tilde{\bs{u}}_2;\bs{v}_1-\bs{v}_{h,1},\bs{v}_2-\bs{v}_{h,2})-b(\bs{u}_1-\tilde{\bs{u}}_1,\bs{u}_2-\tilde{\bs{u}}_2;q_1-q_{h,1},q_2-q_{h,2})&\text{(using \eqref{propadj})}\\
    &+a(\bs{u}_1-\tilde{\bs{u}}_1,\bs{u}_2-\tilde{\bs{u}}_2;\bs{v}_{h,1},\bs{v}_{h,2})-b(\bs{u}_1-\tilde{\bs{u}}_1,\bs{u}_2-\tilde{\bs{u}}_2;q_{h,1},q_{h,2})\\
    &-b(\bs{v}_{h,1},\bs{v}_{h,2};p_1-\tilde{p}_1,p_2-\tilde{p}_2)+b(\bs{v}_1,\bs{v}_2;p_1-\tilde{p}_1,p_2-\tilde{p}_2)\\
    &-b(\bs{v}_1-\bs{v}_{h,1},\bs{v}_2-\bs{v}_{h,2};p_1-\tilde{p}_1,p_2-\tilde{p}_2)\\
    =&\ a(\bs{u}_1-\tilde{\bs{u}}_1,\bs{u}_2-\tilde{\bs{u}}_2;\bs{v}_1-\bs{v}_{h,1},\bs{v}_2-\bs{v}_{h,2})-b(\bs{u}_1-\tilde{\bs{u}}_1,\bs{u}_2-\tilde{\bs{u}}_2;q_1-q_{h,1},q_2-q_{h,2})\\
    &-b(\bs{v}_1-\bs{v}_{h,1},\bs{v}_2-\bs{v}_{h,2};p_1-\tilde{p}_1,p_2-\tilde{p}_2)&\hspace*{-3em}\text{(using \eqref{proj1}-\eqref{proj2} and \eqref{projadj6})}\\
    \le&\ C\left(\sum_{i=1}^2\left[\|\bs{u}_i-\tilde{\bs{u}}_i\|_{1,\Omega_i^t}+\|p_i-\tilde{p}_i\|_{0,\Omega_i^t}\right]\right)\left(\sum_{i=1}^2\left[\|\bs{v}_i-\bs{v}_{h,i}\|_{1,\Omega_i^t}+\|q_i-q_{h,i}\|_{0,\Omega_i^t}\right]\right)\\
    \le&\ Ch^{k+1}\left(\sum_{i=1}^2\left[\|\bs{u}_i\|_{k+1,\Omega_i^t}+\|p_i\|_{k,\Omega_i^t}\right]\right)\left(\sum_{i=1}^2\left[\|\bs{v}_i\|_{2,\Omega_i^t}+\|q_i\|_{1,\Omega_i^t}\right]\right)\\
    \le&\ Ch^{k+1}\left(\sum_{i=1}^2\left[\|\bs{u}_i\|_{k+1,\Omega_i^t}+\|p_i\|_{k,\Omega_i^t}\right]\right)\left(\sum_{i=1}^2\|\tilde{\bs{g}}_i\|_{1,\Omega_i^t}\right),&\text{(using \eqref{regadj3})}
\end{alignat*}
\textcolor{black}{where standard finite element approximations to
the adjoint velocity $\textcolor{black}{\bs{v}}_i\in
(H^2(\Omega_i^t))^d$ in $H^1$ norm and to the adjoint pressure
$q_i\in H^1(\Omega_i^t)$ in $L^2$ norm are applied, held with
optimal convergence orders. Hence, the desired error estimate
\eqref{eq:semi3} is attained.}
\end{proof}

Upon the above results, we are finally in a position to prove the
following convergence result for the ALE-time derivative of the
$H^1$-projection.

\begin{lemma}
\label{lem:semi4} \textcolor{black}{Let
$((\bs{u}_1,\bs{u}_2),(p_1,p_2))$ be the solution to \eqref{weak1}
satisfying the regularity assumptions \eqref{reg1}-\eqref{reg5}, and
let
$((\tilde{\bs{u}}_1,\tilde{\bs{u}}_2),(\tilde{p}_1,\tilde{p}_2))$ be
the solution to \eqref{proj1}-\eqref{proj2}. We} have the following
error \textcolor{black}{estimate}
\begin{equation}
    \label{eq:semi4}
    \begin{aligned}
        &\ \sum_{i=1}^2\left\|\left.\frac{\partial\bs{u}_i}{\partial t}\right|_{\hat{\bs{x}}}^h-\left.\frac{\partial\tilde{\bs{u}}_i}{\partial t}\right|_{\hat{\bs{x}}}^h\right\|_{(L^2(\Omega_i^t))^d}+h\sum_{i=1}^2\left[\left\|\left.\frac{\partial\bs{u}_i}{\partial t}\right|_{\hat{\bs{x}}}^h-\left.\frac{\partial\tilde{\bs{u}}_i}{\partial t}\right|_{\hat{\bs{x}}}^h\right\|_{(H^1(\Omega_i^t))^d}+\left\|\left.\frac{\partial p_i}{\partial t}\right|_{\hat{\bs{x}}}^h-\left.\frac{\partial\tilde{p}_i}{\partial t}\right|_{\hat{\bs{x}}}^h\right\|_{L^2(\Omega_i^t)}\right]\\
        \le&\ Ch^{k+1}\sum_{i=1}^2\left[
        \textcolor{black}{\|\bs{w}_i\|_{(H^{k+1}(\Omega_i^t))^d}}
        +\|\bs{u}_i\|_{(H^{k+1}(\Omega_i^t))^d}
        +\left\|\left.\frac{\partial\bs{u}_i}{\partial t}\right|_{\hat{\bs{x}}}\right\|_{(H^{k+1}(\Omega_i^t))^d}\right.
        +\left.\|p_i\|_{H^k(\Omega_i^t)}
        +\left\|\left.\frac{\partial p_i}{\partial t}\right|_{\hat{\bs{x}}}\right\|_{H^k(\Omega_i^t)}\right].
    \end{aligned}
\end{equation}
\end{lemma}
\begin{proof}
Differentiate the $H^1$-projection equation \eqref{proj1}-\eqref{proj2} with respect to time $t$, and carefully apply the Reynolds transport theorem and Lemma \ref{lem:semi0} to each term, yield
\begin{align*}
    \frac{d}{dt}&(\mu_i\nabla(\bs{u}_i-\tilde{\bs{u}}_i),\nabla\bs{v}_{h,i})_{\Omega_i^t}=\left(\mu_i\nabla\left(\left.\frac{\partial\bs{u}_i}{\partial t}\right|_{\hat{\bs{x}}}^h-\left.\frac{\partial\tilde{\bs{u}}_i}{\partial t}\right|_{\hat{\bs{x}}}^h\right),\nabla\bs{v}_{h,i}\right)_{\Omega_i^t}\\
    &\quad+(\mu_i(\nabla\cdot\bs{w}_{h,i})\nabla(\bs{u}_i-\tilde{\bs{u}}_i),\nabla\bs{v}_{h,i})_{\Omega_i^t}-\left(\mu_i\nabla(\bs{u}_i-\tilde{\bs{u}}_i)\left(\nabla\bs{w}_{h,i}+\nabla\bs{w}_{h,i}^\mathrm{T}\right),\nabla\bs{v}_{h,i}\right)_{\Omega_i^t},\\
    -\frac{d}{dt}&((\bs{w}_{h,i}\cdot\nabla)(\bs{u}_i-\tilde{\bs{u}}_i),\bs{v}_{h,i})_{\Omega_i^t}=-\left((\bs{w}_{h,i}\cdot\nabla)\left(\left.\frac{\partial\bs{u}_i}{\partial t}\right|_{\hat{\bs{x}}}^h-\left.\frac{\partial\tilde{\bs{u}}_i}{\partial t}\right|_{\hat{\bs{x}}}^h\right),\bs{v}_{h,i}\right)_{\Omega_i^t}\\
    &\quad-\left(\left(\left.\frac{\partial\bs{w}_{h,i}}{\partial t}\right|_{\hat{\bs{x}}}^h\cdot\nabla\right)(\bs{u}_i-\tilde{\bs{u}}_i),\bs{v}_{h,i}\right)_{\Omega_i^t}-((\nabla\cdot\bs{w}_{h,i})(\bs{w}_{h,i}\cdot\nabla)(\bs{u}_i-\tilde{\bs{u}}_i),\bs{v}_{h,i})_{\Omega_i^t}\\
    &\quad+\left(\left(\bs{w}_{h,i}\cdot\left(\nabla\bs{w}_{h,i}^\mathrm{T}\nabla\right)\right)(\bs{u}_i-\tilde{\bs{u}}_i),\bs{v}_{h,i}\right)_{\Omega_i^t},\\
    \frac{d}{dt}&\kappa(\bs{u}_i-\tilde{\bs{u}}_i,\bs{v}_{h,i})_{\Omega_i^t}=\kappa\left(\left.\frac{\partial\bs{u}_i}{\partial t}\right|_{\hat{\bs{x}}}^h-\left.\frac{\partial\tilde{\bs{u}}_i}{\partial t}\right|_{\hat{\bs{x}}}^h,\bs{v}_{h,i}\right)_{\Omega_i^t}+\kappa((\nabla\cdot\bs{w}_{h,i})(\bs{u}_i-\tilde{\bs{u}}_i),\bs{v}_{h,i})_{\Omega_i^t},\\
    -\frac{d}{dt}&(p_i-\tilde{p}_i,\nabla\cdot\bs{v}_{h,i})_{\Omega_i^t}=-\left(\left.\frac{\partial p_i}{\partial t}\right|_{\hat{\bs{x}}}^h-\left.\frac{\partial\tilde{p}_i}{\partial t}\right|_{\hat{\bs{x}}}^h,\nabla\cdot\bs{v}_{h,i}\right)_{\Omega_i^t}-((\nabla\cdot\bs{w}_{h,i})(p_i-\tilde{p}_i),\nabla\cdot\bs{v}_{h,i})_{\Omega_i^t}\\
    &\quad+\left(\nabla\bs{w}_{h,i}(p_i-\tilde{p}_i),\nabla\bs{v}_{h,i}^\mathrm{T}\right)_{\Omega_i^t},\\
    \frac{d}{dt}&(\nabla\cdot(\bs{u}_i-\tilde{\bs{u}}_i),q_{h,i})_{\Omega_i^t}=\left(\nabla\cdot\left(\left.\frac{\partial\bs{u}_i}{\partial t}\right|_{\hat{\bs{x}}}^h-\left.\frac{\partial\tilde{\bs{u}}_i}{\partial t}\right|_{\hat{\bs{x}}}^h\right),q_{h,i}\right)_{\Omega_i^t}+(\nabla\cdot(\bs{u}_i-\tilde{\bs{u}}_i),(\nabla\cdot\bs{w}_{h,i})q_{h,i})_{\Omega_i^t}\\
    &\quad-\left(\nabla(\bs{u}_i-\tilde{\bs{u}}_i)^\mathrm{T},\nabla\bs{w}_{h,i}q_{h,i}\right)_{\Omega_i^t}.
\end{align*}
We thereby have
\begin{align}
    \label{eq:semi4-erreq1}
    &\begin{aligned}
        &a\left(\left.\frac{\partial\bs{u}_1}{\partial t}\right|_{\hat{\bs{x}}}^h-\left.\frac{\partial\tilde{\bs{u}}_1}{\partial t}\right|_{\hat{\bs{x}}}^h,\left.\frac{\partial\bs{u}_2}{\partial t}\right|_{\hat{\bs{x}}}^h-\left.\frac{\partial\tilde{\bs{u}}_2}{\partial t}\right|_{\hat{\bs{x}}}^h;\bs{v}_{h,1},\bs{v}_{h,2}\right)-b\left(\bs{v}_{h,1},\bs{v}_{h,2};\textcolor{black}{\left.\frac{\partial p_1}{\partial t}\right|_{\hat{\bs{x}}}^h-\left.\frac{\partial\tilde{p}_1}{\partial t}\right|_{\hat{\bs{x}}}^h,\left.\frac{\partial p_2}{\partial t}\right|_{\hat{\bs{x}}}^h-\left.\frac{\partial\tilde{p}_2}{\partial t}\right|_{\hat{\bs{x}}}^h}\right)\\
        &\qquad=\sum_{i=1}^2\Bigg[-(\mu_i(\nabla\cdot\bs{w}_{h,i})\nabla(\bs{u}_i-\tilde{\bs{u}}_i),\nabla\bs{v}_{h,i})_{\Omega_i^t}+\left(\mu_i\nabla(\bs{u}_i-\tilde{\bs{u}}_i)\left(\nabla\bs{w}_{h,i}+\nabla\bs{w}_{h,i}^\mathrm{T}\right),\nabla\bs{v}_{h,i}\right)_{\Omega_i^t}\\
        &\qquad\qquad+\left(\left(\left.\frac{\partial\bs{w}_{h,i}}{\partial t}\right|_{\hat{\bs{x}}}^h\cdot\nabla\right)(\bs{u}_i-\tilde{\bs{u}}_i),\bs{v}_{h,i}\right)_{\Omega_i^t}+((\nabla\cdot\bs{w}_{h,i})(\bs{w}_{h,i}\cdot\nabla)(\bs{u}_i-\tilde{\bs{u}}_i),\bs{v}_{h,i})_{\Omega_i^t}\\
        &\qquad\qquad-\left(\left(\bs{w}_{h,i}\cdot\left(\nabla\bs{w}_{h,i}^\mathrm{T}\nabla\right)\right)(\bs{u}_i-\tilde{\bs{u}}_i),\bs{v}_{h,i}\right)_{\Omega_i^t}-\kappa((\nabla\cdot\bs{w}_{h,i})(\bs{u}_i-\tilde{\bs{u}}_i),\bs{v}_{h,i})_{\Omega_i^t}\\
        &\qquad\qquad+((\nabla\cdot\bs{w}_{h,i})(p_i-\tilde{p}_i),\nabla\cdot\bs{v}_{h,i})_{\Omega_i^t}-\left(\nabla\bs{w}_{h,i}(p_i-\tilde{p}_i),\nabla\bs{v}_{h,i}^\mathrm{T}\right)_{\Omega_i^t}\Bigg],
    \end{aligned}\\
    \label{eq:semi4-erreq2}
    &\begin{aligned}
        &b\left(\left.\frac{\partial\bs{u}_1}{\partial t}\right|_{\hat{\bs{x}}}^h-\left.\frac{\partial\tilde{\bs{u}}_1}{\partial t}\right|_{\hat{\bs{x}}}^h,\left.\frac{\partial\bs{u}_2}{\partial t}\right|_{\hat{\bs{x}}}^h-\left.\frac{\partial\tilde{\bs{u}}_2}{\partial t}\right|_{\hat{\bs{x}}}^h;q_{h,1},q_{h,2}\right)=\sum_{i=1}^2\Bigg[-(\nabla\cdot(\bs{u}_i-\tilde{\bs{u}}_i),(\nabla\cdot\bs{w}_{h,i})q_{h,i})_{\Omega_i^t}\\
        &\qquad\qquad+\left(\nabla(\bs{u}_i-\tilde{\bs{u}}_i)^\mathrm{T},\nabla\bs{w}_{h,i}q_{h,i}\right)_{\Omega_i^t}\Bigg].
    \end{aligned}
\end{align}
\textcolor{black}{Introducing the quantities
\begin{align*}
    \eta_i=\left.\frac{\partial\bs{u}_i}{\partial t}\right|_{\hat{\bs{x}}}^h-I_h\left.\frac{\partial\bs{u}_i}{\partial t}\right|_{\hat{\bs{x}}},&\qquad
    \xi_i=I_h\left.\frac{\partial\bs{u}_i}{\partial t}\right|_{\hat{\bs{x}}}-\left.\frac{\partial\tilde{\bs{u}}_i}{\partial t}\right|_{\hat{\bs{x}}}^h,\\
    \delta_i=\left.\frac{\partial p_i}{\partial t}\right|_{\hat{\bs{x}}}^h-I_h\left.\frac{\partial p_i}{\partial t}\right|_{\hat{\bs{x}}},&\qquad
    \phi_i=I_h\left.\frac{\partial p_i}{\partial t}\right|_{\hat{\bs{x}}}-\left.\frac{\partial\tilde{p}_i}{\partial t}\right|_{\hat{\bs{x}}}^h-c_h^t,
\end{align*}
where $I_h$ denotes the standard Taylor-Hood finite element interpolation operator applied to the exact ALE-time derivatives and $\displaystyle c_h^t=\frac{1}{|\Omega^t|}\sum_{i=1}^2\int_{\Omega_i^t}\left(I_h\left.\frac{\partial p_i}{\partial t}\right|_{\hat{\bs{x}}}-\left.\frac{\partial\tilde{p}_i}{\partial t}\right|_{\hat{\bs{x}}}^h\right)d\bs{x}$. By these definitions, $\displaystyle\sum_{i=1}^2\int_{\Omega_i^t}\phi_i\,d\bs{x}=0$, so that $(\phi_1,\phi_2)\in Q_h^t$, and
\begin{equation*}
    \left.\frac{\partial p_i}{\partial t}\right|_{\hat{\bs{x}}}^h-\left.\frac{\partial\tilde{p}_i}{\partial t}\right|_{\hat{\bs{x}}}^h=\delta_i+\phi_i+c_h^t,\qquad i=1,2.
\end{equation*}
The correction $c_h^t$ is the same in both subdomains, and its
pairing with any $(\bs{v}_{h,1},\bs{v}_{h,2})\in\bs{U}_h^t$ in the
bilinear form of $b(\cdot,\cdot;\cdot,\cdot)$ vanishes because
\begin{equation}\label{const_c-h}
b(\bs{v}_{h,1},\bs{v}_{h,2};c_h^t,c_h^t)=c_h^t\sum_{i=1}^2\int_{\Omega_i^t}
(\nabla\cdot\bs{v}_{h,i})\,d\bs{x}_i=c_h^t\sum_{i=1}^2\int_{\partial\Omega_i^t\setminus\Gamma^t}\bs{v}_{h,i}\cdot\bs{n}_i\,d\bs{s}
    +c_h^t\int_{\Gamma^t}(\bs{v}_{h,1}-\bs{v}_{h,2})\cdot\bs{n}_1d\bs{s}=0.
\end{equation}
Thus, the second term on the left hand side of
\eqref{eq:semi4-erreq1} can be written in terms of
$\delta_i+\phi_i$, only. We then} reformulate
\eqref{eq:semi4-erreq1}-\eqref{eq:semi4-erreq2} as follows:
\begin{align}
    \label{eq:semi4-erreq3}
    &\begin{aligned}
        &a(\eta_1+\xi_1,\eta_2+\xi_2;\bs{v}_{h,1},\bs{v}_{h,2})-b(\bs{v}_{h,1},\bs{v}_{h,2};\delta_1+\phi_1,\delta_2+\phi_2)\\
        &\qquad=\sum_{i=1}^2\Bigg[-(\mu_i(\nabla\cdot\bs{w}_{h,i})\nabla(\bs{u}_i-\tilde{\bs{u}}_i),\nabla\bs{v}_{h,i})_{\Omega_i^t}+\left(\mu_i\nabla(\bs{u}_i-\tilde{\bs{u}}_i)\left(\nabla\bs{w}_{h,i}+\nabla\bs{w}_{h,i}^\mathrm{T}\right),\nabla\bs{v}_{h,i}\right)_{\Omega_i^t}\\
        &\qquad\qquad+\left(\left(\left.\frac{\partial\bs{w}_{h,i}}{\partial t}\right|_{\hat{\bs{x}}}^h\cdot\nabla\right)(\bs{u}_i-\tilde{\bs{u}}_i),\bs{v}_{h,i}\right)_{\Omega_i^t}+((\nabla\cdot\bs{w}_{h,i})(\bs{w}_{h,i}\cdot\nabla)(\bs{u}_i-\tilde{\bs{u}}_i),\bs{v}_{h,i})_{\Omega_i^t}\\
        &\qquad\qquad-\left(\left(\bs{w}_{h,i}\cdot\left(\nabla\bs{w}_{h,i}^\mathrm{T}\nabla\right)\right)(\bs{u}_i-\tilde{\bs{u}}_i),\bs{v}_{h,i}\right)_{\Omega_i^t}-\kappa((\nabla\cdot\bs{w}_{h,i})(\bs{u}_i-\tilde{\bs{u}}_i),\bs{v}_{h,i})_{\Omega_i^t}\\
        &\qquad\qquad+((\nabla\cdot\bs{w}_{h,i})(p_i-\tilde{p}_i),\nabla\cdot\bs{v}_{h,i})_{\Omega_i^t}-\left(\nabla\bs{w}_{h,i}(p_i-\tilde{p}_i),\nabla\bs{v}_{h,i}^\mathrm{T}\right)_{\Omega_i^t}\Bigg],
    \end{aligned}\\
    \label{eq:semi4-erreq4}
    &b(\eta_1+\xi_1,\eta_2+\xi_2;q_{h,1},q_{h,2})=\sum_{i=1}^2\Bigg[-(\nabla\cdot(\bs{u}_i-\tilde{\bs{u}}_i),(\nabla\cdot\bs{w}_{h,i})q_{h,i})_{\Omega_i^t}+\left(\nabla(\bs{u}_i-\tilde{\bs{u}}_i)^\mathrm{T},\nabla\bs{w}_{h,i}q_{h,i}\right)_{\Omega_i^t}\Bigg].
\end{align}

Let $\bs{v}_{h,i}=\xi_i$ and $q_{h,i}=\phi_i$ in
\eqref{eq:semi4-erreq3}-\eqref{eq:semi4-erreq4}, yields
\begin{align*}
    &a(\xi_1,\xi_2;\xi_1,\xi_2)=-a(\eta_1,\eta_2;\xi_1,\xi_2)+b(\xi_1,\xi_2;\delta_1,\delta_2)-b(\eta_1,\eta_2;\phi_1,\phi_2)\\
    &\qquad+\sum_{i=1}^2\Bigg[-(\mu_i(\nabla\cdot\bs{w}_{h,i})\nabla(\bs{u}_i-\tilde{\bs{u}}_i),\nabla\xi_i)_{\Omega_i^t}+\left(\mu_i\nabla(\bs{u}_i-\tilde{\bs{u}}_i)\left(\nabla\bs{w}_{h,i}+\nabla\bs{w}_{h,i}^\mathrm{T}\right),\nabla\xi_i\right)_{\Omega_i^t}\\
    &\qquad\qquad+\left(\left(\left.\frac{\partial\bs{w}_{h,i}}{\partial t}\right|_{\hat{\bs{x}}}^h\cdot\nabla\right)(\bs{u}_i-\tilde{\bs{u}}_i),\xi_i\right)_{\Omega_i^t}+((\nabla\cdot\bs{w}_{h,i})(\bs{w}_{h,i}\cdot\nabla)(\bs{u}_i-\tilde{\bs{u}}_i),\xi_i)_{\Omega_i^t}\\
    &\qquad\qquad-\left(\left(\bs{w}_{h,i}\cdot\left(\nabla\bs{w}_{h,i}^\mathrm{T}\nabla\right)\right)(\bs{u}_i-\tilde{\bs{u}}_i),\xi_i\right)_{\Omega_i^t}-\kappa((\nabla\cdot\bs{w}_{h,i})(\bs{u}_i-\tilde{\bs{u}}_i),\xi_i)_{\Omega_i^t}\\
    &\qquad\qquad+((\nabla\cdot\bs{w}_{h,i})(p_i-\tilde{p}_i),\nabla\cdot\xi_i)_{\Omega_i^t}-\left(\nabla\bs{w}_{h,i}(p_i-\tilde{p}_i),\nabla\xi_i^\mathrm{T}\right)_{\Omega_i^t}\\
    &\qquad\qquad-(\nabla\cdot(\bs{u}_i-\tilde{\bs{u}}_i),(\nabla\cdot\bs{w}_{h,i})\phi_i)_{\Omega_i^t}+\left(\nabla(\bs{u}_i-\tilde{\bs{u}}_i)^\mathrm{T},\nabla\bs{w}_{h,i}\phi_i\right)_{\Omega_i^t}\Bigg].
\end{align*}
Using \eqref{a-cont}-\eqref{a-coer} and the Cauchy-Schwarz inequality, we obtain
\begin{equation}
    \label{eq:semi4-errest1}
    \begin{aligned}
        \sum_{i=1}^2\|\xi_i\|_{1,\Omega_i^t}^2\le&\ C\sum_{i=1}^2\Bigg[\|\eta_i\|_{1,\Omega_i^t}\|\xi_i\|_{1,\Omega_i^t}+\|\delta_i\|_{0,\Omega_i^t}\|\xi_i\|_{1,\Omega_i^t}+\|\eta_i\|_{1,\Omega_i^t}\|\phi_i\|_{0,\Omega_i^t}\\
        &\qquad\quad+\|\bs{u}_i-\tilde{\bs{u}}_i\|_{1,\Omega_i^t}\|\xi_i\|_{1,\Omega_i^t}+\|p_i-\tilde{p}_i\|_{0,\Omega_i^t}\|\xi_i\|_{1,\Omega_i^t}\\
        &\qquad\quad+\|\bs{u}_i-\tilde{\bs{u}}_i\|_{1,\Omega_i^t}\|\phi_i\|_{0,\Omega_i^t}\Bigg].
    \end{aligned}
\end{equation}
We first proceed to estimate $\|\eta_i\|_{1,\Omega_i^t}$. Using the
triangle inequality, the Cauchy-Schwarz inequality and the identity
$\displaystyle\left.\frac{\partial\bs{u}_i}{\partial
t}\right|_{\hat{\bs{x}}}^h=\left.\frac{\partial\bs{u}_i}{\partial
t}\right|_{\hat{\bs{x}}}+((\bs{w}_{h,i}-\bs{w}_i)\cdot\nabla)\bs{u}_i$,
we obtain \textcolor{black}{\begin{equation}
    \label{eq:semi4-errest2}
    \begin{aligned}
        &\ \|\eta_i\|_{1,\Omega_i^t}=\left\|\left.\frac{\partial\bs{u}_i}{\partial t}\right|_{\hat{\bs{x}}}^h-I_h\left.\frac{\partial\bs{u}_i}{\partial t}\right|_{\hat{\bs{x}}}\right\|_{1,\Omega_i^t}\\
        \le&\ \left\|\left.\frac{\partial\bs{u}_i}{\partial t}\right|_{\hat{\bs{x}}}^h-\left.\frac{\partial\bs{u}_i}{\partial t}\right|_{\hat{\bs{x}}}\right\|_{1,\Omega_i^t}+\left\|\left.\frac{\partial\bs{u}_i}{\partial t}\right|_{\hat{\bs{x}}}-I_h\left.\frac{\partial\bs{u}_i}{\partial t}\right|_{\hat{\bs{x}}}\right\|_{1,\Omega_i^t}\\
        \le&\ C\|\bs{w}_i-\bs{w}_{h,i}\|_{(H^1(\Omega_i^t))^d}\|\bs{u}_i\|_{(W^{2,\infty}(\Omega_i^t))^d}+Ch^k\left\|\left.\frac{\partial\bs{u}_i}{\partial t}\right|_{\hat{\bs{x}}}\right\|_{(H^{k+1}(\Omega_i^t))^d}\\
        \le&\ Ch^k\left(\textcolor{black}{\|\bs{w}_i\|_{(H^{k+1}(\Omega_i^t))^d}}\|\bs{u}_i\|_{(W^{2,\infty}(\Omega_i^t))^d}+\left\|\left.\frac{\partial\bs{u}_i}{\partial t}\right|_{\hat{\bs{x}}}\right\|_{(H^{k+1}(\Omega_i^t))^d}\right).
    \end{aligned}
\end{equation}}
We next consider the estimate of $\|\delta_i\|_{0,\Omega_i^t}$.
Similarly, using the triangle inequality, the Cauchy-Schwarz
inequality and the identity $\displaystyle\left.\frac{\partial
p_i}{\partial t}\right|_{\hat{\bs{x}}}^h=\left.\frac{\partial
p_i}{\partial
t}\right|_{\hat{\bs{x}}}+((\bs{w}_{h,i}-\bs{w}_i)\cdot\nabla)p_i$,
we obtain \textcolor{black}{\begin{equation}
    \label{eq:semi4-errest3}
    \begin{aligned}
        &\ \|\delta_i\|_{0,\Omega_i^t}=\textcolor{black}{\left\|\left.\frac{\partial p_i}{\partial t}\right|_{\hat{\bs{x}}}^h-I_h\left.\frac{\partial p_i}{\partial t}\right|_{\hat{\bs{x}}}\right\|_{0,\Omega_i^t}}\\
        \le&\ \left\|\left.\frac{\partial p_i}{\partial t}\right|_{\hat{\bs{x}}}^h-\left.\frac{\partial p_i}{\partial t}\right|_{\hat{\bs{x}}}\right\|_{0,\Omega_i^t}+\left\|\left.\frac{\partial p_i}{\partial t}\right|_{\hat{\bs{x}}}-I_h\left.\frac{\partial p_i}{\partial t}\right|_{\hat{\bs{x}}}\right\|_{0,\Omega_i^t}\\
        \le&\ C\|\bs{w}_i-\bs{w}_{h,i}\|_{(L^2(\Omega_i^t))^d}\|p_i\|_{W^{1,\infty}(\Omega_i^t)}+Ch^k\left\|\left.\frac{\partial p_i}{\partial t}\right|_{\hat{\bs{x}}}\right\|_{H^k(\Omega_i^t)}\\
        \le&\ Ch^k\left(\textcolor{black}{\|\bs{w}_i\|_{(H^{k+1}(\Omega_i^t))^d}}\|p_i\|_{W^{1,\infty}(\Omega_i^t)}+\left\|\left.\frac{\partial p_i}{\partial t}\right|_{\hat{\bs{x}}}\right\|_{H^k(\Omega_i^t)}\right).
    \end{aligned}
\end{equation}}
\textcolor{black}{We then estimate the correction $c_h^t$. Since
$\displaystyle\sum_{i=1}^2\int_{\Omega_i^t}(p_i-\tilde{p}_i)\,d\bs{x}_i=0$,
an application of the Reynolds transport theorem yields
\begin{equation}
    \label{eq:semi4-errest4}
    \begin{aligned}
        |\Omega^t|^\frac{1}{2}|c_h^t|&=\frac{1}{|\Omega^t|^\frac{1}{2}}\left|\sum_{i=1}^2\int_{\Omega_i^t}\left(I_h\left.\frac{\partial p_i}{\partial t}\right|_{\hat{\bs{x}}}-\left.\frac{\partial\tilde{p}_i}{\partial t}\right|_{\hat{\bs{x}}}^h\right)d\bs{x}\right|\\
        &\le\left(\sum_{i=1}^2\|\delta_i\|_{0,\Omega_i^t}^2\right)^\frac{1}{2}+\frac{1}{|\Omega^t|^\frac{1}{2}}\left|\sum_{i=1}^2\int_{\Omega_i^t}\left(\left.\frac{\partial p_i}{\partial t}\right|_{\hat{\bs{x}}}^h-\left.\frac{\partial\tilde{p}_i}{\partial t}\right|_{\hat{\bs{x}}}^h\right)d\bs{x}\right|\\
        &=\left(\sum_{i=1}^2\|\delta_i\|_{0,\Omega_i^t}^2\right)^\frac{1}{2}+\frac{1}{|\Omega^t|^\frac{1}{2}}\left|\sum_{i=1}^2\int_{\Omega_i^t}(\nabla\cdot\bs{w}_{h,i})(\tilde{p}_i-p_i)\,d\bs{x}\right|\\
        &\le\left(\sum_{i=1}^2\|\delta_i\|_{0,\Omega_i^t}^2\right)^\frac{1}{2}+C\sum_{i=1}^2\|p_i-\tilde{p}_i\|_{0,\Omega_i^t}.
    \end{aligned}
\end{equation}}

To estimate $\|\phi_i\|_{0,\Omega_i^t}$, we use the discrete inf-sup condition and \eqref{eq:semi4-erreq3}, resulting in
\begin{equation}
    \label{eq:semi4-errest5}
    \begin{aligned}
        &\ \gamma\|(\phi_1,\phi_2)\|_0\le\sup_{(\bs{v}_{h,1},\bs{v}_{h,2})\in\bs{U}_h^t}\frac{b(\bs{v}_{h,1},\bs{v}_{h,2};\phi_1,\phi_2)}{\|(\bs{v}_{h,1},\bs{v}_{h,2})\|_1}\\
        =&\ \sup_{(\bs{v}_{h,1},\bs{v}_{h,2})\in\bs{U}_h^t}\frac{1}{\|(\bs{v}_{h,1},\bs{v}_{h,2})\|_1}\Bigg\{a(\xi_1,\xi_2;\bs{v}_{h,1},\bs{v}_{h,2})+a(\eta_1,\eta_2;\bs{v}_{h,1},\bs{v}_{h,2})-b(\bs{v}_{h,1},\bs{v}_{h,2};\delta_1,\delta_2)\\
        &\qquad-\sum_{i=1}^2\Bigg[-(\mu_i(\nabla\cdot\bs{w}_{h,i})\nabla(\bs{u}_i-\tilde{\bs{u}}_i),\nabla\bs{v}_{h,i})_{\Omega_i^t}+\left(\mu_i\nabla(\bs{u}_i-\tilde{\bs{u}}_i)\left(\nabla\bs{w}_{h,i}+\nabla\bs{w}_{h,i}^\mathrm{T}\right),\nabla\bs{v}_{h,i}\right)_{\Omega_i^t}\\
        &\qquad\qquad+\left(\left(\left.\frac{\partial\bs{w}_{h,i}}{\partial t}\right|_{\hat{\bs{x}}}^h\cdot\nabla\right)(\bs{u}_i-\tilde{\bs{u}}_i),\bs{v}_{h,i}\right)_{\Omega_i^t}+((\nabla\cdot\bs{w}_{h,i})(\bs{w}_{h,i}\cdot\nabla)(\bs{u}_i-\tilde{\bs{u}}_i),\bs{v}_{h,i})_{\Omega_i^t}\\
        &\qquad\qquad-\left(\left(\bs{w}_{h,i}\cdot\left(\nabla\bs{w}_{h,i}^\mathrm{T}\nabla\right)\right)(\bs{u}_i-\tilde{\bs{u}}_i),\bs{v}_{h,i}\right)_{\Omega_i^t}-\kappa((\nabla\cdot\bs{w}_{h,i})(\bs{u}_i-\tilde{\bs{u}}_i),\bs{v}_{h,i})_{\Omega_i^t}\\
        &\qquad\qquad+((\nabla\cdot\bs{w}_{h,i})(p_i-\tilde{p}_i),\nabla\cdot\bs{v}_{h,i})_{\Omega_i^t}-\left(\nabla\bs{w}_{h,i}(p_i-\tilde{p}_i),\nabla\bs{v}_{h,i}^\mathrm{T}\right)_{\Omega_i^t}\Bigg]\Bigg\}\\
        \le&\ C\sum_{i=1}^2\left[\|\xi_i\|_{1,\Omega_i^t}+\|\eta_i\|_{1,\Omega_i^t}+\|\delta_i\|_{0,\Omega_i^t}+\|\bs{u}_i-\tilde{\bs{u}}_i\|_{1,\Omega_i^t}+\|p_i-\tilde{p}_i\|_{0,\Omega_i^t}\right].
    \end{aligned}
\end{equation}
Combining \eqref{eq:semi4-errest1}-\eqref{eq:semi4-errest5}, and applying the $\varepsilon$-Young's inequality as well as Lemma \ref{lem:semi1}, we obtain
\begin{equation}
    \label{eq:semi4-errest6}
    \begin{aligned}
        \sum_{i=1}^2\|\xi_i\|_{(H^1(\Omega_i^t))^d}&\le Ch^k\sum_{i=1}^2\left[
        \textcolor{black}{\|\bs{w}_i\|_{(H^{k+1}(\Omega_i^t))^d}}
        +\|\bs{u}_i\|_{(H^{k+1}(\Omega_i^t))^d}
        +\left\|\left.\frac{\partial\bs{u}_i}{\partial t}\right|_{\hat{\bs{x}}}\right\|_{(H^{k+1}(\Omega_i^t))^d}\right.\\
        &\qquad\qquad\quad\ +\left.\|p_i\|_{H^k(\Omega_i^t)}
        +\left\|\left.\frac{\partial p_i}{\partial t}\right|_{\hat{\bs{x}}}\right\|_{H^k(\Omega_i^t)}\right].
    \end{aligned}
\end{equation}
Then, by combining the above estimates with \eqref{eq:semi4-errest5}, we arrive at
\begin{equation}
    \label{eq:semi4-errest7}
    \begin{aligned}
        \sum_{i=1}^2\|\phi_i\|_{L^2(\Omega_i^t)}&\le Ch^k\sum_{i=1}^2\left[
        \textcolor{black}{\|\bs{w}_i\|_{(H^{k+1}(\Omega_i^t))^d}}
        +\|\bs{u}_i\|_{(H^{k+1}(\Omega_i^t))^d}
        +\left\|\left.\frac{\partial\bs{u}_i}{\partial t}\right|_{\hat{\bs{x}}}\right\|_{(H^{k+1}(\Omega_i^t))^d}\right.\\
        &\qquad\qquad\quad\ +\left.\|p_i\|_{H^k(\Omega_i^t)}
        +\left\|\left.\frac{\partial p_i}{\partial t}\right|_{\hat{\bs{x}}}\right\|_{H^k(\Omega_i^t)}\right].
    \end{aligned}
\end{equation}
By using the triangle inequality, we conclude that
\begin{align*}
    &\ \sum_{i=1}^2\left[\left\|\left.\frac{\partial\bs{u}_i}{\partial t}\right|_{\hat{\bs{x}}}^h-\left.\frac{\partial\tilde{\bs{u}}_i}{\partial t}\right|_{\hat{\bs{x}}}^h\right\|_{(H^1(\Omega_i^t))^d}+\left\|\left.\frac{\partial p_i}{\partial t}\right|_{\hat{\bs{x}}}^h-\left.\frac{\partial\tilde{p}_i}{\partial t}\right|_{\hat{\bs{x}}}^h\right\|_{L^2(\Omega_i^t)}\right]\\
    \textcolor{black}{\le}&\ \sum_{i=1}^2\left[\|\eta_i+\xi_i\|_{(H^1(\Omega_i^t))^d}+\|\delta_i+\phi_i\|_{L^2(\Omega_i^t)}\right]\textcolor{black}{+C|\Omega^t|^\frac{1}{2}|c_h^t|}\\
    \le&\ Ch^k\sum_{i=1}^2\left[
    \textcolor{black}{\|\bs{w}_i\|_{(H^{k+1}(\Omega_i^t))^d}}
    +\|\bs{u}_i\|_{(H^{k+1}(\Omega_i^t))^d}
    +\left\|\left.\frac{\partial\bs{u}_i}{\partial t}\right|_{\hat{\bs{x}}}\right\|_{(H^{k+1}(\Omega_i^t))^d}\right.\\
    &\qquad\qquad\ +\left.\|p_i\|_{H^k(\Omega_i^t)}
    +\left\|\left.\frac{\partial p_i}{\partial t}\right|_{\hat{\bs{x}}}\right\|_{H^k(\Omega_i^t)}\right].
\end{align*}

Finally, we estimate
$\displaystyle\left\|\left.\frac{\partial\bs{u}_i}{\partial
t}\right|_{\hat{\bs{x}}}^h-\left.\frac{\partial\tilde{\bs{u}}_i}{\partial
t}\right|_{\hat{\bs{x}}}^h\right\|_{(L^2(\Omega_i^t))^d}$. Let
$\bs{\psi}_i=\left.\dfrac{\partial\bs{u}_i}{\partial
t}\right|_{\hat{\bs{x}}}^h-\left.\dfrac{\partial\tilde{\bs{u}}_i}{\partial
t}\right|_{\hat{\bs{x}}}^h$ and
\textcolor{black}{$\chi_i=\left.\dfrac{\partial p_i}{\partial
t}\right|_{\hat{\bs{x}}}^h-\left.\dfrac{\partial\tilde{p}_i}{\partial
t}\right|_{\hat{\bs{x}}}^h-d_h^t$ with $\displaystyle
d_h^t=\frac{1}{|\Omega^t|}\sum_{j=1}^2\int_{\Omega_j^t}\left(\left.\frac{\partial
p_j}{\partial
t}\right|_{\hat{\bs{x}}}^h-\left.\frac{\partial\tilde{p}_j}{\partial
t}\right|_{\hat{\bs{x}}}^h\right)d\bs{x}$} in
\eqref{projadj1}-\eqref{projadj2}, take any
$(\bs{v}_{h,1},\bs{v}_{h,2})\in\bs{U}_h^t$ and $(q_{h,1},q_{h,2})\in
Q_h^t$, and apply the adjoint property \eqref{propadj}, which yields
\begin{align}
    \notag
    &\ \sum_{i=1}^2\left(\tilde{\bs{f}}_i,\textcolor{black}{\left.\frac{\partial\bs{u}_i}{\partial t}\right|_{\hat{\bs{x}}}^h-\left.\frac{\partial\tilde{\bs{u}}_i}{\partial t}\right|_{\hat{\bs{x}}}^h}\right)_{\Omega_i^t}\\
    \notag
    =&\ a^\star\left(\bs{v}_1,\bs{v}_2;\left.\frac{\partial\bs{u}_1}{\partial t}\right|_{\hat{\bs{x}}}^h-\left.\frac{\partial\tilde{\bs{u}}_1}{\partial t}\right|_{\hat{\bs{x}}}^h,\left.\frac{\partial\bs{u}_2}{\partial t}\right|_{\hat{\bs{x}}}^h-\left.\frac{\partial\tilde{\bs{u}}_2}{\partial t}\right|_{\hat{\bs{x}}}^h\right)-b\left(\left.\frac{\partial\bs{u}_1}{\partial t}\right|_{\hat{\bs{x}}}^h-\left.\frac{\partial\tilde{\bs{u}}_1}{\partial t}\right|_{\hat{\bs{x}}}^h,\left.\frac{\partial\bs{u}_2}{\partial t}\right|_{\hat{\bs{x}}}^h-\left.\frac{\partial\tilde{\bs{u}}_2}{\partial t}\right|_{\hat{\bs{x}}}^h;q_1,q_2\right)\\
    \notag
    =&\ a^\star\left(\bs{v}_1-\bs{v}_{h,1},\bs{v}_2-\bs{v}_{h,2};\left.\frac{\partial\bs{u}_1}{\partial t}\right|_{\hat{\bs{x}}}^h-\left.\frac{\partial\tilde{\bs{u}}_1}{\partial t}\right|_{\hat{\bs{x}}}^h,\left.\frac{\partial\bs{u}_2}{\partial t}\right|_{\hat{\bs{x}}}^h-\left.\frac{\partial\tilde{\bs{u}}_2}{\partial t}\right|_{\hat{\bs{x}}}^h\right)\\
    \notag
    &-b\left(\left.\frac{\partial\bs{u}_1}{\partial t}\right|_{\hat{\bs{x}}}^h-\left.\frac{\partial\tilde{\bs{u}}_1}{\partial t}\right|_{\hat{\bs{x}}}^h,\left.\frac{\partial\bs{u}_2}{\partial t}\right|_{\hat{\bs{x}}}^h-\left.\frac{\partial\tilde{\bs{u}}_2}{\partial t}\right|_{\hat{\bs{x}}}^h;q_1-q_{h,1},q_2-q_{h,2}\right)\\
    \notag
    &+a^\star\left(\bs{v}_{h,1},\bs{v}_{h,2};\left.\frac{\partial\bs{u}_1}{\partial t}\right|_{\hat{\bs{x}}}^h-\left.\frac{\partial\tilde{\bs{u}}_1}{\partial t}\right|_{\hat{\bs{x}}}^h,\left.\frac{\partial\bs{u}_2}{\partial t}\right|_{\hat{\bs{x}}}^h-\left.\frac{\partial\tilde{\bs{u}}_2}{\partial t}\right|_{\hat{\bs{x}}}^h\right)-b\left(\left.\frac{\partial\bs{u}_1}{\partial t}\right|_{\hat{\bs{x}}}^h-\left.\frac{\partial\tilde{\bs{u}}_1}{\partial t}\right|_{\hat{\bs{x}}}^h,\left.\frac{\partial\bs{u}_2}{\partial t}\right|_{\hat{\bs{x}}}^h-\left.\frac{\partial\tilde{\bs{u}}_2}{\partial t}\right|_{\hat{\bs{x}}}^h;q_{h,1},q_{h,2}\right)\\
    \notag
    =&\ a\left(\left.\frac{\partial\bs{u}_1}{\partial t}\right|_{\hat{\bs{x}}}^h-\left.\frac{\partial\tilde{\bs{u}}_1}{\partial t}\right|_{\hat{\bs{x}}}^h,\left.\frac{\partial\bs{u}_2}{\partial t}\right|_{\hat{\bs{x}}}^h-\left.\frac{\partial\tilde{\bs{u}}_2}{\partial t}\right|_{\hat{\bs{x}}}^h;\bs{v}_1-\bs{v}_{h,1},\bs{v}_2-\bs{v}_{h,2}\right)\\
    \notag
    &-b\left(\left.\frac{\partial\bs{u}_1}{\partial t}\right|_{\hat{\bs{x}}}^h-\left.\frac{\partial\tilde{\bs{u}}_1}{\partial t}\right|_{\hat{\bs{x}}}^h,\left.\frac{\partial\bs{u}_2}{\partial t}\right|_{\hat{\bs{x}}}^h-\left.\frac{\partial\tilde{\bs{u}}_2}{\partial t}\right|_{\hat{\bs{x}}}^h;q_1-q_{h,1},q_2-q_{h,2}\right)\\
    \notag
    &+a\left(\left.\frac{\partial\bs{u}_1}{\partial t}\right|_{\hat{\bs{x}}}^h-\left.\frac{\partial\tilde{\bs{u}}_1}{\partial t}\right|_{\hat{\bs{x}}}^h,\left.\frac{\partial\bs{u}_2}{\partial t}\right|_{\hat{\bs{x}}}^h-\left.\frac{\partial\tilde{\bs{u}}_2}{\partial t}\right|_{\hat{\bs{x}}}^h;\bs{v}_{h,1},\bs{v}_{h,2}\right)-b\left(\left.\frac{\partial\bs{u}_1}{\partial t}\right|_{\hat{\bs{x}}}^h-\left.\frac{\partial\tilde{\bs{u}}_1}{\partial t}\right|_{\hat{\bs{x}}}^h,\left.\frac{\partial\bs{u}_2}{\partial t}\right|_{\hat{\bs{x}}}^h-\left.\frac{\partial\tilde{\bs{u}}_2}{\partial t}\right|_{\hat{\bs{x}}}^h;q_{h,1},q_{h,2}\right)\\
    \notag
    &-b\left(\bs{v}_{h,1},\bs{v}_{h,2};\textcolor{black}{\left.\frac{\partial p_1}{\partial t}\right|_{\hat{\bs{x}}}^h-\left.\frac{\partial\tilde{p}_1}{\partial t}\right|_{\hat{\bs{x}}}^h-d_h^t},\textcolor{black}{\left.\frac{\partial p_2}{\partial t}\right|_{\hat{\bs{x}}}^h-\left.\frac{\partial\tilde{p}_2}{\partial t}\right|_{\hat{\bs{x}}}^h-d_h^t}\right)\\
    \notag
    &+b\left(\bs{v}_1,\bs{v}_2;\textcolor{black}{\left.\frac{\partial p_1}{\partial t}\right|_{\hat{\bs{x}}}^h-\left.\frac{\partial\tilde{p}_1}{\partial t}\right|_{\hat{\bs{x}}}^h-d_h^t},\textcolor{black}{\left.\frac{\partial p_2}{\partial t}\right|_{\hat{\bs{x}}}^h-\left.\frac{\partial\tilde{p}_2}{\partial t}\right|_{\hat{\bs{x}}}^h-d_h^t}\right)\\
    \notag
    &-b\left(\bs{v}_1-\bs{v}_{h,1},\bs{v}_2-\bs{v}_{h,2};\textcolor{black}{\left.\frac{\partial p_1}{\partial t}\right|_{\hat{\bs{x}}}^h-\left.\frac{\partial\tilde{p}_1}{\partial t}\right|_{\hat{\bs{x}}}^h-d_h^t},\textcolor{black}{\left.\frac{\partial p_2}{\partial t}\right|_{\hat{\bs{x}}}^h-\left.\frac{\partial\tilde{p}_2}{\partial t}\right|_{\hat{\bs{x}}}^h-d_h^t}\right)\\
    \notag
    =&\ a\left(\left.\frac{\partial\bs{u}_1}{\partial t}\right|_{\hat{\bs{x}}}^h-\left.\frac{\partial\tilde{\bs{u}}_1}{\partial t}\right|_{\hat{\bs{x}}}^h,\left.\frac{\partial\bs{u}_2}{\partial t}\right|_{\hat{\bs{x}}}^h-\left.\frac{\partial\tilde{\bs{u}}_2}{\partial t}\right|_{\hat{\bs{x}}}^h;\bs{v}_1-\bs{v}_{h,1},\bs{v}_2-\bs{v}_{h,2}\right)\\
    \notag
    &-b\left(\left.\frac{\partial\bs{u}_1}{\partial t}\right|_{\hat{\bs{x}}}^h-\left.\frac{\partial\tilde{\bs{u}}_1}{\partial t}\right|_{\hat{\bs{x}}}^h,\left.\frac{\partial\bs{u}_2}{\partial t}\right|_{\hat{\bs{x}}}^h-\left.\frac{\partial\tilde{\bs{u}}_2}{\partial t}\right|_{\hat{\bs{x}}}^h;q_1-q_{h,1},q_2-q_{h,2}\right)\\
    \notag
    &-b\left(\bs{v}_1-\bs{v}_{h,1},\bs{v}_2-\bs{v}_{h,2};\textcolor{black}{\left.\frac{\partial p_1}{\partial t}\right|_{\hat{\bs{x}}}^h-\left.\frac{\partial\tilde{p}_1}{\partial t}\right|_{\hat{\bs{x}}}^h-d_h^t},\textcolor{black}{\left.\frac{\partial p_2}{\partial t}\right|_{\hat{\bs{x}}}^h-\left.\frac{\partial\tilde{p}_2}{\partial t}\right|_{\hat{\bs{x}}}^h-d_h^t}\right)\\
    \notag
    &+\sum_{i=1}^2\Bigg[-(\mu_i(\nabla\cdot\bs{w}_{h,i})\nabla(\bs{u}_i-\tilde{\bs{u}}_i),\nabla\bs{v}_{h,i})_{\Omega_i^t}+\left(\mu_i\nabla(\bs{u}_i-\tilde{\bs{u}}_i)\left(\nabla\bs{w}_{h,i}+\nabla\bs{w}_{h,i}^\mathrm{T}\right),\nabla\bs{v}_{h,i}\right)_{\Omega_i^t}\\
    \notag
    &\qquad+\left(\left(\left.\frac{\partial\bs{w}_{h,i}}{\partial t}\right|_{\hat{\bs{x}}}^h\cdot\nabla\right)(\bs{u}_i-\tilde{\bs{u}}_i),\bs{v}_{h,i}\right)_{\Omega_i^t}+((\nabla\cdot\bs{w}_{h,i})(\bs{w}_{h,i}\cdot\nabla)(\bs{u}_i-\tilde{\bs{u}}_i),\bs{v}_{h,i})_{\Omega_i^t}\\
    \notag
    &\qquad-\left(\left(\bs{w}_{h,i}\cdot\left(\nabla\bs{w}_{h,i}^\mathrm{T}\nabla\right)\right)(\bs{u}_i-\tilde{\bs{u}}_i),\bs{v}_{h,i}\right)_{\Omega_i^t}-\kappa((\nabla\cdot\bs{w}_{h,i})(\bs{u}_i-\tilde{\bs{u}}_i),\bs{v}_{h,i})_{\Omega_i^t}\\
    \notag
    &\qquad+((\nabla\cdot\bs{w}_{h,i})(p_i-\tilde{p}_i),\nabla\cdot\bs{v}_{h,i})_{\Omega_i^t}-\left(\nabla\bs{w}_{h,i}(p_i-\tilde{p}_i),\nabla\bs{v}_{h,i}^\mathrm{T}\right)_{\Omega_i^t}\\
    \label{eq:semi4-errest8}
    &\qquad+(\nabla\cdot(\bs{u}_i-\tilde{\bs{u}}_i),(\nabla\cdot\bs{w}_{h,i})q_{h,i})_{\Omega_i^t}-\left(\nabla(\bs{u}_i-\tilde{\bs{u}}_i)^\mathrm{T},\nabla\bs{w}_{h,i}q_{h,i}\right)_{\Omega_i^t}\Bigg],
\end{align}
{\color{black}where in the last step we apply (\ref{projadj2}),
\eqref{eq:semi4-erreq1}, \eqref{eq:semi4-erreq2}, and
$b(\bs{v}_{1},\bs{v}_{2};d_h^t,d_h^t)=b(\bs{v}_{h,1},\bs{v}_{h,2};d_h^t,d_h^t)=0$;}
$\bs{v}_{h,i}$, $\bs{w}_{h,i}$ and $q_{h,i}$ are replaced by
$(\bs{v}_{h,i}-\bs{v}_i)+\bs{v}_i$,
$(\bs{w}_{h,i}-\bs{w}_i)+\bs{w}_i$ and $(q_{h,i}-q_i)+q_i$,
respectively, if they stay alone. In order to present the underlying
idea of our analysis briefly, without loss of generality, in the
following we present only the detailed estimate for one
representative term,
$(\mu_i(\nabla\cdot\bs{w}_{h,i})\nabla(\bs{u}_i-\tilde{\bs{u}}_i),\nabla\bs{v}_{h,i})_{\Omega_i^t}$,
while the remaining terms can be handled in a similar manner without
introducing any essential additional difficulty.
By H\"older's inequality, together with \eqref{w-bound}, Lemma
\ref{lem:semi3} and standard interpolation error estimates, we have
\begin{align}
    &\ \sum_{i=1}^2(\mu_i(\nabla\cdot\bs{w}_{h,i})\nabla(\bs{u}_i-\tilde{\bs{u}}_i),\nabla\bs{v}_{h,i})_{\Omega_i^t}\notag\\
    =&\ \sum_{i=1}^2\Big[(\mu_i(\nabla\cdot\bs{w}_{h,i})\nabla(\bs{u}_i-\tilde{\bs{u}}_i),\nabla(\bs{v}_{h,i}-\bs{v}_i))_{\Omega_i^t}+(\mu_i(\nabla\cdot(\bs{w}_{h,i}-\bs{w}_i))\nabla(\bs{u}_i-\tilde{\bs{u}}_i),\nabla\bs{v}_i)_{\Omega_i^t}\notag\\
    &\qquad+(\mu_i(\nabla\cdot\bs{w}_i)\nabla(\bs{u}_i-\tilde{\bs{u}}_i),\nabla\bs{v}_i)_{\Omega_i^t}\Big]\notag\\
    \le&\ C\sum_{i=1}^2\big(\|\nabla\cdot\bs{w}_{h,i}\|_{L^\infty}\|\nabla(\bs{u}_i-\tilde{\bs{u}}_i)\|_{L^2}\|\nabla(\bs{v}_{h,i}-\bs{v}_i)\|_{L^2}+\|\nabla\cdot(\bs{w}_{h,i}-\bs{w}_i)\|_{L^3}\|\nabla(\bs{u}_i-\tilde{\bs{u}}_i)\|_{L^2}\|\nabla\bs{v}_i\|_{L^6}\big)\notag\\
    &+Ch^{k+1}\left(\sum_{i=1}^2\big(\|\bs{u}_i\|_{H^{k+1}}+\|p_i\|_{H^k}\big)\right)\left(\sum_{i=1}^2\|(\nabla\cdot\bs{w}_i)\nabla\bs{v}_i\|_{H^1}\right)\qquad\qquad\qquad\qquad\text{(using Lemma
\ref{lem:semi3})}\notag\\
    \le&\ Ch^{k+1}\sum_{i=1}^2\big(\|\bs{w}_{h,i}\|_{W^{1,\infty}}\|\bs{u}_i\|_{H^{k+1}}\|\bs{v}_i\|_{H^2}+\|\bs{w}_i\|_{W^{2,3}}\|\bs{u}_i\|_{H^{k+1}}\|\bs{v}_i\|_{H^2}\big)\notag\\
    &+Ch^{k+1}\left(\sum_{i=1}^2\big(\|\bs{u}_i\|_{H^{k+1}}+\|p_i\|_{H^k}\big)\right)\left(\sum_{i=1}^2\|\bs{w}_i\|_{W^{2,\infty}}\|\bs{v}_i\|_{H^2}\right)\notag\\
    \le&\
    Ch^{k+1}\left(\sum_{i=1}^2\big(\|\bs{u}_i\|_{H^{k+1}}+\|p_i\|_{H^k}\big)\right)\left(\sum_{i=1}^2\|\bs{v}_i\|_{H^2}\right).\label{example_estimate}
\end{align}

Therefore, by applying \eqref{a-cont}, \eqref{b-cont}, the
Cauchy-Schwarz inequality, Lemmas \ref{lem:semi1}-\ref{lem:semi3},
the regularity assumption \eqref{regadj1} and standard interpolation
error estimates, \textcolor{black}{and by conducting similar
estimates to those} in \eqref{example_estimate} for the remaining
terms, we obtain the following estimation for
\eqref{eq:semi4-errest8}:
\begin{align}
    \label{eq:semi4-errest9}
    &\ \sum_{i=1}^2\left(\tilde{\bs{f}}_i,\textcolor{black}{\left.\frac{\partial\bs{u}_i}{\partial t}\right|_{\hat{\bs{x}}}^h-\left.\frac{\partial\tilde{\bs{u}}_i}{\partial t}\right|_{\hat{\bs{x}}}^h}\right)_{\Omega_i^t}\\
    \notag
    \le&\ C\left(\sum_{i=1}^2\left[\left\|\left.\frac{\partial\bs{u}_i}{\partial t}\right|_{\hat{\bs{x}}}^h-\left.\frac{\partial\tilde{\bs{u}}_i}{\partial t}\right|_{\hat{\bs{x}}}^h\right\|_{(H^1(\Omega_i^t))^d}+\left\|\left.\frac{\partial p_i}{\partial t}\right|_{\hat{\bs{x}}}^h-\left.\frac{\partial\tilde{p}_i}{\partial t}\right|_{\hat{\bs{x}}}^h\right\|_{L^2(\Omega_i^t)}\right]\right)\\
    \notag
    &\quad\ \times\left(\sum_{i=1}^2\left[\|\bs{v}_i-\bs{v}_{h,i}\|_{(H^1(\Omega_i^t))^d}+\|q_i-q_{h,i}\|_{L^2(\Omega_i^t)}\right]\right)\\
    \notag
    &+C\left(\sum_{i=1}^2\left[\|\bs{u}_i-\tilde{\bs{u}}_i\|_{(H^1(\Omega_i^t))^d}+\|p_i-\tilde{p}_i\|_{L^2(\Omega_i^t)}\right]\right)\times\left(\sum_{i=1}^2\left[\|\bs{v}_i-\bs{v}_{h,i}\|_{(H^1(\Omega_i^t))^d}+\|q_i-q_{h,i}\|_{L^2(\Omega_i^t)}\right]\right.\\
    \notag
    &\qquad\quad\left.+\sum_{i=1}^2\|\bs{w}_i-\bs{w}_{h,i}\|_{(W^{1,3}(\Omega_i^t))^d}
    \left[\|\bs{v}_i\|_{(W^{1,6}(\Omega_i^t))^d}+\|q_i\|_{L^6(\Omega_i^t)}\right]\right)\\
    \notag
    &+Ch^{k+1}\left(\sum_{i=1}^2\left[\|\bs{u}_i\|_{(H^{k+1}(\Omega_i^t))^d}+\|p_i\|_{H^k(\Omega_i^t)}\right]\right)\times\left(\sum_{i=1}^2\left[\|\bs{v}_i\|_{(H^2(\Omega_i^t))^d}+\|q_i\|_{H^1(\Omega_i^t)}\right]\right)\\
    \notag
    \le&\ Ch^{k+1}\left(\sum_{i=1}^2\left[
    \textcolor{black}{\|\bs{w}_i\|_{(H^{k+1}(\Omega_i^t))^d}}
    +\|\bs{u}_i\|_{(H^{k+1}(\Omega_i^t))^d}
    +\left\|\left.\frac{\partial\bs{u}_i}{\partial t}\right|_{\hat{\bs{x}}}\right\|_{(H^{k+1}(\Omega_i^t))^d}\right.\right.\\
    \notag
    &\qquad\qquad\qquad\left.+\left.\|p_i\|_{H^k(\Omega_i^t)}
    +\left\|\left.\frac{\partial p_i}{\partial t}\right|_{\hat{\bs{x}}}\right\|_{H^k(\Omega_i^t)}\right]\right)\times\left(\sum_{i=1}^2\left[\|\bs{v}_i\|_{(H^2(\Omega_i^t))^d}+\|q_i\|_{H^1(\Omega_i^t)}\right]\right)\\
    \notag
    \le&\ Ch^{k+1}\left(\sum_{i=1}^2\left[
    \textcolor{black}{\|\bs{w}_i\|_{(H^{k+1}(\Omega_i^t))^d}}
    +\|\bs{u}_i\|_{(H^{k+1}(\Omega_i^t))^d}
    +\left\|\left.\frac{\partial\bs{u}_i}{\partial t}\right|_{\hat{\bs{x}}}\right\|_{(H^{k+1}(\Omega_i^t))^d}\right.\right.\\
    \notag
    &\qquad\qquad\qquad\left.+\left.\|p_i\|_{H^k(\Omega_i^t)}
    +\left\|\left.\frac{\partial p_i}{\partial t}\right|_{\hat{\bs{x}}}\right\|_{H^k(\Omega_i^t)}\right]\right)\times\left(\sum_{i=1}^2\|\tilde{\bs{f}}_i\|_{(L^2(\Omega_i^t))^d}\right),
\end{align}
\textcolor{black}{where optimal finite element approximations to the
adjoint velocity $\textcolor{black}{\bs{v}}_i\in
(H^2(\Omega_i^t))^d$ in $H^1$ norm and to the adjoint pressure
$q_i\in H^1(\Omega_i^t)$ in $L^2$ norm are applied,} thereby leading
to
\begin{align*}
    &\ \left\|\left(\left.\frac{\partial\bs{u}_1}{\partial t}\right|_{\hat{\bs{x}}}^h-\left.\frac{\partial\tilde{\bs{u}}_1}{\partial t}\right|_{\hat{\bs{x}}}^h,\left.\frac{\partial\bs{u}_2}{\partial t}\right|_{\hat{\bs{x}}}^h-\left.\frac{\partial\tilde{\bs{u}}_2}{\partial t}\right|_{\hat{\bs{x}}}^h\right)\right\|_0\\
    =&\ \sup_{(\tilde{\bs{f}}_1,\tilde{\bs{f}}_2)\in(L^2(\Omega_1^t))^d\times(L^2(\Omega_2^t))^d}\frac{\displaystyle\left|\sum_{i=1}^2\left(\tilde{\bs{f}}_i,\left.\frac{\partial\bs{u}_i}{\partial t}\right|_{\hat{\bs{x}}}^h-\left.\frac{\partial\tilde{\bs{u}}_i}{\partial t}\right|_{\hat{\bs{x}}}^h\right)_{\Omega_i^t}\right|}{\|(\tilde{\bs{f}}_1,\tilde{\bs{f}}_2)\|_0}\\
    \le&\ Ch^{k+1}\sum_{i=1}^2\left[
    \textcolor{black}{\|\bs{w}_i\|_{(H^{k+1}(\Omega_i^t))^d}}
    +\|\bs{u}_i\|_{(H^{k+1}(\Omega_i^t))^d}
    +\left\|\left.\frac{\partial\bs{u}_i}{\partial t}\right|_{\hat{\bs{x}}}\right\|_{(H^{k+1}(\Omega_i^t))^d}\right.\\
    &\qquad\qquad\quad\ +\left.\|p_i\|_{H^k(\Omega_i^t)}
    +\left\|\left.\frac{\partial p_i}{\partial t}\right|_{\hat{\bs{x}}}\right\|_{H^k(\Omega_i^t)}\right],
\end{align*}
which is the desired $L^2$ error estimate shown in \eqref{eq:semi4}.
\end{proof}


\subsection{Error analysis of the semi-discrete scheme}

\begin{theorem}
Suppose that $((\bs{u}_1,\bs{u}_2),(p_1,p_2))$ is the solution to
\eqref{weak1} satisfying the regularity assumptions
\eqref{reg1}-\eqref{reg5}, and
$((\bs{u}_{h,1},\bs{u}_{h,2}),(p_{h,1},p_{h,2}))$ is the solution to
\eqref{semi1}-\eqref{semi2}, \textcolor{black}{with the initial
condition
$\bs{u}_{h,i}(\hat{\bs{x}}_i,0)=\tilde{\bs{u}}_i(\hat{\bs{x}}_i,0)\,
(i=1,2)$ defined in the $H^1$-projection
(\ref{proj1})-(\ref{proj2}). We also assume that the prescribed ALE
velocity $\bs{w}_i\in H^1(0,T;[(H^{k+1}\cap
W^{2,\infty})(\Omega_i^t)]^d)$ and its discrete counterpart
$\bs{w}_{h,i}$ satisfy the convergence properties \eqref{w-conv1}
and \eqref{w-conv2}, and that the adjoint regularity estimates
\eqref{regadj1}, \eqref{regadj2} and \eqref{regadj3} hold for the
adjoint problems of (\ref{proj1})-(\ref{proj2}). Then} we have the
following error estimate
\begin{equation}
    \label{semi-errest}
    \begin{aligned}
        &\ \sum_{i=1}^2\left[\|\bs{u}_i-\bs{u}_{h,i}\|_{L^\infty(0,T;(L^2(\Omega_i^t))^d)}+\left\|\left.\frac{\partial\bs{u}_i}{\partial t}\right|_{\hat{\bs{x}}}^h-\left.\frac{\partial\bs{u}_{h,i}}{\partial t}\right|_{\hat{\bs{x}}}^h\right\|_{L^2(0,T;(L^2(\Omega_i^t))^d)}\right]\\
        &+h\sum_{i=1}^2\left[\|\bs{u}_i-\bs{u}_{h,i}\|_{L^\infty(0,T;(H^1(\Omega_i^t))^d)}+\|p_i-p_{h,i}\|_{L^2(0,T;L^2(\Omega_i^t))}\right]\\
        \le&\ Ch^{k+1}\sum_{i=1}^2\left[
        \textcolor{black}{\|\bs{w}_i\|_{L^2(0,T;(H^{k+1}(\Omega_i^t))^d)}}
        +\|\bs{u}_i\|_{L^\infty(0,T;(H^{k+1}(\Omega_i^t))^d)}
        +\left\|\left.\frac{\partial\bs{u}_i}{\partial t}\right|_{\hat{\bs{x}}}\right\|_{L^2(0,T;(H^{k+1}(\Omega_i^t))^d)}\right.\\
        &\qquad\qquad\quad\ +\left.\|p_i\|_{L^\infty(0,T;H^k(\Omega_i^t))}
        +\left\|\left.\frac{\partial p_i}{\partial t}\right|_{\hat{\bs{x}}}\right\|_{L^2(0,T;H^k(\Omega_i^t))}\right].
    \end{aligned}
\end{equation}
\end{theorem}
\begin{proof}
Since $((\bs{u}_1,\bs{u}_2),(p_1,p_2))$ is the solution to
\eqref{weak1}, it satisfies the following equation
\begin{align*}
    &\sum_{i=1}^2\left[\left(\left.\frac{\partial\bs{u}_i}{\partial t}\right|_{\hat{\bs{x}}}^h,\bs{v}_{h,i}\right)_{\Omega_i^t}+(\mu_i\nabla\bs{u}_i,\nabla\bs{v}_{h,i})_{\Omega_i^t}-((\bs{w}_{h,i}\cdot\nabla)\bs{u}_i,\bs{v}_{h,i})_{\Omega_i^t}\right.\\
    &\qquad-(p_i,\nabla\cdot\bs{v}_{h,i})_{\Omega_i^t}\Bigg]=\sum_{i=1}^2(\bs{f}_i,\bs{v}_{h,i})_{\Omega_i^t}+\langle\bs{g},\bs{v}_{h,1}\rangle_{\Gamma^t},\qquad\forall\,(\bs{v}_{h,1},\bs{v}_{h,2})\in\bs{U}_h^t,
\end{align*}
\begin{equation*}
    \sum_{i=1}^2(\nabla\cdot\bs{u}_i,q_{h,i})_{\Omega_i^t}=0,\qquad\forall\,(q_{h,1},q_{h,2})\in Q_h^t.
\end{equation*}
Subtracting \eqref{semi1}-\eqref{semi2} from the above equations and employing $H^1$-projection \eqref{proj1}-\eqref{proj2}, we obtain the error equation
\begin{align*}
    &\sum_{i=1}^2\left[\left(\left.\frac{\partial\bs{u}_i}{\partial t}\right|_{\hat{\bs{x}}}^h-\left.\frac{\partial\bs{u}_{h,i}}{\partial t}\right|_{\hat{\bs{x}}}^h,\bs{v}_{h,i}\right)_{\Omega_i^t}+(\mu_i\nabla(\tilde{\bs{u}}_i-\bs{u}_{h,i}),\nabla\bs{v}_{h,i})_{\Omega_i^t}-((\bs{w}_{h,i}\cdot\nabla)(\tilde{\bs{u}}_i-\bs{u}_{h,i}),\bs{v}_{h,i})_{\Omega_i^t}\right.\\
    &\qquad-\kappa(\bs{u}_i-\tilde{\bs{u}}_i,\bs{v}_{h,i})_{\Omega_i^t}-(\tilde{p}_i-p_{h,i},\nabla\cdot\bs{v}_{h,i})_{\Omega_i^t}\Bigg]=0,\\
    &\sum_{i=1}^2(\nabla\cdot(\tilde{\bs{u}}_i-\bs{u}_{h,i}),q_{h,i})_{\Omega_i^t}=0.
\end{align*}
Introducing new variables $\eta_i=\bs{u}_i-\tilde{\bs{u}}_i$, $\xi_i=\tilde{\bs{u}}_i-\bs{u}_{h,i}$, $\delta_i=p_i-\tilde{p}_i$, $\phi_i=\tilde{p}_i-p_{h,i}$, we can rewrite the above error equations as follows:
\begin{eqnarray}
    &&\sum_{i=1}^2\left[\left(\left.\frac{\partial(\eta_i+\xi_i)}{\partial t}\right|_{\hat{\bs{x}}}^h,\bs{v}_{h,i}\right)_{\Omega_i^t}+(\mu_i\nabla\xi_i,\nabla\bs{v}_{h,i})_{\Omega_i^t}-((\bs{w}_{h,i}\cdot\nabla)\xi_i,\bs{v}_{h,i})_{\Omega_i^t}\right.\notag\\
    &&\qquad-\kappa(\eta_i,\bs{v}_{h,i})_{\Omega_i^t}-(\phi_i,\nabla\cdot\bs{v}_{h,i})_{\Omega_i^t}\Bigg]=0,\label{eq:semi-erreq1}\\
    &&\sum_{i=1}^2(\nabla\cdot\xi_i,q_{h,i})_{\Omega_i^t}=0.\label{eq:semi-erreq2}
\end{eqnarray}

Let $\bs{v}_{h,i}=\left.\dfrac{\partial\xi_i}{\partial
t}\right|_{\hat{\bs{x}}}^h$ in \eqref{eq:semi-erreq1}, yields
\begin{align*}
    &\sum_{i=1}^2\left[\left(\left.\frac{\partial(\eta_i+\xi_i)}{\partial t}\right|_{\hat{\bs{x}}}^h,\left.\frac{\partial\xi_i}{\partial t}\right|_{\hat{\bs{x}}}^h\right)_{\Omega_i^t}+\left(\mu_i\nabla\xi_i,\nabla\left(\left.\frac{\partial\xi_i}{\partial t}\right|_{\hat{\bs{x}}}^h\right)\right)_{\Omega_i^t}-\left((\bs{w}_{h,i}\cdot\nabla)\xi_i,\left.\frac{\partial\xi_i}{\partial t}\right|_{\hat{\bs{x}}}^h\right)_{\Omega_i^t}\right.\\
    &\qquad-\left.\kappa\left(\eta_i,\left.\frac{\partial\xi_i}{\partial t}\right|_{\hat{\bs{x}}}^h\right)_{\Omega_i^t}-\left(\phi_i,\nabla\cdot\left(\left.\frac{\partial\xi_i}{\partial t}\right|_{\hat{\bs{x}}}^h\right)\right)_{\Omega_i^t}\right]=0.
\end{align*}
Applying the Reynolds transport theorem, we obtain
\begin{align*}
    \left(\mu_i\nabla\xi_i,\nabla\left(\left.\frac{\partial\xi_i}{\partial t}\right|_{\hat{\bs{x}}}^h\right)\right)_{\Omega_i^t}&=\frac{1}{2}\left[\frac{d}{dt}(\mu_i\nabla\xi_i,\nabla\xi_i)_{\Omega_i^t}-(\mu_i(\nabla\cdot\bs{w}_{h,i})\nabla\xi_i,\nabla\xi_i)_{\Omega_i^t}\right.\\
    &\qquad\ \ +\left(\mu_i\nabla\xi_i\left(\nabla\bs{w}_{h,i}+\nabla\bs{w}_{h,i}^\mathrm{T}\right),\nabla\xi_i\right)_{\Omega_i^t}\bigg],\\
    \left(\phi_i,\nabla\cdot\left(\left.\frac{\partial\xi_i}{\partial t}\right|_{\hat{\bs{x}}}^h\right)\right)_{\Omega_i^t}&=-(\phi_i,(\nabla\cdot\bs{w}_{h,i})(\nabla\cdot\xi_i))_{\Omega_i^t}+\left(\phi_i\nabla\bs{w}_{h,i},\nabla\xi_i^\mathrm{T}\right)_{\Omega_i^t}.
\end{align*}
Then, we have
\begin{align*}
    &\ \sum_{i=1}^2\left[\left(\left.\frac{\partial\xi_i}{\partial t}\right|_{\hat{\bs{x}}}^h,\left.\frac{\partial\xi_i}{\partial t}\right|_{\hat{\bs{x}}}^h\right)_{\Omega_i^t}+\frac{1}{2}\frac{d}{dt}(\mu_i\nabla\xi_i,\nabla\xi_i)_{\Omega_i^t}\right]=\sum_{i=1}^2\left[-\left(\left.\frac{\partial\eta_i}{\partial t}\right|_{\hat{\bs{x}}}^h,\left.\frac{\partial\xi_i}{\partial t}\right|_{\hat{\bs{x}}}^h\right)_{\Omega_i^t}\right.\\
    +&\ \frac{1}{2}(\mu_i(\nabla\cdot\bs{w}_{h,i})\nabla\xi_i,\nabla\xi_i)_{\Omega_i^t}-\frac{1}{2}\left(\mu_i\nabla\xi_i\left(\nabla\bs{w}_{h,i}+\nabla\bs{w}_{h,i}^\mathrm{T}\right),\nabla\xi_i\right)_{\Omega_i^t}+\left((\bs{w}_{h,i}\cdot\nabla)\xi_i,\left.\frac{\partial\xi_i}{\partial t}\right|_{\hat{\bs{x}}}^h\right)_{\Omega_i^t}\\
    +&\left.\kappa\left(\eta_i,\left.\frac{\partial\xi_i}{\partial t}\right|_{\hat{\bs{x}}}^h\right)_{\Omega_i^t}-(\phi_i,(\nabla\cdot\bs{w}_{h,i})(\nabla\cdot\xi_i))_{\Omega_i^t}+\left(\phi_i\nabla\bs{w}_{h,i},\nabla\xi_i^\mathrm{T}\right)_{\Omega_i^t}\right]=\sum_{k=1}^7H_k.
\end{align*}
Using the Cauchy-Schwarz inequality and the $\varepsilon$-Young's inequality, it follows that
\begin{align*}
    H_1&\le C\sum_{i=1}^2\left\|\left.\frac{\partial\eta_i}{\partial t}\right|_{\hat{\bs{x}}}^h\right\|_{0,\Omega_i^t}^2+\varepsilon\sum_{i=1}^2\left\|\left.\frac{\partial\xi_i}{\partial t}\right|_{\hat{\bs{x}}}^h\right\|_{0,\Omega_i^t}^2,\\
    H_2+H_3&\le C\sum_{i=1}^2\|\nabla\xi_i\|_{0,\Omega_i^t}^2,\\
    H_4&\le C\sum_{i=1}^2\|\nabla\xi_i\|_{0,\Omega_i^t}^2+\varepsilon\sum_{i=1}^2\left\|\left.\frac{\partial\xi_i}{\partial t}\right|_{\hat{\bs{x}}}^h\right\|_{0,\Omega_i^t}^2,\\
    H_5&\le C\sum_{i=1}^2\|\eta_i\|_{0,\Omega_i^t}^2+\varepsilon\sum_{i=1}^2\left\|\left.\frac{\partial\xi_i}{\partial t}\right|_{\hat{\bs{x}}}^h\right\|_{0,\Omega_i^t}^2,\\
    H_6+H_7&\le C\sum_{i=1}^2\|\nabla\xi_i\|_{0,\Omega_i^t}^2+\varepsilon_\phi\sum_{i=1}^2\|\phi_i\|_{0,\Omega_i^t}^2.
\end{align*}
Choosing a sufficiently small $\varepsilon$, leads to
\begin{equation}
    \label{eq:semi-errest1}
    \sum_{i=1}^2\left[\left\|\left.\frac{\partial\xi_i}{\partial t}\right|_{\hat{\bs{x}}}^h\right\|_{0,\Omega_i^t}^2+\frac{d}{dt}\|\nabla\xi_i\|_{0,\Omega_i^t}^2\right]\le C\sum_{i=1}^2\left[\|\nabla\xi_i\|_{0,\Omega_i^t}^2+\|\eta_i\|_{0,\Omega_i^t}^2+\left\|\left.\frac{\partial\eta_i}{\partial t}\right|_{\hat{\bs{x}}}^h\right\|_{0,\Omega_i^t}^2\right]+\varepsilon_\phi\sum_{i=1}^2\|\phi_i\|_{0,\Omega_i^t}^2.
\end{equation}

To estimate $\|\phi_i\|_{0,\Omega_i^t}$, we use the discrete inf-sup condition and the error equation \eqref{eq:semi-erreq1}, resulting in
\begin{equation}
    \label{eq:semi-errest2}
    \begin{aligned}
        \gamma\|(\phi_1,\phi_2)\|_0&\le\sup_{(\bs{v}_{h,1},\bs{v}_{h,2})\in\bs{U}_h^t}\frac{\displaystyle\sum_{i=1}^2(\phi_i,\nabla\cdot\bs{v}_{h,i})_{\Omega_i^t}}{\|(\bs{v}_{h,1},\bs{v}_{h,2})\|_1}\\
        &=\sup_{(\bs{v}_{h,1},\bs{v}_{h,2})\in\bs{U}_h^t}\frac{1}{\|(\bs{v}_{h,1},\bs{v}_{h,2})\|_1}\sum_{i=1}^2\left[\left(\left.\frac{\partial(\eta_i+\xi_i)}{\partial t}\right|_{\hat{\bs{x}}}^h,\bs{v}_{h,i}\right)_{\Omega_i^t}\right.\\
        &\qquad\qquad\qquad\quad+(\mu_i\nabla\xi_i,\nabla\bs{v}_{h,i})_{\Omega_i^t}-((\bs{w}_{h,i}\cdot\nabla)\xi_i,\bs{v}_{h,i})_{\Omega_i^t}-\kappa(\eta_i,\bs{v}_{h,i})_{\Omega_i^t}\Bigg]\\
        &\le\sum_{i=1}^2\left[\left\|\left.\frac{\partial\xi_i}{\partial t}\right|_{\hat{\bs{x}}}^h\right\|_{0,\Omega_i^t}+\left\|\left.\frac{\partial\eta_i}{\partial t}\right|_{\hat{\bs{x}}}^h\right\|_{0,\Omega_i^t}\right]+C\sum_{i=1}^2\left[\|\nabla\xi_i\|_{0,\Omega_i^t}+\|\eta_i\|_{0,\Omega_i^t}\right].
    \end{aligned}
\end{equation}
Substituting \eqref{eq:semi-errest2} into \eqref{eq:semi-errest1} and choosing a sufficiently small $\varepsilon_\phi$, we obtain
\begin{equation*}
    \sum_{i=1}^2\left[\left\|\left.\frac{\partial\xi_i}{\partial t}\right|_{\hat{\bs{x}}}^h\right\|_{0,\Omega_i^t}^2+\frac{d}{dt}\|\nabla\xi_i\|_{0,\Omega_i^t}^2\right]\le C\sum_{i=1}^2\left[\|\nabla\xi_i\|_{0,\Omega_i^t}^2+\|\eta_i\|_{0,\Omega_i^t}^2+\left\|\left.\frac{\partial\eta_i}{\partial t}\right|_{\hat{\bs{x}}}^h\right\|_{0,\Omega_i^t}^2\right].
\end{equation*}
Integrating both sides of the above inequality in time over $[0,T]$,
and applying the Poincar\'e inequality and the Gr\"onwall's
inequality, together with the initial condition
$\bs{u}_{h,i}(\hat{\bs{x}}_i,0)=\tilde{\bs{u}}_i(\hat{\bs{x}}_i,0)$,
we obtain
\begin{equation}
    \label{eq:semi-errest3}
    \begin{aligned}
        &\ \sum_{i=1}^2\left[\left\|\left.\frac{\partial\xi_i}{\partial t}\right|_{\hat{\bs{x}}}^h\right\|_{L^2(0,T;(L^2(\Omega_i^t))^d)}+\|\xi_i\|_{L^\infty(0,T;(H^1(\Omega_i^t))^d)}\right]\\
        \le&\ C\sum_{i=1}^2\left[\|\eta_i\|_{L^2(0,T;(L^2(\Omega_i^t))^d)}+\left\|\left.\frac{\partial\eta_i}{\partial t}\right|_{\hat{\bs{x}}}^h\right\|_{L^2(0,T;(L^2(\Omega_i^t))^d)}\right].
    \end{aligned}
\end{equation}
Moreover, integrating both sides of \eqref{eq:semi-errest2} in time over $[0,T]$, we obtain
\begin{equation}
    \label{eq:semi-errest4}
    \sum_{i=1}^2\|\phi_i\|_{L^2(0,T;L^2(\Omega_i^t))}\le C\sum_{i=1}^2\left[\|\eta_i\|_{L^2(0,T;(L^2(\Omega_i^t))^d)}+\left\|\left.\frac{\partial\eta_i}{\partial t}\right|_{\hat{\bs{x}}}^h\right\|_{L^2(0,T;(L^2(\Omega_i^t))^d)}\right].
\end{equation}
Finally, by using the triangle inequality, Lemmas \ref{lem:semi1}
and \ref{lem:semi4}, together with \eqref{eq:semi-errest3} and
\eqref{eq:semi-errest4}, we conclude that
\begin{align*}
    &\ \sum_{i=1}^2\left[\|\bs{u}_i-\bs{u}_{h,i}\|_{L^\infty(0,T;(L^2(\Omega_i^t))^d)}+\left\|\left.\frac{\partial\bs{u}_i}{\partial t}\right|_{\hat{\bs{x}}}^h-\left.\frac{\partial\bs{u}_{h,i}}{\partial t}\right|_{\hat{\bs{x}}}^h\right\|_{L^2(0,T;(L^2(\Omega_i^t))^d)}\right]\\
    =&\ \sum_{i=1}^2\left[\|\eta_i+\xi_i\|_{L^\infty(0,T;(L^2(\Omega_i^t))^d)}+\left\|\left.\frac{\partial(\eta_i+\xi_i)}{\partial t}\right|_{\hat{\bs{x}}}^h\right\|_{L^2(0,T;(L^2(\Omega_i^t))^d)}\right]\\
    \le&\ C\sum_{i=1}^2\left[\|\eta_i\|_{L^\infty(0,T;(L^2(\Omega_i^t))^d)}+\|\eta_i\|_{L^2(0,T;(L^2(\Omega_i^t))^d)}+\left\|\left.\frac{\partial\eta_i}{\partial t}\right|_{\hat{\bs{x}}}^h\right\|_{L^2(0,T;(L^2(\Omega_i^t))^d)}\right]\\
    \le&\ Ch^{k+1}\sum_{i=1}^2\left[
    \textcolor{black}{\|\bs{w}_i\|_{L^2(0,T;(H^{k+1}(\Omega_i^t))^d)}}
    +\|\bs{u}_i\|_{L^\infty(0,T;(H^{k+1}(\Omega_i^t))^d)}
    +\left\|\left.\frac{\partial\bs{u}_i}{\partial t}\right|_{\hat{\bs{x}}}\right\|_{L^2(0,T;(H^{k+1}(\Omega_i^t))^d)}\right.\\
    &\qquad\qquad\quad\ +\left.\|p_i\|_{L^\infty(0,T;H^k(\Omega_i^t))}
    +\left\|\left.\frac{\partial p_i}{\partial t}\right|_{\hat{\bs{x}}}\right\|_{L^2(0,T;H^k(\Omega_i^t))}\right],
\end{align*}
and
\begin{align*}
    &\ \sum_{i=1}^2\left[\|\bs{u}_i-\bs{u}_{h,i}\|_{L^\infty(0,T;(H^1(\Omega_i^t))^d)}+\|p_i-p_{h,i}\|_{L^2(0,T;L^2(\Omega_i^t))}\right]\\
    =&\ \sum_{i=1}^2\left[\|\eta_i+\xi_i\|_{L^\infty(0,T;(H^1(\Omega_i^t))^d)}+\|\delta_i+\phi_i\|_{L^2(0,T;L^2(\Omega_i^t))}\right]\\
    \le&\ C\sum_{i=1}^2\left[\|\eta_i\|_{L^\infty(0,T;(H^1(\Omega_i^t))^d)}+\|\eta_i\|_{L^2(0,T;(L^2(\Omega_i^t))^d)}+\left\|\left.\frac{\partial\eta_i}{\partial t}\right|_{\hat{\bs{x}}}^h\right\|_{L^2(0,T;(L^2(\Omega_i^t))^d)}+\|\delta_i\|_{L^2(0,T;L^2(\Omega_i^t))}\right]\\
    \le&\ Ch^k\sum_{i=1}^2\left[
    \textcolor{black}{\|\bs{w}_i\|_{L^2(0,T;(H^{k+1}(\Omega_i^t))^d)}}
    +\|\bs{u}_i\|_{L^\infty(0,T;(H^{k+1}(\Omega_i^t))^d)}
    +\left\|\left.\frac{\partial\bs{u}_i}{\partial t}\right|_{\hat{\bs{x}}}\right\|_{L^2(0,T;(H^{k+1}(\Omega_i^t))^d)}\right.\\
    &\qquad\qquad\ +\left.\|p_i\|_{L^\infty(0,T;H^k(\Omega_i^t))}
    +\left\|\left.\frac{\partial p_i}{\partial t}\right|_{\hat{\bs{x}}}\right\|_{L^2(0,T;H^k(\Omega_i^t))}\right],
\end{align*}
which yields the error estimate \eqref{semi-errest}.
\end{proof}

\section{Fully discrete ALE-finite element approximation}\label{sec:full}

Let $\Delta t\ (0<\Delta t<1)$ be the time step size and
$t^n=n\Delta t$ for $n=0,\cdots,N$ such that $T=N\Delta t$. For
$n=0,\cdots,N$ and $i=1,2$, we denote $\Omega_i^n=\Omega_i^{t^n}$,
$\bs{U}_h^n=\bs{U}_h^{t^n}$, $Q_h^n=Q_h^{t^n}$,
$\bs{X}_{h,i}^{n}=\bs{X}_{h,i}^{t^n}$,
$\varphi^n(\bs{x}_i)=\varphi(\bs{x}_i(\hat{\bs{x}}_i,t^n),t^n)
=\hat{\varphi}^n(\hat{\bs{x}}_i)\circ\big(\bs{X}_{h,i}^{n}\big)^{-1}$,
$\bs{X}_{h,i}^{n+1,n}=\bs{X}_{h,i}^n\circ\big(\bs{X}_{h,i}^{n+1}\big)^{-1}$.
With the above notations, we introduce the following fully discrete,
backward Euler ALE-FEM for \eqref{weak1}: for $n=0,\cdots,N-1$, find
$\big(\bs{u}_{h,1}^{n+1},\bs{u}_{h,2}^{n+1}\big)\in\bs{U}_h^{n+1}$
and $\big(p_{h,1}^{n+1},p_{h,2}^{n+1}\big)\in Q_h^{n+1}$ such that
\begin{eqnarray}
        &&\sum_{i=1}^2\left[\left(\frac{\bs{u}_{h,i}^{n+1}-\bs{u}_{h,i}^n\circ\bs{X}_{h,i}^{n+1,n}}{\Delta t},\bs{v}_{h,i}^{n+1}\right)_{\Omega_i^{n+1}}+\left(\mu_i\nabla\bs{u}_{h,i}^{n+1},\nabla\bs{v}_{h,i}^{n+1}\right)_{\Omega_i^{n+1}}-\left(\big(\bs{w}_{h,i}^{n+1}\cdot\nabla\big)\bs{u}_{h,i}^{n+1},\bs{v}_{h,i}^{n+1}\right)_{\Omega_i^{n+1}}\right.\notag\\
        &&-\left(p_{h,i}^{n+1},\nabla\cdot\bs{v}_{h,i}^{n+1}\right)_{\Omega_i^{n+1}}\Bigg]=\sum_{i=1}^2\left(\bs{f}_i^{n+1},\bs{v}_{h,i}^{n+1}\right)_{\Omega_i^{n+1}}+\left\langle\bs{g}^{n+1},\bs{v}_{h,1}^{n+1}\right\rangle_{\Gamma^{n+1}},\ \forall\,\big(\bs{v}_{h,1}^{n+1},\bs{v}_{h,2}^{n+1}\big)\in\bs{U}_h^{n+1},\label{full1}\\
&&\sum_{i=1}^2\left(\nabla\cdot\bs{u}_{h,i}^{n+1},q_{h,i}^{n+1}\right)_{\Omega_i^{n+1}}=0,\
\forall\,\big(q_{h,1}^{n+1},q_{h,2}^{n+1}\big)\in
Q_h^{n+1}.\label{full2}
\end{eqnarray}

\begin{lemma}[\cite{san2009convergence,lan2020monolithic}]
\label{lem:full0}
For each $i=1,2$, suppose that $\bs{F}_{h,i}^t$ is the Jacobian matrix of the ALE mapping $\bs{X}_{h,i}^t$ with its Jacobian $J_{h,i}^t=\det(\bs{F}_{h,i}^t)$ for $t\in[0,T]$, then the following bounds hold
\begin{equation*}
    0<C_1\le\big\|J_{h,i}^t\big\|_{L^\infty(\Omega_i^0)}\le C_2,\qquad C_2^{-1}\le\big\|(J_{h,i}^t)^{-1}\big\|_{L^\infty(\Omega_i^t)}\le C_1^{-1},
\end{equation*}
\begin{equation*}
    0<C_1\le\big\|\bs{F}_{h,i}^t\big\|_{(L^\infty(\Omega_i^0))^{d\times d}}\le C_2,\qquad C_2^{-1}\le\big\|(\bs{F}_{h,i}^t)^{-1}\big\|_{(L^\infty(\Omega_i^t))^{d\times d}}\le C_1^{-1},
\end{equation*}
where $C_1,C_2$ are constants depending only on the continuous ALE mapping $\bs{X}_i^t$ ($i=1,2$).
\end{lemma}

\subsection{Error analysis of the fully discrete scheme}

\begin{theorem}
Suppose that $((\bs{u}_1,\bs{u}_2),(p_1,p_2))$ is the solution to
\eqref{weak1} satisfying the regularity assumptions
\eqref{reg1}-\eqref{reg5}, and
$\big(\big(\bs{u}_{h,1}^{n+1},\bs{u}_{h,2}^{n+1}\big),\big(p_{h,1}^{n+1},p_{h,2}^{n+1}\big)\big)$
($n=1,\cdots,N$) is the solution to \eqref{full1}-\eqref{full2},
\textcolor{black}{with the initial condition
$\bs{u}_{h,i}^0=\tilde{\bs{u}}_i^0\, (i=1,2)$ defined in the
$H^1$-projection (\ref{proj1})-(\ref{proj2}). We also assume that
the prescribed ALE velocity $\bs{w}_i\in H^1(0,T;[(H^{k+1}\cap
W^{2,\infty})(\Omega_i^t)]^d)$ and its discrete counterpart
$\bs{w}_{h,i}$ satisfy the convergence properties \eqref{w-conv1}
and \eqref{w-conv2}, and that the adjoint regularity estimates
\eqref{regadj1}, \eqref{regadj2} and \eqref{regadj3} hold for the
adjoint problems of (\ref{proj1})-(\ref{proj2}). Then} we have the
following error \textcolor{black}{estimates}
\begin{align}
    \label{full-errest1}
    &\begin{aligned}
        &\ \sum_{i=1}^2\left[\left\|\bs{u}_i^N-\bs{u}_{h,i}^N\right\|_{0,\Omega_i^N}+\left(\Delta t\sum_{n=1}^N\left\|\partial_t(\bs{u}_i^n-\bs{u}_{h,i}^n)\right\|_{0,\Omega_i^n}^2\right)^\frac{1}{2}\right]\\
        \le&\ C(h^{k+1}+\Delta t)\sum_{i=1}^2\left[
        \textcolor{black}{\|\bs{w}_i\|_{L^2(0,T;(H^{k+1}(\Omega_i^t))^d)}}
        +\|\bs{u}_i\|_{L^\infty(0,T;(H^{k+1}(\Omega_i^t))^d)}
        +\left\|\left.\frac{\partial\bs{u}_i}{\partial t}\right|_{\hat{\bs{x}}}\right\|_{L^2(0,T;(H^{k+1}(\Omega_i^t))^d)}\right.\\
        &\qquad\qquad\qquad\qquad\ +\|p_i\|_{L^\infty(0,T;H^k(\Omega_i^t))}
        +\left\|\left.\frac{\partial p_i}{\partial t}\right|_{\hat{\bs{x}}}\right\|_{L^2(0,T;H^k(\Omega_i^t))}
        +\left.\left\|\left.\frac{\partial^2\bs{u}_i}{\partial t^2}\right|_{\hat{\bs{x}}}\right\|_{L^2(0,T;(L^2(\Omega_i^t))^d)}\right],
    \end{aligned}\\
    \label{full-errest2}
    &\begin{aligned}
        &\ \sum_{i=1}^2\left[\left\|\bs{u}_i^N-\bs{u}_{h,i}^N\right\|_{1,\Omega_i^N}+\left(\Delta t\sum_{n=1}^N\left\|p_i^n-p_{h,i}^n\right\|_{0,\Omega_i^n}^2\right)^\frac{1}{2}\right]\\
        \le&\ C(h^k+\Delta t)\sum_{i=1}^2\left[
        \textcolor{black}{\|\bs{w}_i\|_{L^2(0,T;(H^{k+1}(\Omega_i^t))^d)}}
        +\|\bs{u}_i\|_{L^\infty(0,T;(H^{k+1}(\Omega_i^t))^d)}
        +\left\|\left.\frac{\partial\bs{u}_i}{\partial t}\right|_{\hat{\bs{x}}}\right\|_{L^2(0,T;(H^{k+1}(\Omega_i^t))^d)}\right.\\
        &\qquad\qquad\qquad\quad\ +\|p_i\|_{L^\infty(0,T;H^k(\Omega_i^t))}
        +\left\|\left.\frac{\partial p_i}{\partial t}\right|_{\hat{\bs{x}}}\right\|_{L^2(0,T;H^k(\Omega_i^t))}
        +\left.\left\|\left.\frac{\partial^2\bs{u}_i}{\partial t^2}\right|_{\hat{\bs{x}}}\right\|_{L^2(0,T;(L^2(\Omega_i^t))^d)}\right],
    \end{aligned}
\end{align}
where $\partial_t\bs{\phi}_i^n=\dfrac{\bs{\phi}_i^n-\bs{\phi}_i^{n-1}\circ\bs{X}_{h,i}^{n,n-1}}{\Delta t}$.
\end{theorem}
\begin{proof}
Taking $t=t^{n+1}$, since
$\big(\big(\bs{u}_1^{n+1},\bs{u}_2^{n+1}),(p_1^{n+1},p_2^{n+1}\big)\big)$
is the solution to \eqref{weak1}, it satisfies the following
equation
\begin{eqnarray*}
    &&\sum_{i=1}^2\left[\left(\left(\left.\frac{\partial\bs{u}_i}{\partial t}\right|_{\hat{\bs{x}}}^h\right)^{n+1},\bs{v}_{h,i}^{n+1}\right)_{\Omega_i^{n+1}}+\left(\mu_i\nabla\bs{u}_i^{n+1},\nabla\bs{v}_{h,i}^{n+1}\right)_{\Omega_i^{n+1}}-\left(\big(\bs{w}_{h,i}^{n+1}\cdot\nabla\big)\bs{u}_i^{n+1},\bs{v}_{h,i}^{n+1}\right)_{\Omega_i^{n+1}}\right.\\
    &&\qquad-\left.\left(p_i^{n+1},\nabla\cdot\bs{v}_{h,i}^{n+1}\right)_{\Omega_i^{n+1}}\rule{0ex}{5ex}\right]=\sum_{i=1}^2\left(\bs{f}_i^{n+1},\bs{v}_{h,i}^{n+1}\right)_{\Omega_i^{n+1}}+\left\langle\bs{g}^{n+1},\bs{v}_{h,1}^{n+1}\right\rangle_{\Gamma^{n+1}},\ \forall\,\big(\bs{v}_{h,1}^{n+1},\bs{v}_{h,2}^{n+1}\big)\in\bs{U}_h^{n+1},\\
    &&\sum_{i=1}^2\left(\nabla\cdot\bs{u}_i^{n+1},q_{h,i}^{n+1}\right)_{\Omega_i^{n+1}}=0,\ \forall\,\big(q_{h,1}^{n+1},q_{h,2}^{n+1}\big)\in Q_h^{n+1}.
\end{eqnarray*}
Subtracting \eqref{full1}-\eqref{full2} from the above equations and
employing $H^1$-projection \eqref{proj1}-\eqref{proj2}, we obtain
the error equation
\begin{eqnarray*}
    &&\sum_{i=1}^2\left[\left(\left(\left.\frac{\partial\bs{u}_i}{\partial t}\right|_{\hat{\bs{x}}}^h\right)^{n+1}-\frac{\bs{u}_{h,i}^{n+1}-\bs{u}_{h,i}^n\circ\bs{X}_{h,i}^{n+1,n}}{\Delta t},\bs{v}_{h,i}^{n+1}\right)_{\Omega_i^{n+1}}+\left(\mu_i\nabla\big(\tilde{\bs{u}}_i^{n+1}-\bs{u}_{h,i}^{n+1}\big),\nabla\bs{v}_{h,i}^{n+1}\right)_{\Omega_i^{n+1}}\right.\\
    &&\qquad-\left(\big(\bs{w}_{h,i}^{n+1}\cdot\nabla\big)\big(\tilde{\bs{u}}_i^{n+1}-\bs{u}_{h,i}^{n+1}\big),\bs{v}_{h,i}^{n+1}\right)_{\Omega_i^{n+1}}-\kappa\left(\bs{u}_i^{n+1}-\tilde{\bs{u}}_i^{n+1},\bs{v}_{h,i}^{n+1}\right)_{\Omega_i^{n+1}}\\
    &&\qquad-\left.\left(\tilde{p}_i^{n+1}-p_{h,i}^{n+1},\nabla\cdot\bs{v}_{h,i}^{n+1}\right)_{\Omega_i^{n+1}}\rule{0ex}{5ex}\right]=0,\\
    &&\sum_{i=1}^2\left(\nabla\cdot\big(\tilde{\bs{u}}_i^{n+1}-\bs{u}_{h,i}^{n+1}\big),q_{h,i}^{n+1}\right)_{\Omega_i^{n+1}}=0.
\end{eqnarray*}
Introducing new variables
$\eta_i^{n+1}=\bs{u}_i^{n+1}-\tilde{\bs{u}}_i^{n+1}$,
$\xi_i^{n+1}=\tilde{\bs{u}}_i^{n+1}-\bs{u}_{h,i}^{n+1}$,
$\delta_i^{n+1}=p_i^{n+1}-\tilde{p}_i^{n+1}$,
$\phi_i^{n+1}=\tilde{p}_i^{n+1}-p_{h,i}^{n+1}$, we can rewrite the
above error equations as follows:
\begin{eqnarray}
        &&\sum_{i=1}^2\left[\left(\left(\left.\frac{\partial\bs{u}_i}{\partial t}\right|_{\hat{\bs{x}}}^h\right)^{n+1}-\frac{\bs{u}_i^{n+1}-\bs{u}_i^n\circ\bs{X}_{h,i}^{n+1,n}}{\Delta t},\bs{v}_{h,i}^{n+1}\right)_{\Omega_i^{n+1}}+\left(\frac{\eta_i^{n+1}-\eta_i^n\circ\bs{X}_{h,i}^{n+1,n}}{\Delta t},\bs{v}_{h,i}^{n+1}\right)_{\Omega_i^{n+1}}\right.\notag\\
        &&\qquad+\left(\frac{\xi_i^{n+1}-\xi_i^n\circ\bs{X}_{h,i}^{n+1,n}}{\Delta t},\bs{v}_{h,i}^{n+1}\right)_{\Omega_i^{n+1}}+\left(\mu_i\nabla\xi_i^{n+1},\nabla\bs{v}_{h,i}^{n+1}\right)_{\Omega_i^{n+1}}-\left(\big(\bs{w}_{h,i}^{n+1}\cdot\nabla\big)\xi_i^{n+1},\bs{v}_{h,i}^{n+1}\right)_{\Omega_i^{n+1}}\notag\\
        &&\qquad-\left.\kappa\left(\eta_i^{n+1},\bs{v}_{h,i}^{n+1}\right)_{\Omega_i^{n+1}}-\left(\phi_i^{n+1},\nabla\cdot\bs{v}_{h,i}^{n+1}\right)_{\Omega_i^{n+1}}\rule{0ex}{5ex}\right]=0,
        \label{eq:full-erreq1}\\
&&\sum_{i=1}^2\left(\nabla\cdot\xi_i^{n+1},q_{h,i}^{n+1}\right)_{\Omega_i^{n+1}}=0.
\label{eq:full-erreq2}
\end{eqnarray}

Let
$\bs{v}_{h,i}^{n+1}=\dfrac{\xi_i^{n+1}-\xi_i^n\circ\bs{X}_{h,i}^{n+1,n}}{\Delta
t}$ in \eqref{eq:full-erreq1}, yields
\begin{align}
    &\ \sum_{i=1}^2\left[\left\|\frac{\xi_i^{n+1}-\xi_i^n\circ\bs{X}_{h,i}^{n+1,n}}{\Delta t}\right\|_{0,\Omega_i^{n+1}}^2+\left(\mu_i\nabla\xi_i^{n+1},\nabla\left(\frac{\xi_i^{n+1}-\xi_i^n\circ\bs{X}_{h,i}^{n+1,n}}{\Delta t}\right)\right)_{\Omega_i^{n+1}}\right]\notag\\
    =&\ \sum_{i=1}^2\left[-\left(\left(\left.\frac{\partial\bs{u}_i}{\partial t}\right|_{\hat{\bs{x}}}^h\right)^{n+1}-\frac{\bs{u}_i^{n+1}-\bs{u}_i^n\circ\bs{X}_{h,i}^{n+1,n}}{\Delta t},\frac{\xi_i^{n+1}-\xi_i^n\circ\bs{X}_{h,i}^{n+1,n}}{\Delta t}\right)_{\Omega_i^{n+1}}\right.\notag\\
    &\ -\left(\frac{\eta_i^{n+1}-\eta_i^n\circ\bs{X}_{h,i}^{n+1,n}}{\Delta t},\frac{\xi_i^{n+1}-\xi_i^n\circ\bs{X}_{h,i}^{n+1,n}}{\Delta t}\right)_{\Omega_i^{n+1}}+\left(\big(\bs{w}_{h,i}^{n+1}\cdot\nabla\big)\xi_i^{n+1},\frac{\xi_i^{n+1}-\xi_i^n\circ\bs{X}_{h,i}^{n+1,n}}{\Delta t}\right)_{\Omega_i^{n+1}}\notag\\
    &\ +\left.\kappa\left(\eta_i^{n+1},\frac{\xi_i^{n+1}-\xi_i^n\circ\bs{X}_{h,i}^{n+1,n}}{\Delta t}\right)_{\Omega_i^{n+1}}+\left(\phi_i^{n+1},\nabla\cdot\left(\frac{\xi_i^{n+1}-\xi_i^n\circ\bs{X}_{h,i}^{n+1,n}}{\Delta
    t}\right)\right)_{\Omega_i^{n+1}}\rule{0ex}{5ex}\right]=\sum_{k=1}^5H_k.\label{eq-Full}
\end{align}
Using the Cauchy-Schwarz inequality and the $\varepsilon$-Young's inequality, it follows that
\begin{align*}
    H_1&\le C\sum_{i=1}^2\left\|\left(\left.\frac{\partial\bs{u}_i}{\partial t}\right|_{\hat{\bs{x}}}^h\right)^{n+1}-\frac{\bs{u}_i^{n+1}-\bs{u}_i^n\circ\bs{X}_{h,i}^{n+1,n}}{\Delta t}\right\|_{0,\Omega_i^{n+1}}^2+\varepsilon\sum_{i=1}^2\left\|\frac{\xi_i^{n+1}-\xi_i^n\circ\bs{X}_{h,i}^{n+1,n}}{\Delta t}\right\|_{0,\Omega_i^{n+1}}^2,\\
    H_2&\le C\sum_{i=1}^2\left\|\frac{\eta_i^{n+1}-\eta_i^n\circ\bs{X}_{h,i}^{n+1,n}}{\Delta t}\right\|_{0,\Omega_i^{n+1}}^2+\varepsilon\sum_{i=1}^2\left\|\frac{\xi_i^{n+1}-\xi_i^n\circ\bs{X}_{h,i}^{n+1,n}}{\Delta t}\right\|_{0,\Omega_i^{n+1}}^2,\\
    H_3&\le C\sum_{i=1}^2\left\|\nabla\xi_i^{n+1}\right\|_{0,\Omega_i^{n+1}}^2+\varepsilon\sum_{i=1}^2\left\|\frac{\xi_i^{n+1}-\xi_i^n\circ\bs{X}_{h,i}^{n+1,n}}{\Delta t}\right\|_{0,\Omega_i^{n+1}}^2,\\
    H_4&\le C\sum_{i=1}^2\left\|\eta_i^{n+1}\right\|_{0,\Omega_i^{n+1}}^2+\varepsilon\sum_{i=1}^2\left\|\frac{\xi_i^{n+1}-\xi_i^n\circ\bs{X}_{h,i}^{n+1,n}}{\Delta t}\right\|_{0,\Omega_i^{n+1}}^2.
\end{align*}

We first proceed to estimate $\left(\mu_i\nabla\xi_i^{n+1},\nabla\left(\dfrac{\xi_i^{n+1}-\xi_i^n\circ\bs{X}_{h,i}^{n+1,n}}{\Delta t}\right)\right)_{\Omega_i^{n+1}}$. By using the inequality $-ab\ge-\dfrac{a^2+b^2}{2}$, we obtain
\begin{align*}
    &\ \left(\mu_i\nabla\xi_i^{n+1},\nabla\left(\frac{\xi_i^{n+1}-\xi_i^n\circ\bs{X}_{h,i}^{n+1,n}}{\Delta t}\right)\right)_{\Omega_i^{n+1}}\\
    =&\ \frac{\mu_i}{\Delta t}\left\|\nabla\xi_i^{n+1}\right\|_{0,\Omega_i^{n+1}}^2-\frac{\mu_i}{\Delta t}\left(\nabla\xi_i^{n+1},\nabla\big(\xi_i^n\circ\bs{X}_{h,i}^{n+1,n}\big)\right)_{\Omega_i^{n+1}}\\
    \ge&\ \frac{\mu_i}{\Delta t}\left\|\nabla\xi_i^{n+1}\right\|_{0,\Omega_i^{n+1}}^2-\frac{\mu_i}{2\Delta t}\left\|\nabla\xi_i^{n+1}\right\|_{0,\Omega_i^{n+1}}^2-\frac{\mu_i}{2\Delta t}\left\|\nabla\big(\xi_i^n\circ\bs{X}_{h,i}^{n+1,n}\big)\right\|_{0,\Omega_i^{n+1}}^2\\
    =&\ \frac{\mu_i}{2\Delta t}\left\|\nabla\xi_i^{n+1}\right\|_{0,\Omega_i^{n+1}}^2-\frac{\mu_i}{2\Delta t}\left\|\nabla\big(\xi_i^n\circ\bs{X}_{h,i}^{n+1,n}\big)\right\|_{0,\Omega_i^{n+1}}^2\\
    =&\ \frac{\mu_i}{2\Delta t}\left(\left\|\nabla\xi_i^{n+1}\right\|_{0,\Omega_i^{n+1}}^2-\left\|\nabla\xi_i^n\right\|_{0,\Omega_i^n}^2\right)+\frac{\mu_i}{2\Delta t}\left(\left\|\nabla\xi_i^n\right\|_{0,\Omega_i^n}^2-\left\|\nabla\big(\xi_i^n\circ\bs{X}_{h,i}^{n+1,n}\big)\right\|_{0,\Omega_i^{n+1}}^2\right).
\end{align*}
For the second term on the right-hand side of the above expression,
we have
\begin{align*}
    &\ \frac{\mu_i}{2\Delta t}\left(\left\|\nabla\xi_i^n\right\|_{0,\Omega_i^n}^2-\left\|\nabla\big(\xi_i^n\circ\bs{X}_{h,i}^{n+1,n}\big)\right\|_{0,\Omega_i^{n+1}}^2\right)\\
    =&\ -\frac{\mu_i}{2\Delta t}\int_{t^n}^{t^{n+1}}\frac{d}{dt}\left(\int_{\Omega_i^t}\left(\nabla_{\bs{x}_i^t}\big(\xi_i^n\circ\bs{X}_{h,i}^{t,n}\big)\right)^2\,d\bs{x}\right)\,dt\\
    =&\ -\frac{\mu_i}{2\Delta t}\int_{t^n}^{t^{n+1}}\frac{d}{dt}\left(\int_{\Omega_i^0}\left(\nabla_{\hat{\bs{x}}_i}\hat{\xi}_i^n\left(\frac{\partial\bs{x}_i^t}{\partial\hat{\bs{x}}_i}\right)^{-1}\right)^2J_{h,i}^t\,d\hat{\bs{x}}\right)\,dt\\
    =&\ -\frac{\mu_i}{2\Delta t}\int_{t^n}^{t^{n+1}}\int_{\Omega_i^0}\left(2\left(\nabla_{\hat{\bs{x}}_i}\hat{\xi}_i^n\left(\frac{\partial\bs{x}_i^t}{\partial\hat{\bs{x}}_i}\right)^{-1}\right):\left(-\nabla_{\hat{\bs{x}}_i}\hat{\xi}_i^n\left(\frac{\partial\bs{x}_i^t}{\partial\hat{\bs{x}}_i}\right)^{-1}\left(\frac{\partial\bs{w}_{h,i}}{\partial\hat{\bs{x}}_i}\right)\left(\frac{\partial\bs{x}_i^t}{\partial\hat{\bs{x}}_i}\right)^{-1}\right)J_{h,i}^t\right.\\
    &\qquad\qquad\qquad\qquad\ \ +\left.\left(\nabla_{\hat{\bs{x}}_i}\hat{\xi}_i^n\left(\frac{\partial\bs{x}_i^t}{\partial\hat{\bs{x}}_i}\right)^{-1}\right)^2J_{h,i}^t\big(\nabla_{\bs{x}_i^t}\cdot\bs{w}_{h,i}\big)\right)\,d\hat{\bs{x}}\,dt\qquad\qquad\qquad\qquad\qquad \text{(using \eqref{identity-F})}\\
    =&\ -\frac{\mu_i}{2\Delta t}\int_{t^n}^{t^{n+1}}\int_{\Omega_i^n}\left[2\left(\nabla_{\bs{x}_i^n}\xi_i^n\left(\frac{\partial\bs{x}_i^t}{\partial\bs{x}_i^n}\right)^{-1}\right):\left(-\nabla_{\bs{x}_i^n}\xi_i^n\left(\frac{\partial\bs{x}_i^t}{\partial\bs{x}_i^n}\right)^{-1}\left(\frac{\partial\bs{w}_{h,i}}{\partial\bs{x}_i^n}\right)\left(\frac{\partial\bs{x}_i^t}{\partial\bs{x}_i^n}\right)^{-1}\right)\right.\\
    &\qquad\qquad\qquad\qquad\ \ +\left.\left(\nabla_{\bs{x}_i^n}\xi_i^n\left(\frac{\partial\bs{x}_i^t}{\partial\bs{x}_i^n}\right)^{-1}\right)^2\left(\left(\frac{\partial\bs{w}_{h,i}}{\partial\bs{x}_i^n}\right):\left(\frac{\partial\bs{x}_i^t}{\partial\bs{x}_i^n}\right)^{-\mathrm{T}}\right)\right]\frac{J_{h,i}^t}{J_{h,i}^n}\,d\bs{x}^n\,dt\\
    \ge&\ -\frac{C}{\Delta t}\int_{t^n}^{t^{n+1}}\left\|\nabla\xi_i^n\right\|_{0,\Omega_i^n}^2\,dt\ge
    -C\left\|\nabla\xi_i^n\right\|_{0,\Omega_i^n}^2,\qquad\qquad\qquad\qquad\qquad\quad\qquad\text{(using \eqref{w-bound} and Lemma \ref{lem:full0})}
\end{align*}
thereby leading to
\begin{equation*}
    \left(\mu_i\nabla\xi_i^{n+1},\nabla\left(\frac{\xi_i^{n+1}-\xi_i^n\circ\bs{X}_{h,i}^{n+1,n}}{\Delta t}\right)\right)_{\Omega_i^{n+1}}\ge\frac{\mu_i}{2\Delta t}\left(\left\|\nabla\xi_i^{n+1}\right\|_{0,\Omega_i^{n+1}}^2-\left\|\nabla\xi_i^n\right\|_{0,\Omega_i^n}^2\right)-C\left\|\nabla\xi_i^n\right\|_{0,\Omega_i^n}^2.
\end{equation*}

We next consider the estimate of
$\left\|\left(\left.\dfrac{\partial\bs{u}_i}{\partial
t}\right|_{\hat{\bs{x}}}^h\right)^{n+1}-\dfrac{\bs{u}_i^{n+1}-\bs{u}_i^n\circ\bs{X}_{h,i}^{n+1,n}}{\Delta
t}\right\|_{0,\Omega_i^{n+1}}$. \textcolor{black}{By} the Taylor's
expansion,
$\displaystyle\hat{\bs{u}}_i^n=\hat{\bs{u}}_i^{n+1}-\Delta
t\left(\frac{\partial\hat{\bs{u}}_i}{\partial
t}\right)^{n+1}+\int_{t^n}^{t^{n+1}}(t-t^n)\frac{\partial^2\hat{\bs{u}}_i^t}{\partial
t^2}\,dt$, we have
\begin{align*}
    &\ \left\|\left(\left.\frac{\partial\bs{u}_i}{\partial t}\right|_{\hat{\bs{x}}}^h\right)^{n+1}-\frac{\bs{u}_i^{n+1}-\bs{u}_i^n\circ\bs{X}_{h,i}^{n+1,n}}{\Delta t}\right\|_{0,\Omega_i^{n+1}}^2
    =\ \int_{\Omega_i^0}\left(\left(\frac{\partial\hat{\bs{u}}_i}{\partial t}\right)^{n+1}-\frac{\hat{\bs{u}}_i^{n+1}-\hat{\bs{u}}_i^n}{\Delta t}\right)^2J_{h,i}^{n+1}\,d\hat{\bs{x}}\\
    =&\ \frac{1}{(\Delta t)^2}\int_{\Omega_i^0}\left(\int_{t^n}^{t^{n+1}}(t-t^n)\frac{\partial^2\hat{\bs{u}}_i^t}{\partial t^2}\,dt\right)^2J_{h,i}^{n+1}\,d\hat{\bs{x}}\\
    \le&\ \frac{1}{(\Delta t)^2}\int_{\Omega_i^0}\left(\int_{t^n}^{t^{n+1}}(t-t^n)^2\,dt\right)\left(\int_{t^n}^{t^{n+1}}\left(\frac{\partial^2\hat{\bs{u}}_i^t}{\partial t^2}\right)^2\,dt\right)J_{h,i}^{n+1}\,d\hat{\bs{x}}\\
    =&\ \frac{\Delta t}{3}\int_{t^n}^{t^{n+1}}\left(\int_{\Omega_i^0}\left(\frac{\partial^2\hat{\bs{u}}_i^t}{\partial t^2}\right)^2J_{h,i}^{n+1}\,d\hat{\bs{x}}\right)\,dt
    =\ \frac{\Delta t}{3}\int_{t^n}^{t^{n+1}}\left(\int_{\Omega_i^t}\left(\left.\frac{\partial^2\bs{u}_i}{\partial t^2}\right|_{\hat{\bs{x}}}^h\right)^2\frac{J_{h,i}^{n+1}}{J_{h,i}^t}\,d\bs{x}\right)\,dt\\
    \le&\ C\Delta t\left\|\left.\frac{\partial^2\bs{u}_i}{\partial t^2}\right|_{\hat{\bs{x}}}^h\right\|_{L^2(t^n,t^{n+1};(L^2(\Omega_i^t))^d)}^2.
\end{align*}

We then estimate
$\left\|\dfrac{\eta_i^{n+1}-\eta_i^n\circ\bs{X}_{h,i}^{n+1,n}}{\Delta
t}\right\|_{0,\Omega_i^{n+1}}$. \textcolor{black}{Denote
$\hat{\eta}_i(t)=\eta_i(t)\circ\bs{X}_{h,i}^t$. Using the
Cauchy-Schwarz inequality, we have
\begin{align*}
    \left\|\frac{\eta_i^{n+1}-\eta_i^n\circ\bs{X}_{h,i}^{n+1,n}}{\Delta t}\right\|_{0,\Omega_i^{n+1}}^2
    =&\ \frac{1}{(\Delta t)^2}\int_{\Omega_i^0}|\hat{\eta}_i^{n+1}-\hat{\eta}_i^n|^2J_{h,i}^{n+1}\,d\hat{\bs{x}}\\
    =&\ \frac{1}{(\Delta t)^2}\int_{\Omega_i^0}\left|\int_{t^n}^{t^{n+1}}\frac{\partial\hat{\eta}_i}{\partial t}\,dt\right|^2J_{h,i}^{n+1}\,d\hat{\bs{x}}\\
    \le&\ \frac{1}{\Delta t}\int_{t^n}^{t^{n+1}}\left(\int_{\Omega_i^0}\left|\frac{\partial\hat{\eta}_i}{\partial t}\right|^2J_{h,i}^{n+1}\,d\hat{\bs{x}}\right)dt\\
    =&\ \frac{1}{\Delta t}\int_{t^n}^{t^{n+1}}\left(\int_{\Omega_i^t}\left|\left.\frac{\partial\eta_i}{\partial t}\right|_{\hat{\bs{x}}}^h\right|^2\frac{J_{h,i}^{n+1}}{J_{h,i}^t}\,d\bs{x}\right)dt\\
    \le&\ \frac{C}{\Delta t}\left\|\left.\frac{\partial\eta_i}{\partial t}\right|_{\hat{\bs{x}}}^h\right\|_{L^2(t^n,t^{n+1};(L^2(\Omega_i^t))^d)}^2.
\end{align*}}

We further estimate
$\textcolor{black}{\sum\limits_{i=1}^2}\left(\phi_i^{n+1},\nabla\cdot\left(\dfrac{\xi_i^{n+1}-\xi_i^n\circ\bs{X}_{h,i}^{n+1,n}}{\Delta
t}\right)\right)_{\Omega_i^{n+1}}$. Since
\begin{align*}
    &\ \textcolor{black}{\sum_{i=1}^2}\left(\phi_i^{n+1},\nabla\cdot\left(\frac{\xi_i^{n+1}-\xi_i^n\circ\bs{X}_{h,i}^{n+1,n}}{\Delta t}\right)\right)_{\Omega_i^{n+1}}\\
    =&\ \frac{1}{\Delta t}\textcolor{black}{\sum_{i=1}^2}\left[\left(\phi_i^{n+1},\nabla\cdot\xi_i^{n+1}\right)_{\Omega_i^{n+1}}-\left(\phi_i^{n+1},\nabla\cdot\big(\xi_i^n\circ\bs{X}_{h,i}^{n+1,n}\big)\right)_{\Omega_i^{n+1}}\right]\\
    =&\ \frac{1}{\Delta t}\textcolor{black}{\sum_{i=1}^2}\left[\left(\phi_i^{n+1},\nabla_{\bs{x}_i^{n+1}}\cdot\xi_i^{n+1}\right)_{\Omega_i^{n+1}}-\left(\phi_i^{n+1}\circ\bs{X}_{h,i}^{n,n+1},\nabla_{\bs{x}_i^n}\xi_i^n:\left(\frac{\partial\bs{x}_i^{n+1}}{\partial\bs{x}_i^n}\right)^{-\mathrm{T}}\frac{J_{h,i}^{n+1}}{J_{h,i}^n}\right)_{\Omega_i^n}\right]\\
    =&\ \frac{1}{\Delta t}\textcolor{black}{\sum_{i=1}^2}\left[\left(\phi_i^{n+1},\nabla_{\bs{x}_i^{n+1}}\cdot\xi_i^{n+1}\right)_{\Omega_i^{n+1}}-\left(\phi_i^{n+1}\circ\bs{X}_{h,i}^{n,n+1},\nabla_{\bs{x}_i^n}\cdot\xi_i^n\right)_{\Omega_i^n}\right]\\
    &+\frac{1}{\Delta t}\textcolor{black}{\sum_{i=1}^2}\left[\left(\phi_i^{n+1}\circ\bs{X}_{h,i}^{n,n+1},\nabla_{\bs{x}_i^n}\cdot\xi_i^n\right)_{\Omega_i^n}-\left(\phi_i^{n+1}\circ\bs{X}_{h,i}^{n,n+1},\nabla_{\bs{x}_i^n}\xi_i^n:\left(\frac{\partial\bs{x}_i^{n+1}}{\partial\bs{x}_i^n}\right)^{-\mathrm{T}}\frac{J_{h,i}^{n+1}}{J_{h,i}^n}\right)_{\Omega_i^n}\right]\\
    =&\ G_1+G_2.
\end{align*}
For Term $G_1$, \textcolor{black}{let
$\bar{\phi}^{n+1}=\displaystyle\frac{1}{|\Omega^n|}\sum_{i=1}^2\int_{\Omega_i^n}\phi_i^{n+1}\circ\bs{X}_{h,i}^{n,n+1}\,d\bs{x}$.
Since $(\phi_1^{n+1},\phi_2^{n+1})\in Q_h^{n+1}$ and the pulled-back
pressure pair with its global mean subtracted belongs to $Q_h^n$, we
obtain} \textcolor{black}{
\begin{equation*}
    G_1=\frac{1}{\Delta t}\sum_{i=1}^2\left(\phi_i^{n+1},\nabla_{\bs{x}_i^{n+1}}
    \cdot\xi_i^{n+1}\right)_{\Omega_i^{n+1}}-\frac{1}{\Delta t}\sum_{i=1}^2
    \left(\phi_i^{n+1}\circ\bs{X}_{h,i}^{n,n+1}-\bar{\phi}^{n+1},\nabla_{\bs{x}_i^{n}}\cdot\xi_i^n\right)_{\Omega_i^n}-\frac{\bar{\phi}^{n+1}}{\Delta
    t}\sum_{i=1}^2\int_{\Omega_i^n}\nabla_{\bs{x}_i^{n}}\cdot\xi_i^n\,d\bs{x},
\end{equation*}}
\textcolor{black}{where the first two terms vanish by the global
divergence-free condition \eqref{eq:full-erreq2} applied at
$t^{n+1}$ and $t^n$, respectively, and the last term vanishes as
well by the divergence theorem, owing to the homogeneous velocity
boundary condition and velocity continuity across the interface.
Hence, we have $G_1=0$.} Now we estimate $G_2$ as follows.
\begin{align*}
G_2
    =&\ -\frac{1}{\Delta t}\textcolor{black}{\sum_{i=1}^2}\int_{t^n}^{t^{n+1}}\frac{d}{dt}\left(\phi_i^{n+1}\circ\bs{X}_{h,i}^{n,n+1},\nabla_{\bs{x}_i^n}\xi_i^n:\left(\frac{\partial\bs{x}_i^t}{\partial\bs{x}_i^n}\right)^{-\mathrm{T}}\frac{J_{h,i}^t}{J_{h,i}^n}\right)_{\Omega_i^n}\,dt\\
    =&\ -\frac{1}{\Delta t}\textcolor{black}{\sum_{i=1}^2}\int_{t^n}^{t^{n+1}}\left(\phi_i^{n+1}\circ\bs{X}_{h,i}^{n,n+1},\nabla_{\bs{x}_i^n}\xi_i^n:\frac{d}{dt}\left(\left(\frac{\partial\bs{x}_i^t}{\partial\bs{x}_i^n}\right)^{-\mathrm{T}}\frac{J_{h,i}^t}{J_{h,i}^n}\right)\right)_{\Omega_i^n}\,dt\\
    =&\ -\frac{1}{\Delta t}\textcolor{black}{\sum_{i=1}^2}\int_{t^n}^{t^{n+1}}\left(\phi_i^{n+1}\circ\bs{X}_{h,i}^{n,n+1},\nabla_{\bs{x}_i^n}\xi_i^n:\left(-\left(\frac{\partial\bs{x}_i^t}{\partial\bs{x}_i^n}\right)^{-\mathrm{T}}\left(\frac{\partial\bs{w}_{h,i}}{\partial\bs{x}_i^n}\right)^\mathrm{T}\left(\frac{\partial\bs{x}_i^t}{\partial\bs{x}_i^n}\right)^{-\mathrm{T}}\frac{J_{h,i}^t}{J_{h,i}^n}\right.\right.\\
    &\qquad\qquad\qquad\left.\left.+\left(\frac{\partial\bs{x}_i^t}{\partial\bs{x}_i^n}\right)^{-\mathrm{T}}\frac{J_{h,i}^t(\nabla_{\bs{x}_i^t}\cdot\bs{w}_{h,i})}{J_{h,i}^n}\right)\right)_{\Omega_i^n}\,dt\qquad\qquad\qquad\qquad\qquad \text{(using \eqref{identity-F})}\\
    \le&\ \frac{C}{\Delta t}\textcolor{black}{\sum_{i=1}^2}\int_{t^n}^{t^{n+1}}\left\|\phi_i^{n+1}\circ\bs{X}_{h,i}^{n,n+1}\right\|_{0,\Omega_i^n}\left\|\nabla_{\bs{x}_i^n}\xi_i^n\right\|_{0,\Omega_i^n}\,dt\le C\textcolor{black}{\sum_{i=1}^2}\left\|\phi_i^{n+1}\right\|_{0,\Omega_i^{n+1}}\left\|\nabla\xi_i^n\right\|_{0,\Omega_i^n}\\
    \le&\ C\textcolor{black}{\sum_{i=1}^2}\left\|\nabla\xi_i^n\right\|_{0,\Omega_i^n}^2+\varepsilon_\phi\textcolor{black}{\sum_{i=1}^2}\left\|\phi_i^{n+1}\right\|_{0,\Omega_i^{n+1}}^2.
\end{align*}
Hence, it follows that
\begin{equation*}
    \textcolor{black}{\sum_{i=1}^2}\left(\phi_i^{n+1},\nabla\cdot\left(\frac{\xi_i^{n+1}-\xi_i^n\circ\bs{X}_{h,i}^{n+1,n}}{\Delta t}\right)\right)_{\Omega_i^{n+1}}\le C\textcolor{black}{\sum_{i=1}^2}\left\|\nabla\xi_i^n\right\|_{0,\Omega_i^n}^2+\varepsilon_\phi\textcolor{black}{\sum_{i=1}^2}\left\|\phi_i^{n+1}\right\|_{0,\Omega_i^{n+1}}^2.
\end{equation*}
By combining all the above estimates for \eqref{eq-Full} and
choosing a sufficiently small $\varepsilon$, we obtain
\begin{equation}
    \label{eq:full-errest1}
    \begin{aligned}
        &\ \sum_{i=1}^2\left[\left\|\frac{\xi_i^{n+1}-\xi_i^n\circ\bs{X}_{h,i}^{n+1,n}}{\Delta t}\right\|_{0,\Omega_i^{n+1}}^2+\frac{\left\|\nabla\xi_i^{n+1}\right\|_{0,\Omega_i^{n+1}}^2-\left\|\nabla\xi_i^n\right\|_{0,\Omega_i^n}^2}{\Delta t}\right]\\
        \le&\ C\sum_{i=1}^2\left[\left\|\nabla\xi_i^n\right\|_{0,\Omega_i^n}^2+\left\|\nabla\xi_i^{n+1}\right\|_{0,\Omega_i^{n+1}}^2+\left\|\eta_i^{n+1}\right\|_{0,\Omega_i^{n+1}}^2+\textcolor{black}{\frac{1}{\Delta t}\left\|\left.\frac{\partial\eta_i}{\partial t}\right|_{\hat{\bs{x}}}^h\right\|_{L^2(t^n,t^{n+1};(L^2(\Omega_i^t))^d)}^2}\right.\\
        &\qquad\quad\ +\left.\Delta t\left\|\left.\frac{\partial^2\bs{u}_i}{\partial t^2}\right|_{\hat{\bs{x}}}^h\right\|_{L^2(t^n,t^{n+1};(L^2(\Omega_i^t))^d)}^2\right]+\varepsilon_\phi\sum_{i=1}^2\left\|\phi_i^{n+1}\right\|_{0,\Omega_i^{n+1}}^2.
    \end{aligned}
\end{equation}

To estimate $\left\|\phi_i^{n+1}\right\|_{0,\Omega_i^{n+1}}$, we use the discrete inf-sup condition and the error equation \eqref{eq:full-erreq1}, resulting in
\begin{equation}
    \label{eq:full-errest2}
    \begin{aligned}
        &\ \gamma\left\|\big(\phi_1^{n+1},\phi_2^{n+1}\big)\right\|_0\le\sup_{\big(\bs{v}_{h,1}^{n+1},\bs{v}_{h,2}^{n+1}\big)\in\bs{U}_h^{n+1}}\frac{\displaystyle\sum_{i=1}^2\left(\phi_i^{n+1},\nabla\cdot\bs{v}_{h,i}^{n+1}\right)_{\Omega_i^{n+1}}}{\left\|\big(\bs{v}_{h,1}^{n+1},\bs{v}_{h,2}^{n+1}\big)\right\|_1}\\
        =&\ \sup_{\big(\bs{v}_{h,1}^{n+1},\bs{v}_{h,2}^{n+1}\big)\in\bs{U}_h^{n+1}}\frac{1}{\left\|\big(\bs{v}_{h,1}^{n+1},\bs{v}_{h,2}^{n+1}\big)\right\|_1}\sum_{i=1}^2\left[\left(\left(\left.\frac{\partial\bs{u}_i}{\partial t}\right|_{\hat{\bs{x}}}^h\right)^{n+1}-\frac{\bs{u}_i^{n+1}-\bs{u}_i^n\circ\bs{X}_{h,i}^{n+1,n}}{\Delta t},\bs{v}_{h,i}^{n+1}\right)_{\Omega_i^{n+1}}\right.\\
        &\qquad\qquad\quad+\left(\frac{\eta_i^{n+1}-\eta_i^n\circ\bs{X}_{h,i}^{n+1,n}}{\Delta t},\bs{v}_{h,i}^{n+1}\right)_{\Omega_i^{n+1}}+\left(\frac{\xi_i^{n+1}-\xi_i^n\circ\bs{X}_{h,i}^{n+1,n}}{\Delta t},\bs{v}_{h,i}^{n+1}\right)_{\Omega_i^{n+1}}\\
        &\qquad\qquad\quad+\left(\mu_i\nabla\xi_i^{n+1},\nabla\bs{v}_{h,i}^{n+1}\right)_{\Omega_i^{n+1}}-\left.\left(\big(\bs{w}_{h,i}^{n+1}\cdot\nabla\big)\xi_i^{n+1},\bs{v}_{h,i}^{n+1}\right)_{\Omega_i^{n+1}}-\kappa\left(\eta_i^{n+1},\bs{v}_{h,i}^{n+1}\right)_{\Omega_i^{n+1}}\rule{0ex}{5ex}\right]\\
        \le&\ \sum_{i=1}^2\left[\left\|\left(\left.\frac{\partial\bs{u}_i}{\partial t}\right|_{\hat{\bs{x}}}^h\right)^{n+1}-\frac{\bs{u}_i^{n+1}-\bs{u}_i^n\circ\bs{X}_{h,i}^{n+1,n}}{\Delta t}\right\|_{0,\Omega_i^{n+1}}+\left\|\frac{\eta_i^{n+1}-\eta_i^n\circ\bs{X}_{h,i}^{n+1,n}}{\Delta t}\right\|_{0,\Omega_i^{n+1}}\right.\\
        &\qquad\ +\left.\left\|\frac{\xi_i^{n+1}-\xi_i^n\circ\bs{X}_{h,i}^{n+1,n}}{\Delta t}\right\|_{0,\Omega_i^{n+1}}\right]+C\sum_{i=1}^2\left[\left\|\nabla\xi_i^{n+1}\right\|_{0,\Omega_i^{n+1}}+\left\|\eta_i^{n+1}\right\|_{0,\Omega_i^{n+1}}\right]\\
        \le&\ \sum_{i=1}^2\left\|\frac{\xi_i^{n+1}-\xi_i^n\circ\bs{X}_{h,i}^{n+1,n}}{\Delta t}\right\|_{0,\Omega_i^{n+1}}+C\sum_{i=1}^2\left[\rule{0ex}{5ex}\left\|\nabla\xi_i^{n+1}\right\|_{0,\Omega_i^{n+1}}+\left\|\eta_i^{n+1}\right\|_{0,\Omega_i^{n+1}}\right.\\
        &\qquad\ +\left.\textcolor{black}{\frac{1}{(\Delta t)^\frac{1}{2}}\left\|\left.\frac{\partial\eta_i}{\partial t}\right|_{\hat{\bs{x}}}^h\right\|_{L^2(t^n,t^{n+1};(L^2(\Omega_i^t))^d)}}+(\Delta t)^\frac{1}{2}\left\|\left.\frac{\partial^2\bs{u}_i}{\partial t^2}\right|_{\hat{\bs{x}}}^h\right\|_{L^2(t^n,t^{n+1};(L^2(\Omega_i^t))^d)}\right].
    \end{aligned}
\end{equation}
Substituting \eqref{eq:full-errest2} into \eqref{eq:full-errest1} and choosing a sufficiently small $\varepsilon_\phi$, we obtain
\begin{align*}
    &\ \sum_{i=1}^2\left[\left\|\frac{\xi_i^{n+1}-\xi_i^n\circ\bs{X}_{h,i}^{n+1,n}}{\Delta t}\right\|_{0,\Omega_i^{n+1}}^2+\frac{\left\|\nabla\xi_i^{n+1}\right\|_{0,\Omega_i^{n+1}}^2-\left\|\nabla\xi_i^n\right\|_{0,\Omega_i^n}^2}{\Delta t}\right]\\
    \le&\ C\sum_{i=1}^2\left[\rule{0ex}{5ex}\left\|\nabla\xi_i^n\right\|_{0,\Omega_i^n}^2+\left\|\nabla\xi_i^{n+1}\right\|_{0,\Omega_i^{n+1}}^2+\left\|\eta_i^{n+1}\right\|_{0,\Omega_i^{n+1}}^2\right.\\
    &\qquad\quad+\left.\textcolor{black}{\frac{1}{\Delta t}\left\|\left.\frac{\partial\eta_i}{\partial t}\right|_{\hat{\bs{x}}}^h\right\|_{L^2(t^n,t^{n+1};(L^2(\Omega_i^t))^d)}^2}+\Delta t\left\|\left.\frac{\partial^2\bs{u}_i}{\partial t^2}\right|_{\hat{\bs{x}}}^h\right\|_{L^2(t^n,t^{n+1};(L^2(\Omega_i^t))^d)}^2\right].
\end{align*}
Multiplying both sides of the above inequality by $\Delta t$, and summing over $n=0,\dots,N-1$, together with the initial condition $\bs{u}_{h,i}^0=\tilde{\bs{u}}_i^0$, we obtain
\begin{equation}
    \label{eq:full-errest3}
    \begin{aligned}
        &\ \sum_{i=1}^2\left[\Delta t\sum_{n=0}^{N-1}\left\|\frac{\xi_i^{n+1}-\xi_i^n\circ\bs{X}_{h,i}^{n+1,n}}{\Delta t}\right\|_{0,\Omega_i^{n+1}}^2+\left\|\nabla\xi_i^N\right\|_{0,\Omega_i^N}^2\right]\\
        \le&\ C\sum_{i=1}^2\left[\Delta t\sum_{n=0}^N\left\|\nabla\xi_i^n\right\|_{0,\Omega_i^n}^2+\Delta t\sum_{n=1}^N\left\|\eta_i^n\right\|_{0,\Omega_i^n}^2+\textcolor{black}{\left\|\left.\frac{\partial\eta_i}{\partial t}\right|_{\hat{\bs{x}}}^h\right\|_{L^2(0,T;(L^2(\Omega_i^t))^d)}^2}+(\Delta t)^2\left\|\left.\frac{\partial^2\bs{u}_i}{\partial t^2}\right|_{\hat{\bs{x}}}^h\right\|_{L^2(0,T;(L^2(\Omega_i^t))^d)}^2\right].
    \end{aligned}
\end{equation}
Moreover, \textcolor{black}{squaring both sides of
\eqref{eq:full-errest2} and then multiplying} by $\Delta t$, summing
over $n=0,\dots,N-1$, we obtain
\begin{equation}
    \label{eq:full-errest4}
    \begin{aligned}
        \sum_{i=1}^2\Delta t\sum_{n=1}^N\left\|\phi_i^n\right\|_{0,\Omega_i^n}^2\le&\ C\sum_{i=1}^2\left[\Delta t\sum_{n=0}^{N-1}\left\|\frac{\xi_i^{n+1}-\xi_i^n\circ\bs{X}_{h,i}^{n+1,n}}{\Delta t}\right\|_{0,\Omega_i^{n+1}}^2+\Delta t\sum_{n=1}^N\left\|\nabla\xi_i^n\right\|_{0,\Omega_i^n}^2\right.\\
        &\qquad+\left.\Delta t\sum_{n=1}^N\left\|\eta_i^n\right\|_{0,\Omega_i^n}^2+\textcolor{black}{\left\|\left.\frac{\partial\eta_i}{\partial t}\right|_{\hat{\bs{x}}}^h\right\|_{L^2(0,T;(L^2(\Omega_i^t))^d)}^2}+(\Delta t)^2\left\|\left.\frac{\partial^2\bs{u}_i}{\partial t^2}\right|_{\hat{\bs{x}}}^h\right\|_{L^2(0,T;(L^2(\Omega_i^t))^d)}^2\right].
    \end{aligned}
\end{equation}
Combining \eqref{eq:full-errest3} and \eqref{eq:full-errest4},
applying the Poincar\'e inequality and the discrete Gr\"onwall's
inequality, we obtain
\begin{equation}
    \label{eq:full-errest5}
    \begin{aligned}
        &\ \sum_{i=1}^2\left[\Delta t\sum_{n=0}^{N-1}\left\|\frac{\xi_i^{n+1}-\xi_i^n\circ\bs{X}_{h,i}^{n+1,n}}{\Delta t}\right\|_{0,\Omega_i^{n+1}}^2+\left\|\xi_i^N\right\|_{1,\Omega_i^N}^2+\Delta t\sum_{n=1}^N\left\|\phi_i^n\right\|_{0,\Omega_i^n}^2\right]\\
        \le&\ C\sum_{i=1}^2\left[\Delta t\sum_{n=1}^N\left\|\eta_i^n\right\|_{0,\Omega_i^n}^2+\textcolor{black}{\left\|\left.\frac{\partial\eta_i}{\partial t}\right|_{\hat{\bs{x}}}^h\right\|_{L^2(0,T;(L^2(\Omega_i^t))^d)}^2}+(\Delta t)^2\left\|\left.\frac{\partial^2\bs{u}_i}{\partial t^2}\right|_{\hat{\bs{x}}}^h\right\|_{L^2(0,T;(L^2(\Omega_i^t))^d)}^2\right].
    \end{aligned}
\end{equation}
Finally, by using the triangle inequality, Lemmas \ref{lem:semi1}
and \ref{lem:semi4}, together with \eqref{eq:full-errest1} and
\eqref{eq:full-errest5}, we conclude that
\begin{align*}
    &\ \sum_{i=1}^2\left[\left\|\bs{u}_i^N-\bs{u}_{h,i}^N\right\|_{0,\Omega_i^N}+\left(\Delta t\sum_{n=1}^N\left\|\partial_t(\bs{u}_i^n-\bs{u}_{h,i}^n)\right\|_{0,\Omega_i^n}^2\right)^\frac{1}{2}\right]\\
    =&\ \sum_{i=1}^2\left[\left\|\eta_i^N+\xi_i^N\right\|_{0,\Omega_i^N}+\left(\Delta t\sum_{n=1}^N\left\|\partial_t(\eta_i^n+\xi_i^n)\right\|_{0,\Omega_i^n}^2\right)^\frac{1}{2}\right]\\
    \le&\ C\sum_{i=1}^2\left[\left\|\eta_i^N\right\|_{0,\Omega_i^N}+\left(\Delta t\sum_{n=1}^N\left\|\eta_i^n\right\|_{0,\Omega_i^n}^2\right)^\frac{1}{2}+\textcolor{black}{\left\|\left.\frac{\partial\eta_i}{\partial t}\right|_{\hat{\bs{x}}}^h\right\|_{L^2(0,T;(L^2(\Omega_i^t))^d)}}+\Delta t\left\|\left.\frac{\partial^2\bs{u}_i}{\partial t^2}\right|_{\hat{\bs{x}}}^h\right\|_{L^2(0,T;(L^2(\Omega_i^t))^d)}\right]\\
    \le&\ Ch^{k+1}\sum_{i=1}^2\left[
    \textcolor{black}{\|\bs{w}_i\|_{L^2(0,T;(H^{k+1}(\Omega_i^t))^d)}}
    +\|\bs{u}_i\|_{L^\infty(0,T;(H^{k+1}(\Omega_i^t))^d)}
    +\left\|\left.\frac{\partial\bs{u}_i}{\partial t}\right|_{\hat{\bs{x}}}\right\|_{L^2(0,T;(H^{k+1}(\Omega_i^t))^d)}\right.\\
    &+\left.\|p_i\|_{L^\infty(0,T;H^k(\Omega_i^t))}
    +\left\|\left.\frac{\partial p_i}{\partial t}\right|_{\hat{\bs{x}}}\right\|_{L^2(0,T;H^k(\Omega_i^t))}\right]+C\Delta t\sum_{i=1}^2\left\|\left.\frac{\partial^2\bs{u}_i}{\partial t^2}\right|_{\hat{\bs{x}}}^h\right\|_{L^2(0,T;(L^2(\Omega_i^t))^d)}\\
    \le&\ C(h^{k+1}+\Delta t)\sum_{i=1}^2\left[
    \textcolor{black}{\|\bs{w}_i\|_{L^2(0,T;(H^{k+1}(\Omega_i^t))^d)}}
    +\|\bs{u}_i\|_{L^\infty(0,T;(H^{k+1}(\Omega_i^t))^d)}
    +\left\|\left.\frac{\partial\bs{u}_i}{\partial t}\right|_{\hat{\bs{x}}}\right\|_{L^2(0,T;(H^{k+1}(\Omega_i^t))^d)}\right.\\
    & +\|p_i\|_{L^\infty(0,T;H^k(\Omega_i^t))}
    +\left\|\left.\frac{\partial p_i}{\partial t}\right|_{\hat{\bs{x}}}\right\|_{L^2(0,T;H^k(\Omega_i^t))}
    +\left.\left\|\left.\frac{\partial^2\bs{u}_i}{\partial t^2}\right|_{\hat{\bs{x}}}\right\|_{L^2(0,T;(L^2(\Omega_i^t))^d)}\right],
\end{align*}
and
\begin{align*}
    &\ \sum_{i=1}^2\left[\left\|\bs{u}_i^N-\bs{u}_{h,i}^N\right\|_{1,\Omega_i^N}+\left(\Delta t\sum_{n=1}^N\left\|p_i^n-p_{h,i}^n\right\|_{0,\Omega_i^n}^2\right)^\frac{1}{2}\right]\\
    =&\ \sum_{i=1}^2\left[\left\|\eta_i^N+\xi_i^N\right\|_{1,\Omega_i^N}+\left(\Delta t\sum_{n=1}^N\left\|\delta_i^n+\phi_i^n\right\|_{0,\Omega_i^n}^2\right)^\frac{1}{2}\right]\\
    \le&\ C\sum_{i=1}^2\left[\left\|\eta_i^N\right\|_{1,\Omega_i^N}+\left(\Delta t\sum_{n=1}^N\left\|\eta_i^n\right\|_{0,\Omega_i^n}^2\right)^\frac{1}{2}+\textcolor{black}{\left\|\left.\frac{\partial\eta_i}{\partial t}\right|_{\hat{\bs{x}}}^h\right\|_{L^2(0,T;(L^2(\Omega_i^t))^d)}}\right.\\
    &\qquad\quad+\left.\Delta t\left\|\left.\frac{\partial^2\bs{u}_i}{\partial t^2}\right|_{\hat{\bs{x}}}^h\right\|_{L^2(0,T;(L^2(\Omega_i^t))^d)}+\left(\Delta t\sum_{n=1}^N\left\|\delta_i^n\right\|_{0,\Omega_i^n}^2\right)^\frac{1}{2}\right]\\
    \le&\ Ch^k\sum_{i=1}^2\left[
    \textcolor{black}{\|\bs{w}_i\|_{L^2(0,T;(H^{k+1}(\Omega_i^t))^d)}}
    +\|\bs{u}_i\|_{L^\infty(0,T;(H^{k+1}(\Omega_i^t))^d)}
    +\left\|\left.\frac{\partial\bs{u}_i}{\partial t}\right|_{\hat{\bs{x}}}\right\|_{L^2(0,T;(H^{k+1}(\Omega_i^t))^d)}\right.\\
    &+\left.\|p_i\|_{L^\infty(0,T;H^k(\Omega_i^t))}
    +\left\|\left.\frac{\partial p_i}{\partial t}\right|_{\hat{\bs{x}}}\right\|_{L^2(0,T;H^k(\Omega_i^t))}\right]+C\Delta t\sum_{i=1}^2\left\|\left.\frac{\partial^2\bs{u}_i}{\partial t^2}\right|_{\hat{\bs{x}}}^h\right\|_{L^2(0,T;(L^2(\Omega_i^t))^d)}\\
    \le&\ C(h^k+\Delta t)\sum_{i=1}^2\left[
    \textcolor{black}{\|\bs{w}_i\|_{L^2(0,T;(H^{k+1}(\Omega_i^t))^d)}}
    +\|\bs{u}_i\|_{L^\infty(0,T;(H^{k+1}(\Omega_i^t))^d)}
    +\left\|\left.\frac{\partial\bs{u}_i}{\partial t}\right|_{\hat{\bs{x}}}\right\|_{L^2(0,T;(H^{k+1}(\Omega_i^t))^d)}\right.\\
    &\qquad\qquad\qquad\quad\ +\|p_i\|_{L^\infty(0,T;H^k(\Omega_i^t))}
    +\left\|\left.\frac{\partial p_i}{\partial t}\right|_{\hat{\bs{x}}}\right\|_{L^2(0,T;H^k(\Omega_i^t))}
    +\left.\left\|\left.\frac{\partial^2\bs{u}_i}{\partial t^2}\right|_{\hat{\bs{x}}}\right\|_{L^2(0,T;(L^2(\Omega_i^t))^d)}\right],
\end{align*}
which yields the error estimates \eqref{full-errest1}-\eqref{full-errest2}.
\end{proof}

\section{Numerical experiments}\label{sec:experiment}

\subsection{Prescribed interface motion}

We first consider a two-dimensional numerical example with a
prescribed interface motion. Let $\Omega=[0,1]\times[0,1]$ and
$T=1$. The interface $\Gamma^t$ is initially a square
$[0.4,0.6]\times[0.4,0.6]$ and translates with constant velocity
$(w_1,w_2)^\mathrm{T}=(0.1,0.05)^\mathrm{T}$. Define
\begin{equation*}
    F(x,y,t):=(x-(0.4+w_1t))(x-(0.6+w_1t))(y-(0.4+w_2t))(y-(0.6+w_2t))\textcolor{black}{e^{2\sin(t)}},
\end{equation*}
which satisfies $F(x,y,t)=0$ for any $(x,y)\in\Gamma^t$ and
$t\in[0,T]$. Moreover, we assume that $\mu=\mu_i(\bs{x})$ is
piecewise constant in $\Omega_i^t$ ($i=1,2$). Under this setting,
the exact solution $\bs{u}=(u,v)^\mathrm{T}$ and pressure $p$ for
problem \eqref{prob1}-\eqref{prob11} can be constructed as follows:
\begin{align*}
    u&=F(x,y,t)F_y(x,y,t)/\mu,\\
    v&=-F(x,y,t)F_x(x,y,t)/\mu,\\
    p&=\textcolor{black}{\cos(\pi x)\cos(\pi y)(1-e^{-t}),}
\end{align*}
where $F_x$ and $F_y$ denote the partial derivatives of $F$ with
respect to $x$ and $y$, respectively. \textcolor{black}{The
nonhomogeneous Dirichlet data on the outer boundary are prescribed
by the trace of this exact velocity field.} This constructed
solution \textcolor{black}{has zero mean pressure over the fixed
domain $\Omega$ and} satisfies the regularity assumptions
\eqref{reg1}-\eqref{reg5}, and it also holds that
$\nabla\cdot\bs{u}=0$ in $\Omega_i^t$ ($i=1,2$), while the velocity
is continuous across the interface $\Gamma^t$. In this example, we
take $\mu_1=1$, $\mu_2=1000$, with the terms $\bs{f}_1,\bs{f}_2$ and
$\bs{g}$ chosen appropriately so that the exact solution satisfies
the two-dimensional Stokes moving interface problem
\eqref{prob1}-\eqref{prob11}. Subject to the prescribed interface
motion, we solve the discrete ALE mapping $\bs{X}_{h,i}^t$ on
$\Omega_i^0$ \textcolor{black}{by harmonic extension} to generate
the evolving meshes $\mathcal{T}_{h,i}^t$ ($i=1,2$) for $t\in[0,T]$.
\textcolor{black}{At each time step, the interface nodes are updated
according to
\begin{equation*}
    \bs{x}_{\Gamma,h}^{n+1}=\bs{x}_{\Gamma,h}^{n}+\Delta t(w_1,w_2)^\mathrm{T},
\end{equation*}
independently of the fluid velocity. The interior mesh coordinates are obtained by solving the discrete harmonic extension problem on the reference mesh, with the updated interface positions and the fixed outer boundary imposed as Dirichlet conditions. The discrete mesh velocity is computed from the nodal displacements as
\begin{equation*}
    \bs{w}_{h,i}^{n+1}(\bs{x})=\left[\frac{\bs{X}_{h,i}^{n+1}-\bs{X}_{h,i}^{n}}{\Delta t}\right]\circ(\bs{X}_{h,i}^{n+1})^{-1}(\bs{x}),\qquad\bs{x}\in\Omega_i^{n+1},\qquad i=1,2.
\end{equation*}}
Figure
\ref{fig:EXP1_domain} illustrates the initial and terminal domains
together with the corresponding meshes generated by the discrete ALE
mapping.

\begin{figure}[htbp]
    \centering
    \includegraphics[width=0.495\textwidth]{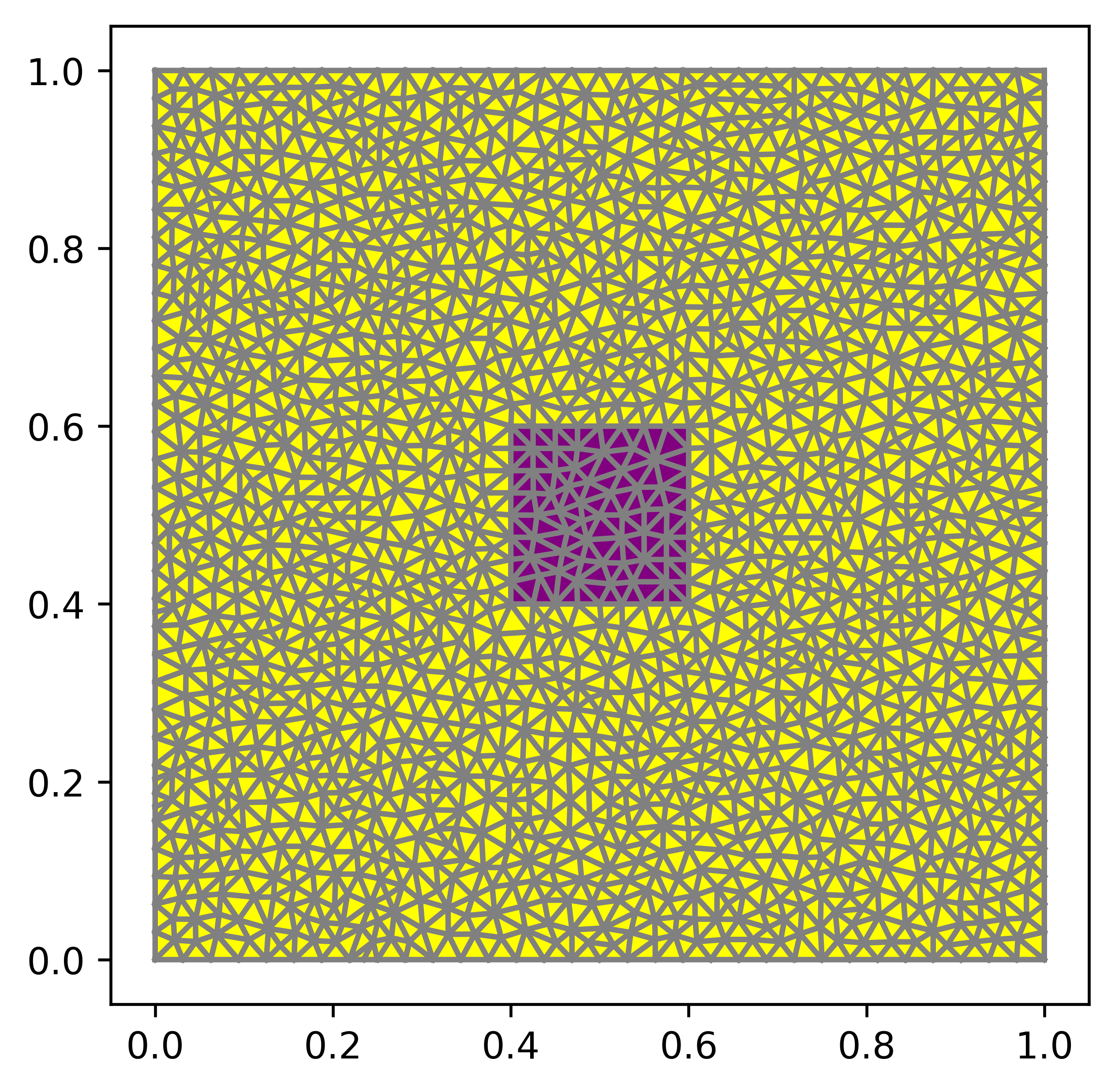}
    \includegraphics[width=0.495\textwidth]{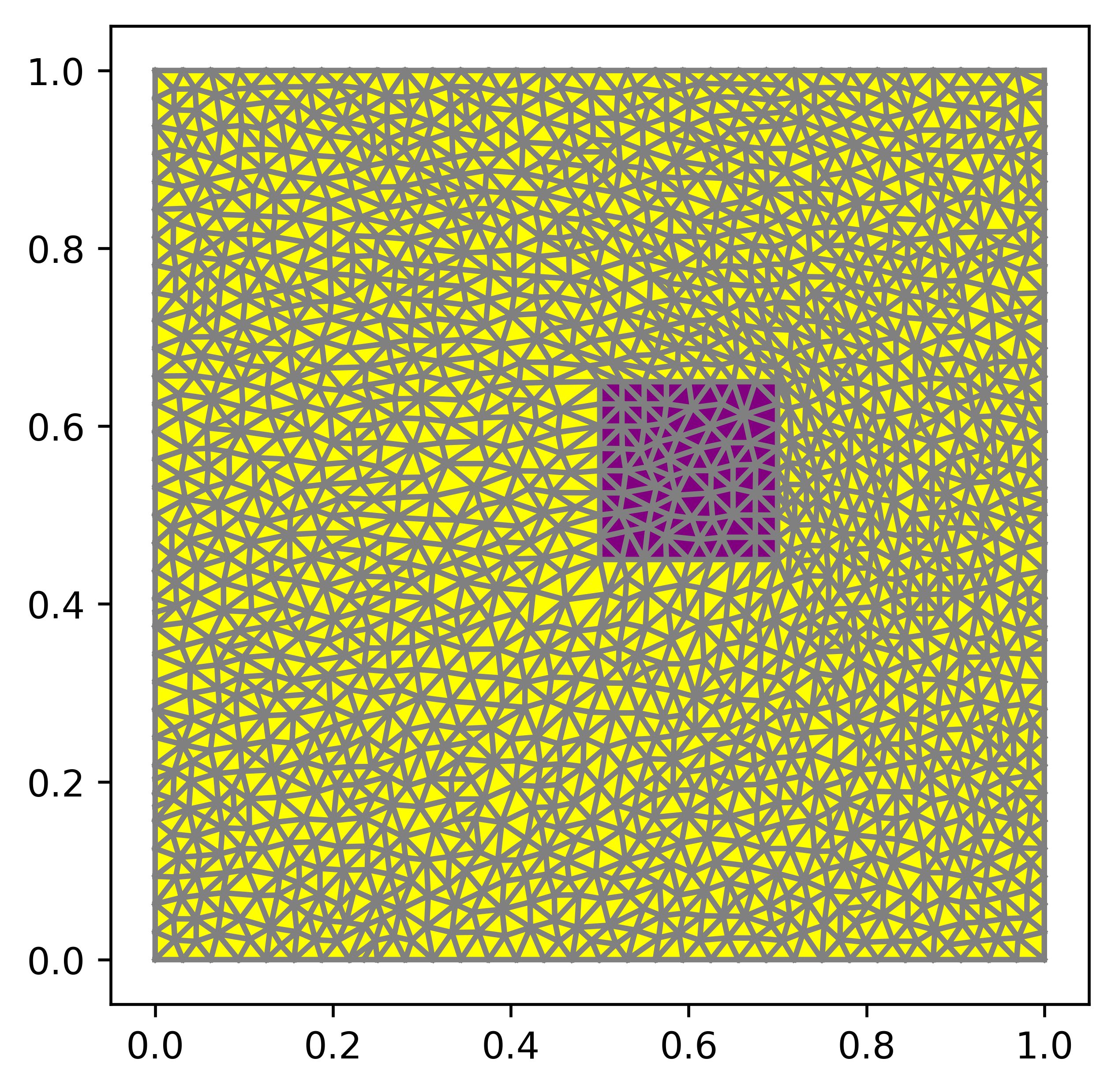}
    \caption{The initial (left) and terminal (right) domains and meshes with the mesh size $h=1/32$.}
    \label{fig:EXP1_domain}
\end{figure}

We employ the fully discrete ALE finite element scheme
\eqref{full1}-\eqref{full2} with the Taylor-Hood element of order
$k=2$, i.e., the $P^2$-$P^1$ element pair, as a stable
discretization for the Stokes problem to approximate
$((\bs{u}_1,\bs{u}_2),(p_1,p_2))$. \textcolor{black}{All numerical
computations herein and in what follows are implemented using legacy
FEniCS, with $mshr$ used for the mesh generation component.
Furthermore, we adopt the piecewise-linear triangular geometry, the
direct LU solver provided through the PETSc backend, and the default
quadrature scheme of FEniCS Form Compiler (FFC) with its
automatically selected degree. The velocity errors are evaluated in
the global $L^2$ and $H^1$ norms, and the pressure errors are
evaluated in the global $L^2$ norm, all over the computational
domain.} To investigate the convergence behavior, the mesh size is
successively refined from $h=1/4$ to $h=1/32$, while the time step
size \textcolor{black}{is chosen as $\Delta t=4h^3$}.
\textcolor{black}{This choice is intended to keep the temporal error
sufficiently small along the refinement path so that the spatial
convergence behavior can be observed.} The numerical convergence
results are reported in Table \ref{tab:EXP1_rate1}, and the
convergence histories of all primary variables are displayed in
Figure \ref{fig:EXP1_rate1} via a log-log plot. From these results,
we observe that the velocity error in $H^1$ norm and the pressure
error in $L^2$ norm both exhibit second-order convergence, while the
velocity error in $L^2$ norm achieves third-order convergence.
\textcolor{black}{These combined space-time refinement results are
consistent with the expected theoretical convergence orders in space
as well as the first-order temporal accuracy under the choice
$\Delta t\propto h^3$. }

\begin{table}[htbp]
    \centering
    \caption{Convergence performance of the $P^2$-$P^1$ element for the case of prescribed interface motion.}
    \label{tab:EXP1_rate1}
    \begin{tabular}{ccccccc}
        \toprule
        $h$ & $\|u-u_h\|_0$ & Order & $\|u-u_h\|_1$ & Order & $\|p-p_h\|_0$ & Order \\
        \midrule
        \textcolor{black}{$1/4$} & \textcolor{black}{$4.580832\times10^{-3}$} & \textcolor{black}{$-$} & \textcolor{black}{$1.557252\times10^{-1}$} & \textcolor{black}{$-$} & \textcolor{black}{$3.981102\times10^{-2}$} & \textcolor{black}{$-$} \\
        \textcolor{black}{$1/8$} & \textcolor{black}{$7.172154\times10^{-4}$} & \textcolor{black}{$2.68$} & \textcolor{black}{$4.902675\times10^{-2}$} & \textcolor{black}{$1.67$} & \textcolor{black}{$8.608226\times10^{-3}$} & \textcolor{black}{$2.21$} \\
        \textcolor{black}{$1/16$} & \textcolor{black}{$8.402116\times10^{-5}$} & \textcolor{black}{$3.09$} & \textcolor{black}{$1.119348\times10^{-2}$} & \textcolor{black}{$2.13$} & \textcolor{black}{$1.764411\times10^{-3}$} & \textcolor{black}{$2.29$} \\
        \textcolor{black}{$1/32$} & \textcolor{black}{$1.082695\times10^{-5}$} & \textcolor{black}{$2.96$} & \textcolor{black}{$2.863505\times10^{-3}$} & \textcolor{black}{$1.97$} & \textcolor{black}{$3.918973\times10^{-4}$} & \textcolor{black}{$2.17$} \\
        \bottomrule
    \end{tabular}
\end{table}

\begin{figure}[htbp]
    \centering
    \includegraphics[width=0.75\textwidth]{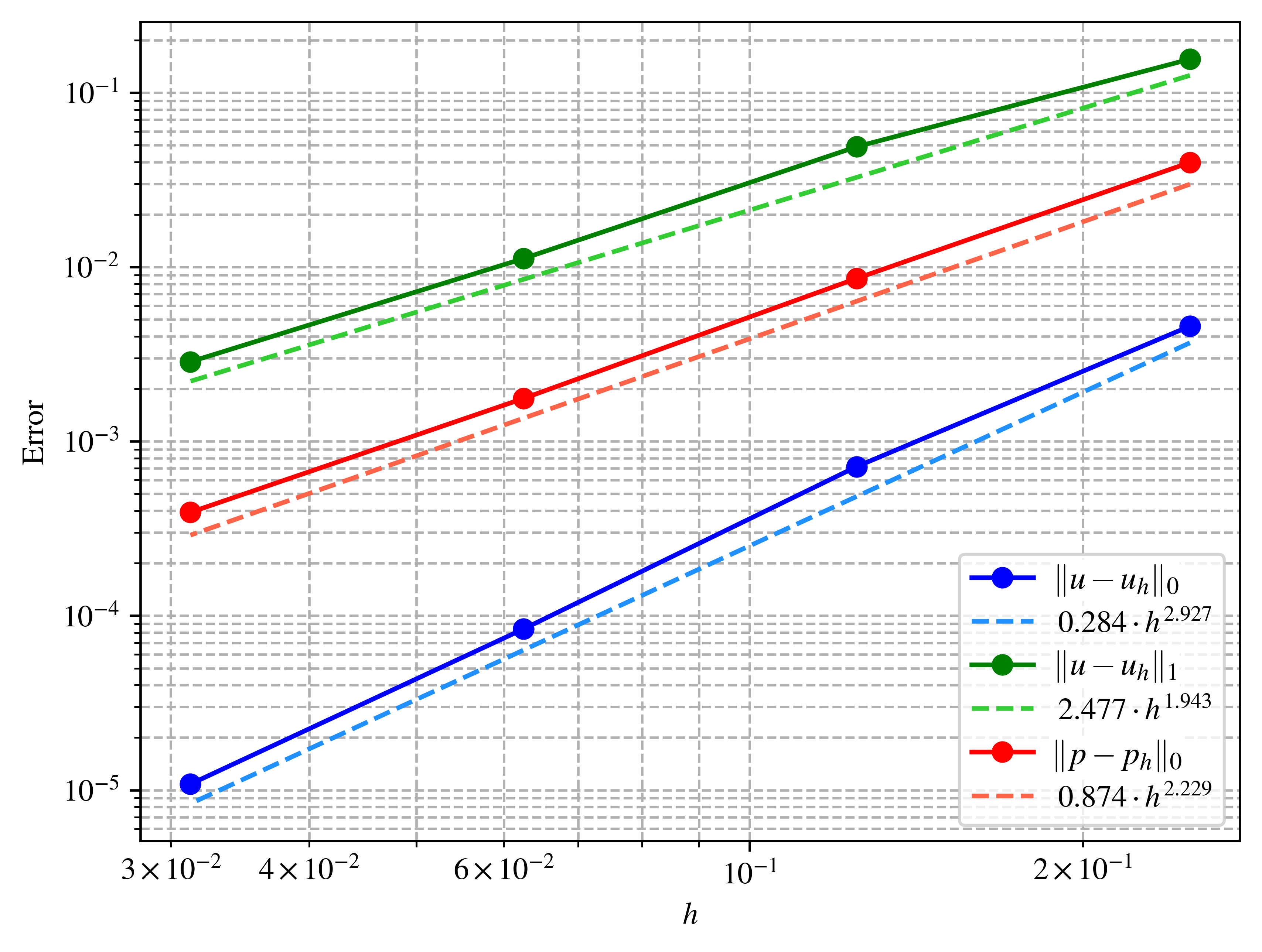}
    \caption{\textcolor{black}{Convergence history of the $P^2$-$P^1$ element for the case of prescribed interface motion.}}
    \label{fig:EXP1_rate1}
\end{figure}

\textcolor{black}{On the other hand, we also perform a temporal
refinement study for the $P^2$-$P^1$ element by fixing $h=1/128$ and
varying $\Delta t$ from $1/2$ to $1/16$. Numerical convergence
results are reported in Table \ref{tab:EXP1_rate1_time}, and
convergence histories of all primary variables are displayed in
Figure \ref{fig:EXP1_rate1_time} via a log-log plot. The velocity
error in both $L^2$ and $H^1$ norms, and the pressure error in $L^2$
norm all exhibit approximately first-order temporal convergence,
providing another numerical evidence of first-order convergence in
time for all three errors overall. Therefore, for the purpose of
simplicity, we assess first-order temporal accuracy in the remaining
numerical experiments using a combined space-time refinement study
with a choice of \(\Delta t \propto h^3\), only, without doing a
separate temporal refinement study again.}
\begin{table}[htbp]
    \centering
    \caption{\textcolor{black}{Temporal convergence performance with a fixed spatial mesh size of $h=1/128$.}}
    \label{tab:EXP1_rate1_time}
    \textcolor{black}{
    \begin{tabular}{ccccccc}
        \toprule
        $\Delta t$ & $\|u-u_h\|_0$ & Order & $\|u-u_h\|_1$ & Order & $\|p-p_h\|_0$ & Order \\
        \midrule
        $1/2$ & $2.386344\times10^{-4}$ & $-$ & $3.256486\times10^{-3}$ & $-$ & $8.255548\times10^{-3}$ & $-$ \\
        $1/4$ & $1.340949\times10^{-4}$ & $0.83$ & $1.827482\times10^{-3}$ & $0.83$ & $4.624786\times10^{-3}$ & $0.84$ \\
        $1/8$ & $6.963938\times10^{-5}$ & $0.95$ & $9.567281\times10^{-4}$ & $0.93$ & $2.388394\times10^{-3}$ & $0.95$ \\
        $1/16$ & $3.520622\times10^{-5}$ & $0.98$ & $5.054395\times10^{-4}$ & $0.92$ & $1.202761\times10^{-3}$ & $0.99$ \\
        \bottomrule
    \end{tabular}
    }
\end{table}

\begin{figure}[htbp]
    \centering
    \includegraphics[width=0.75\textwidth]{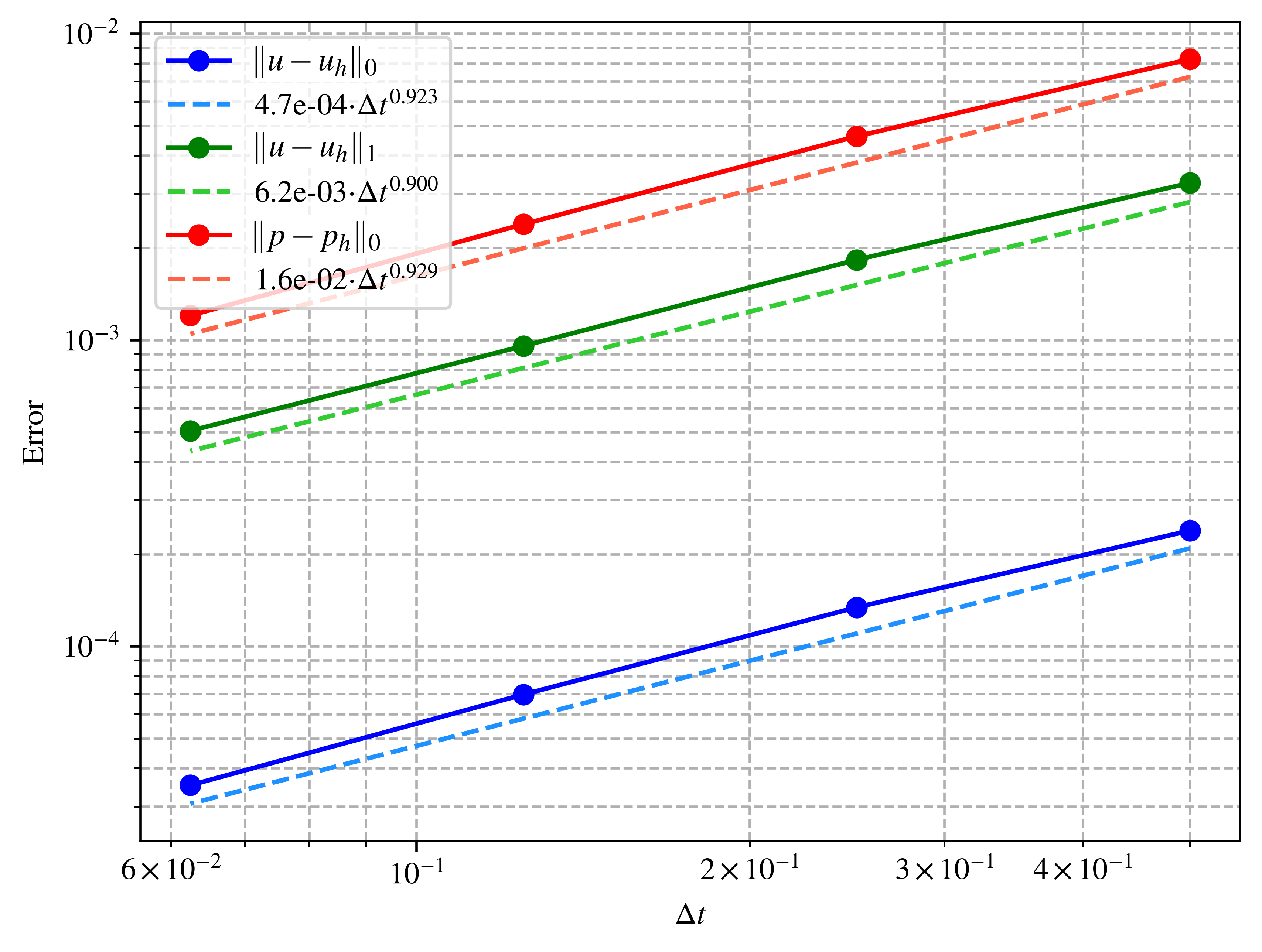}
    \caption{\textcolor{black}{Temporal convergence history with a fixed spatial mesh size of $h=1/128$.}}
    \label{fig:EXP1_rate1_time}
\end{figure}

\textcolor{black}{We further investigate the MINI-element
approximation, with the corresponding velocity and pressure spaces
defined as follows:
\begin{align*}
    \bs{U}_{h,\mathrm{MINI}}^t&=\left\{(\bs{v}_{h,1},\bs{v}_{h,2})\in\bs{U}^t\,\bigg|\,\bs{v}_{h,i}\circ\bs{X}_{h,i}^t\in(H^1(\Omega_i^0))^d,\ (\bs{v}_{h,i}\circ\bs{X}_{h,i}^t)|_K\in(P_b^1(K))^d,\ \forall K\in\mathcal{T}_{h,i}^0,\ i=1,2\right\},\\
    Q_{h,\mathrm{MINI}}^t&=\left\{(q_{h,1},q_{h,2})\in Q^t\,\bigg|\,q_{h,i}\circ\bs{X}_{h,i}^t\in H^1(\Omega_i^0),\ (q_{h,i}\circ\bs{X}_{h,i}^t)|_K\in P^1(K),\ \forall K\in\mathcal{T}_{h,i}^0,\ i=1,2\right\},
\end{align*}
where $P_b^1(K):=P^1(K)\oplus B_1(K)$ denotes the bubble-enriched
linear polynomial space, with $B_1(K)$ being the local
element-bubble space, and the pressure satisfies the global
zero-mean condition. Under the corresponding coupled inf-sup
condition and standard approximation properties, the arguments of
Lemmas 3.1-3.4 yield the same optimal-order estimates for the MINI
pair, which in turn give the corresponding optimal-order bounds for
the semi-discrete and fully discrete schemes in space, i.e., the
velocity approximation exhibits second-order convergence in \(L^2\)
norm and first-order convergence in \(H^1\) norm, while the pressure
approximation converges at first order in \(L^2\) norm. We use the
MINI element below to illustrate its numerical performance.} In this
case, the mesh resolution is again successively refined from $h=1/8$
to $h=1/64$, while the time step size \textcolor{black}{is chosen as
$\Delta t=4h^2$}. The corresponding numerical convergence results
are presented in Table \ref{tab:EXP1_rate2}, and the convergence
histories are shown in Figure \ref{fig:EXP1_rate2}. It can be
observed that both the velocity error in $H^1$ norm and the pressure
error in $L^2$ norm converge with first-order accuracy, while the
velocity error in $L^2$ norm exhibits second-order convergence.
\textcolor{black}{These combined space-time refinement results are
consistent with the expected first-order temporal accuracy under the
choice $\Delta t\propto h^2$. These observations further support the
expected convergence behavior. Hence for simplicity, in the rest
numerical experiments we test the spatial and temporal convergence
orders using the $P^2$-$P^1$ Taylor-Hood element only, without
performing a separate MINI-element convergence test again.}
\begin{table}[htbp]
    \centering
    \caption{\textcolor{black}{Convergence performance of the MINI element for the case of prescribed interface motion.}}
    \label{tab:EXP1_rate2}
    \textcolor{black}{
    \begin{tabular}{ccccccc}
        \toprule
        $h$ & $\|u-u_h\|_0$ & Order & $\|u-u_h\|_1$ & Order & $\|p-p_h\|_0$ & Order \\
        \midrule
        $1/8$  & $1.140422\times10^{-2}$ & $-$ & $4.521263\times10^{-1}$ & $-$ & $2.927666\times10^{-1}$ & $-$ \\
        $1/16$ & $2.501960\times10^{-3}$ & $2.19$ & $1.999182\times10^{-1}$ & $1.18$ & $1.055651\times10^{-1}$ & $1.47$ \\
        $1/32$ & $6.270133\times10^{-4}$ & $2.00$ & $1.019550\times10^{-1}$ & $0.97$ & $5.498307\times10^{-2}$ & $0.94$ \\
        $1/64$ & $1.546018\times10^{-4}$ & $2.02$ & $4.879125\times10^{-2}$ & $1.06$ & $2.168833\times10^{-2}$ & $1.34$ \\
        \bottomrule
    \end{tabular}
    }
\end{table}
\begin{figure}[htbp]
    \centering
    \includegraphics[width=0.75\textwidth]{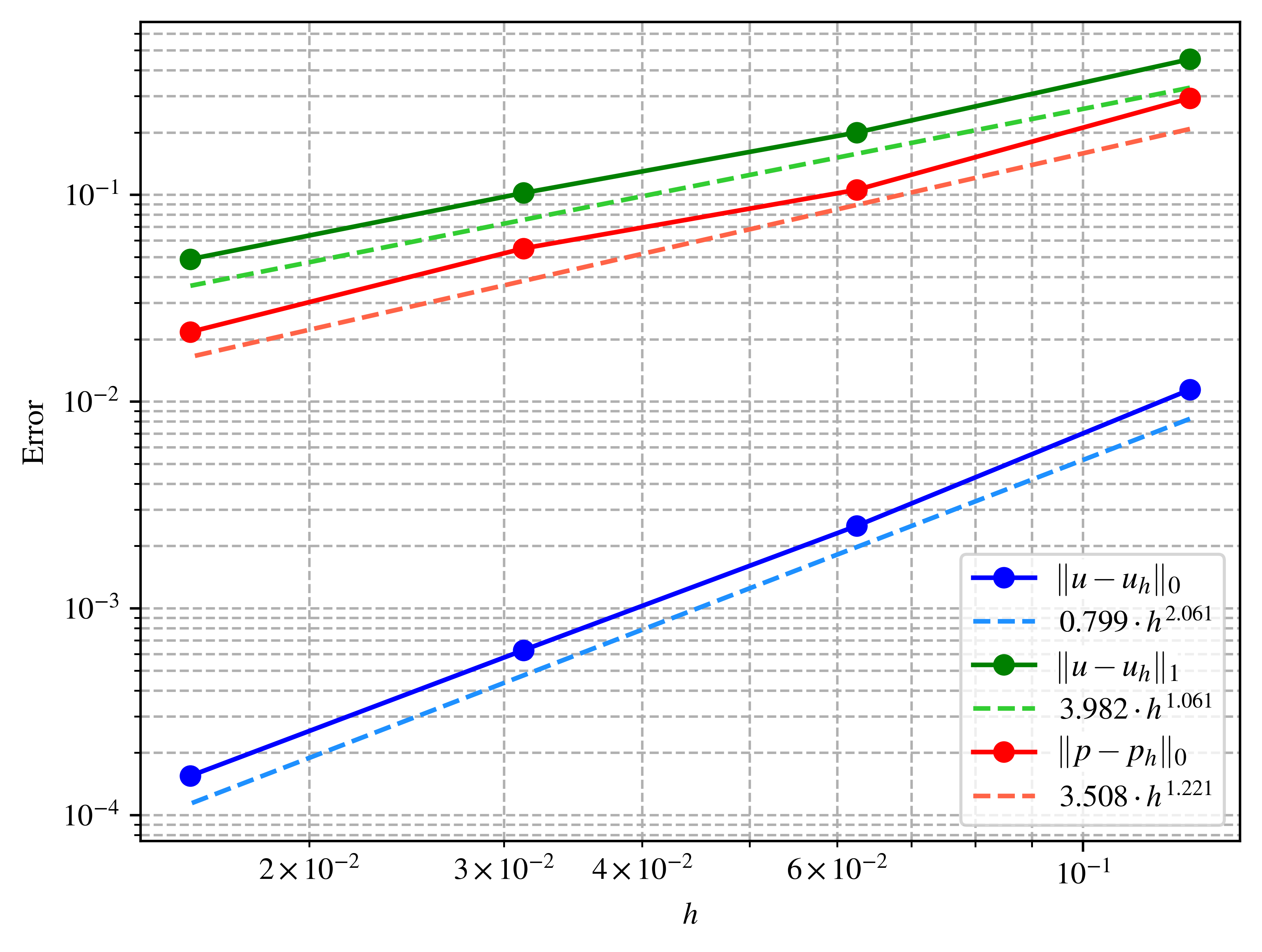}
    \caption{\textcolor{black}{Convergence history of the MINI element for the case of prescribed interface motion.}}
    \label{fig:EXP1_rate2}
\end{figure}

\subsection{Solution-driven interface motion}

We next consider a two-dimensional numerical example in which the
evolution of the domain, including the interface motion, is
determined by the velocity solution. Let
$\textcolor{black}{\Omega^0=[0,1]\times[0,1]}$ and $T=1$. For an
arbitrary smooth function $C(t)>0$, we construct the following exact
solution $\bs{u}=(u,v)^\mathrm{T}$ and pressure $p$ for problem
\eqref{prob1}-\eqref{prob11}:
\begin{align*}
    u&=-\frac{C'(t)}{C(t)}(x-(x_0+w_1t))+w_1,\\
    v&=\frac{C'(t)}{C(t)}(y-(y_0+w_2t))+w_2,\\
    p&=\cos(\pi x)\cos(\pi y)(1-e^{-t}).
\end{align*}
Consider the initial elliptical interface defined by
\begin{equation*}
    C(0)^2(x-x_0)^2+C(0)^{-2}(y-y_0)^2-R^2=0.
\end{equation*}
Under the evolution induced by the velocity field $\bs{u}$, the interface remains elliptical for all $t\in(0,T]$ and satisfies
\begin{equation*}
    F(x,y,t):=C(t)^2(x-(x_0+w_1t))^2+C(t)^{-2}(y-(y_0+w_2t))^2-R^2=0.
\end{equation*}
Indeed, this follows immediately from the identity
$\dfrac{DF}{Dt}=F_t(x,y,t)+uF_x(x,y,t)+vF_y(x,y,t)=0$, which holds
for the above definitions of $u,v$ and $F$, showing that the
interface is transported exactly by the fluid velocity. Here,
$(x_0,y_0)=(0.5,0.5)$ denotes the initial center of the ellipse,
$(w_1,w_2)^\mathrm{T}=(0.1,0.05)^\mathrm{T}$ represents the
translational velocity of the ellipse center, and $R=0.15$ is
related to the size of the initial ellipse. \textcolor{black}{For
$C(t)=1+t/5$, the flow map generated by the above velocity field
sends $(\hat x,\hat y)\in\Omega^0$ to
\begin{equation*}
    x=x_0+w_1t+\frac{\hat x-x_0}{C(t)},\qquad y=y_0+w_2t+C(t)(\hat y-y_0).
\end{equation*}
Consequently, the outer domain evolves as
\begin{equation*}
    \Omega^t=\left(x_0+w_1t-\frac{1}{2C(t)},x_0+w_1t+\frac{1}{2C(t)}\right)\times\left(y_0+w_2t-\frac{C(t)}{2},y_0+w_2t+\frac{C(t)}{2}\right).
\end{equation*}
The nonhomogeneous Dirichlet data on the outer boundary are prescribed by the trace of the exact velocity field.}
This constructed solution satisfies the regularity assumptions
\eqref{reg1}-\eqref{reg5}, and it also holds that
$\nabla\cdot\bs{u}=0$ in $\Omega_i^t$ ($i=1,2$), while the velocity
is continuous across the interface $\Gamma^t$. In this example, we
take $\mu_1=1$, $\mu_2=1000$ and $C(t)=1+t/5$, with the terms
$\bs{f}_1,\bs{f}_2$ and $\bs{g}$ chosen appropriately so that the
exact solution satisfies the two-dimensional Stokes moving interface
problem \eqref{prob1}-\eqref{prob11}. Figure \ref{fig:EXP2_domain}
illustrates the initial and terminal domains together with the
corresponding meshes induced by the velocity solution.

\begin{figure}[htbp]
    \centering
    \includegraphics[width=0.54\textwidth]{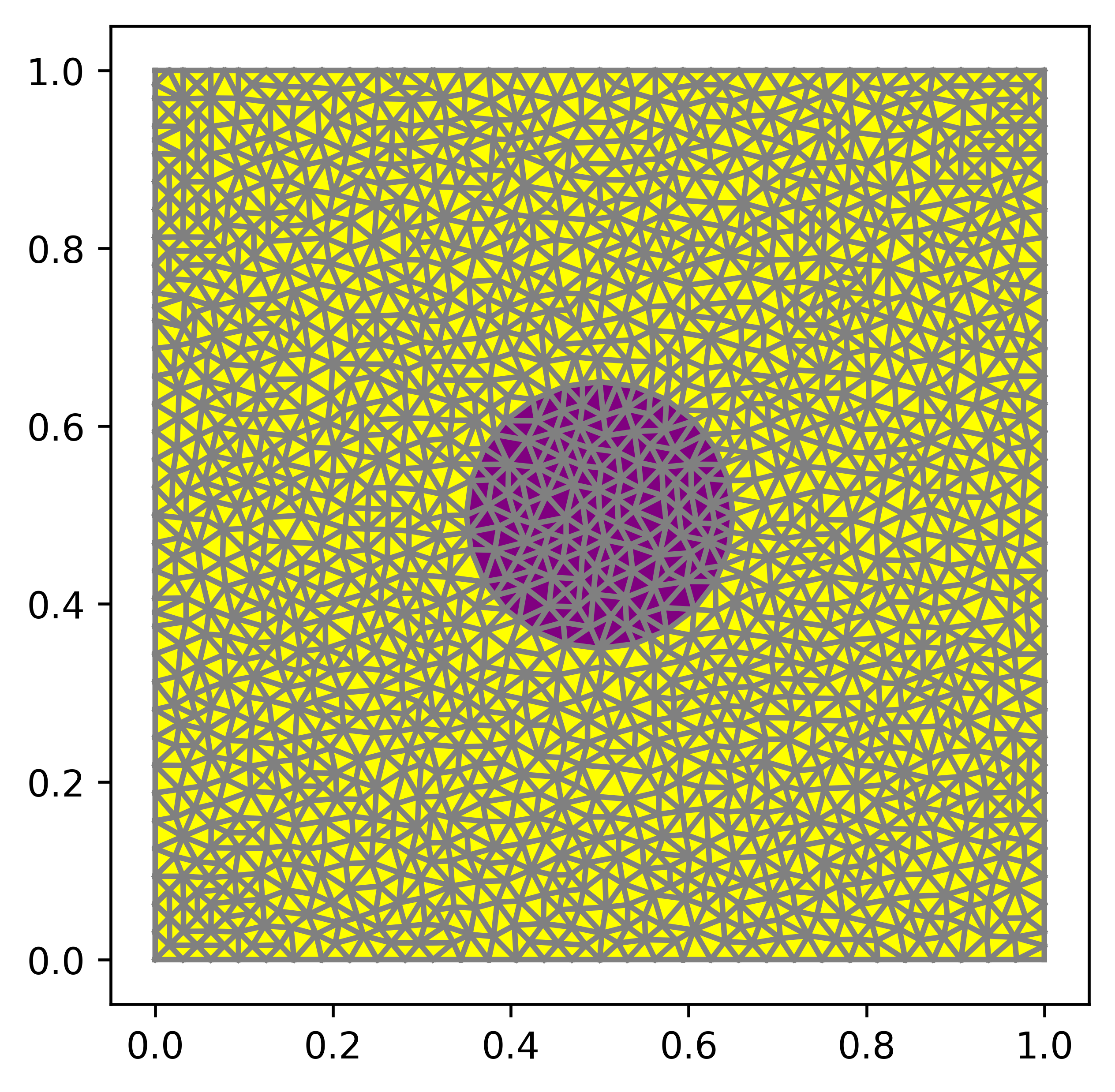}
    \includegraphics[width=0.45\textwidth]{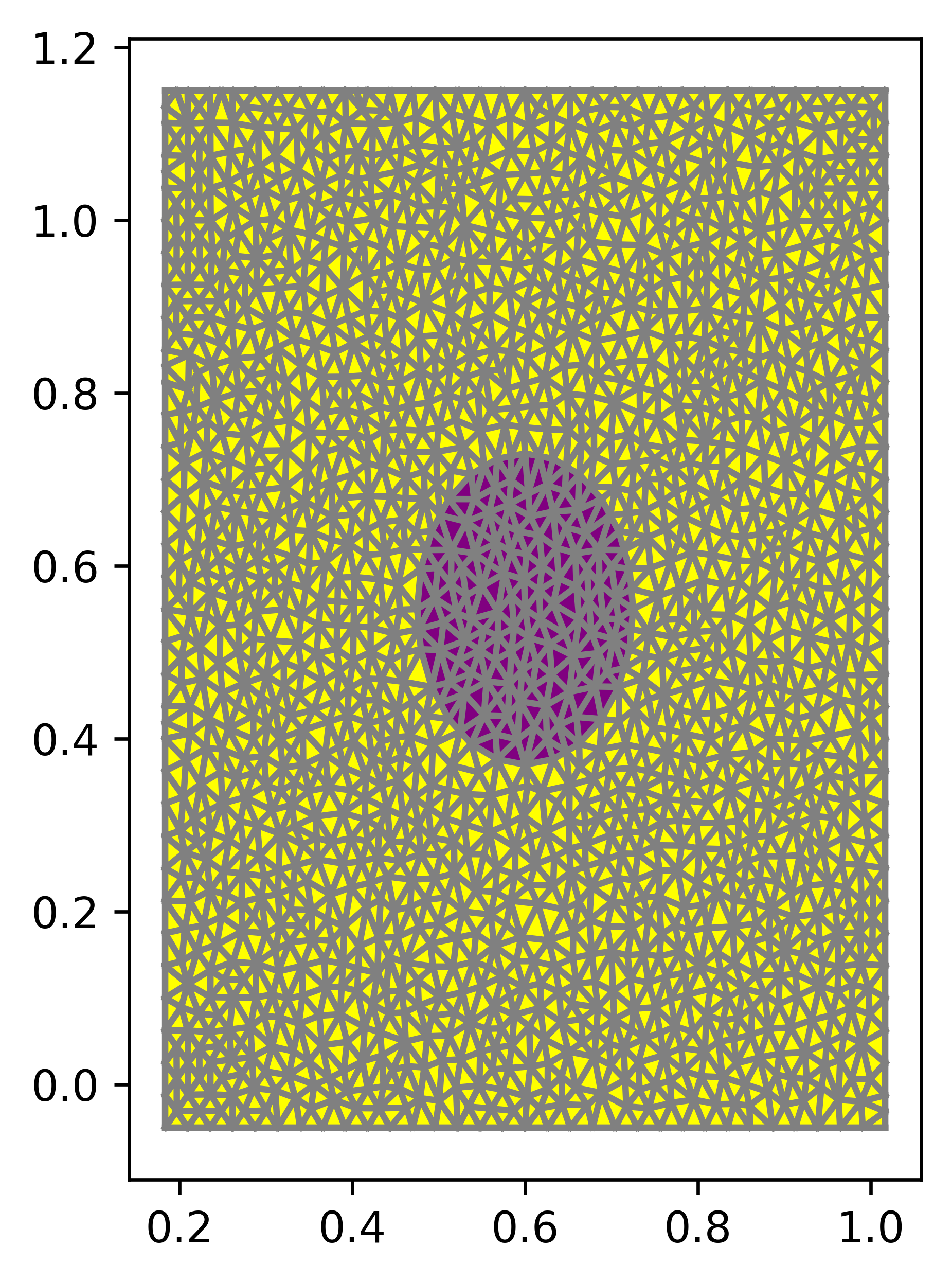}
    \caption{The initial (left) and terminal (right) domains and meshes with the mesh size $h=1/32$.}
    \label{fig:EXP2_domain}
\end{figure}

We first consider the case in which the interface evolution is
driven by the exact solution. \textcolor{black}{The outer-boundary
and interface motions are prescribed by the trace of the exact
velocity field, and the remaining mesh nodes are determined by
harmonic extension. Since this mesh motion is prescribed
independently of the discrete solution of two-phase flow solver, the
theoretical analysis developed earlier remains applicable under the
stated regularity assumptions. The discrete mesh velocity is
obtained by differentiating the discrete ALE mapping, and the
estimates established in \eqref{X-conv1}-\eqref{w-conv2} apply to
this prescribed mesh motion.} As in the previous example, we use the
$P^2$-$P^1$ element to solve the Stokes moving interface problem for
$((\bs{u}_1,\bs{u}_2),(p_1,p_2))$. The mesh size is successively
refined from $h=1/4$ to $h=1/32$, while the time step size satisfies
\textcolor{black}{$\Delta t=4h^3$}. Numerical convergence results
are reported in Table \ref{tab:EXP2_rate1}, and the convergence
histories are shown in Figure \ref{fig:EXP2_rate1}.
\textcolor{black}{The coupled space-time refinement results show
second-order convergence of the velocity error in $H^1$ norm and the
pressure error in $L^2$ norm, together with third-order convergence
of the velocity error in $L^2$ norm. These observations are in good
agreement with the theoretically predicted spatial convergence rates
and first-order temporal accuracy.}


\begin{table}[htbp]
    \centering
    \caption{{Convergence performance of the $P^2$-$P^1$ element for the case of interface motion driven by the exact solution.}}
    \label{tab:EXP2_rate1}
    \textcolor{black}{
    \begin{tabular}{ccccccc}
        \toprule
        $h$ & $\|u-u_h\|_0$ & Order & $\|u-u_h\|_1$ & Order & $\|p-p_h\|_0$ & Order \\
        \midrule
        $1/4$  & $1.932615\times10^{-4}$ & $-$ & $4.827079\times10^{-3}$ & $-$ & $1.121253\times10^{-2}$ & $-$ \\
        $1/8$  & $2.512777\times10^{-5}$ & $2.94$ & $1.238703\times10^{-3}$ & $1.96$ & $3.088942\times10^{-3}$ & $1.86$ \\
        $1/16$ & $3.481649\times10^{-6}$ & $2.85$ & $3.311782\times10^{-4}$ & $1.90$ & $7.547299\times10^{-4}$ & $2.03$ \\
        $1/32$ & $4.602643\times10^{-7}$ & $2.92$ & $8.607634\times10^{-5}$ & $1.94$ & $1.843074\times10^{-4}$ & $2.03$ \\
        \bottomrule
    \end{tabular}
    }
\end{table}

\begin{figure}[htbp]
    \centering
    \includegraphics[width=0.75\textwidth]{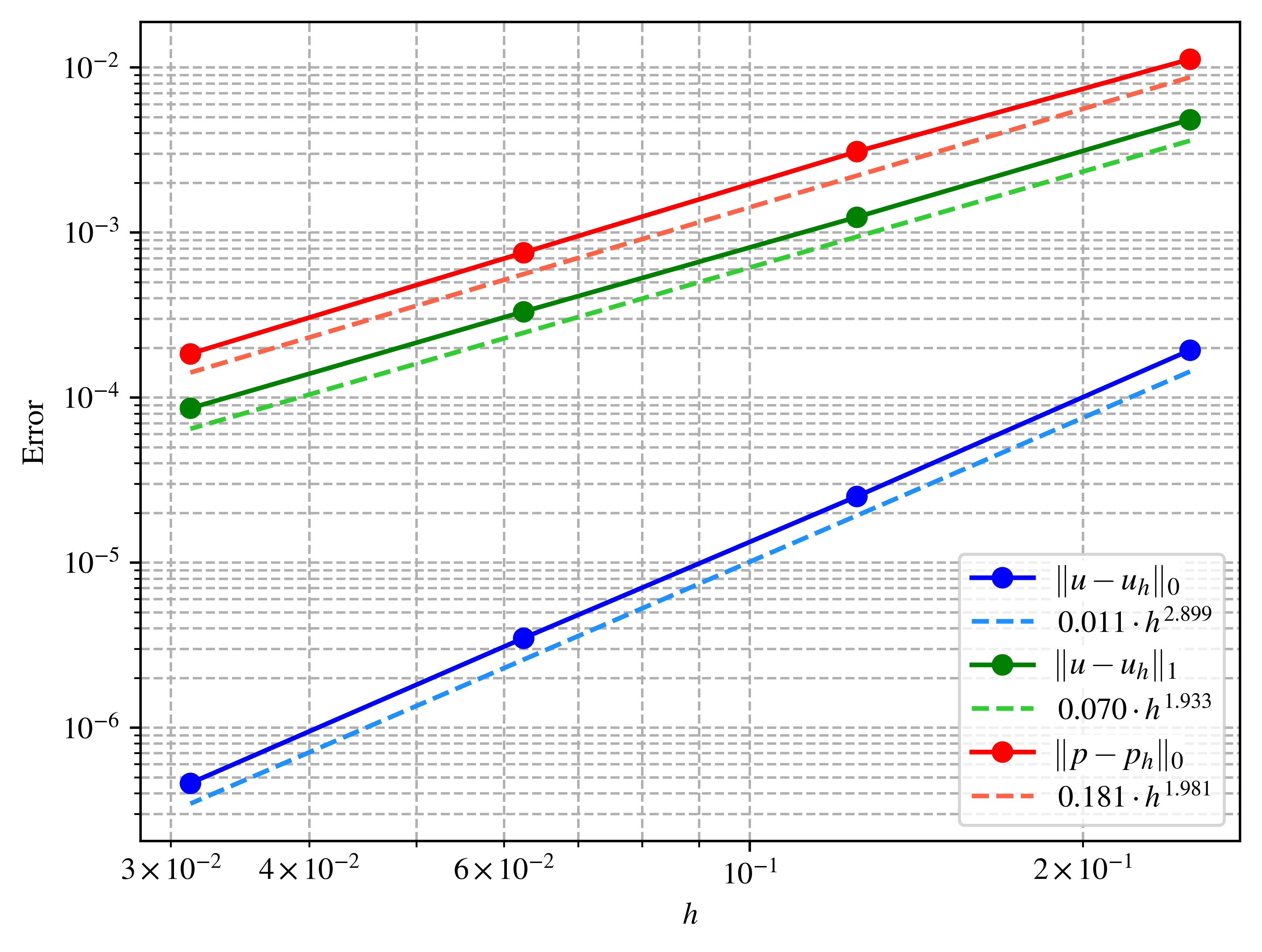}
    \caption{Convergence history of the $P^2$-$P^1$ element for the case of interface motion driven by the exact solution.}
    \label{fig:EXP2_rate1}
\end{figure}

Finally, we consider the more practical situation in which the
interface evolution is driven by the numerical velocity solution
rather than the exact solution. In realistic applications, the exact
solution is generally unavailable, and the interface motion must
therefore be determined solely from the numerical solution. In this
setting, we still solve the discrete ALE mapping $\bs{X}_{h,i}^t$ on
$\Omega_i^0$ to generate the moving meshes $\mathcal{T}_{h,i}^t$
($i=1,2$) for $t\in[0,T]$. However, unlike the previous cases, the
interface displacement is no longer prescribed and instead depends
explicitly on the numerical velocity solution. \textcolor{black}{At
each time step, the interface nodes are updated by the implicit
Euler relation
\begin{equation*}
    \bs{x}_{\Gamma,h}^{n+1}=\bs{x}_{\Gamma,h}^{n}+\Delta t\,\bs{u}_h^{n+1}(\bs{x}_{\Gamma,h}^{n+1}).
\end{equation*}
The numerical velocity is evaluated at the new time level $t^{n+1}$.
Since this velocity depends on the updated mesh, the discrete
two-phase flow solver and interface update are iterated within each
time step until the numerical velocity used in the interface update
agrees with the velocity returned by the discrete flow solver on the
updated mesh. The outer boundary follows the prescribed velocity
boundary data, and the remaining interior mesh nodes are determined
by harmonic extension. The discrete mesh velocity is then obtained
by the difference quotient pulled back to the current configuration,
as defined in Section 5.1.} A rigorous theoretical analysis for such
interface problems, where the mesh motion is coupled with the
numerical solution itself, will be the subject of future work.
Nevertheless, we again employ the $P^2$-$P^1$ element to solve the
Stokes moving interface problem for
$((\bs{u}_1,\bs{u}_2),(p_1,p_2))$. The mesh size is refined
successively from $h=1/4$ to $h=1/32$, with
\textcolor{black}{$\Delta t=4h^3$}. The numerical convergence
results are reported in Table \ref{tab:EXP2_rate2}, and the
convergence histories are displayed in Figure \ref{fig:EXP2_rate2}.
\textcolor{black}{The numerical results show second-order
convergence of the velocity error in $H^1$ norm and the pressure
error in $L^2$ norm, together with third-order convergence of the
velocity error in $L^2$ norm along the chosen coupled space-time
refinement path. These results are consistent with the spatial
convergence rates and first-order temporal accuracy observed in the
preceding examples/cases. However, in this case, the coupling
between mesh motion and the numerical solution makes a rigorous
convergence analysis for solution-driven interface updates more
challenging than for prescribed interface motion. We therefore leave
this analysis for future work.}


\begin{table}[htbp]
    \centering
    \caption{{Convergence performance of the $P^2$-$P^1$ element for the case of interface motion driven by the numerical solution.}}
    \label{tab:EXP2_rate2}
    \textcolor{black}{
    \begin{tabular}{ccccccc}
        \toprule
        $h$ & $\|u-u_h\|_0$ & Order & $\|u-u_h\|_1$ & Order & $\|p-p_h\|_0$ & Order \\
        \midrule
        $1/4$  & $1.930112\times10^{-4}$ & $-$ & $4.823160\times10^{-3}$ & $-$ & $1.121460\times10^{-2}$ & $-$ \\
        $1/8$  & $2.513006\times10^{-5}$ & $2.94$ & $1.238727\times10^{-3}$ & $1.96$ & $3.089176\times10^{-3}$ & $1.86$ \\
        $1/16$ & $3.481615\times10^{-6}$ & $2.85$ & $3.311766\times10^{-4}$ & $1.90$ & $7.547321\times10^{-4}$ & $2.03$ \\
        $1/32$ & $4.602639\times10^{-7}$ & $2.92$ & $8.607629\times10^{-5}$ & $1.94$ & $1.843075\times10^{-4}$ & $2.03$ \\
        \bottomrule
    \end{tabular}
    }
\end{table}
\begin{figure}[htbp]
    \centering
    \includegraphics[width=0.75\textwidth]{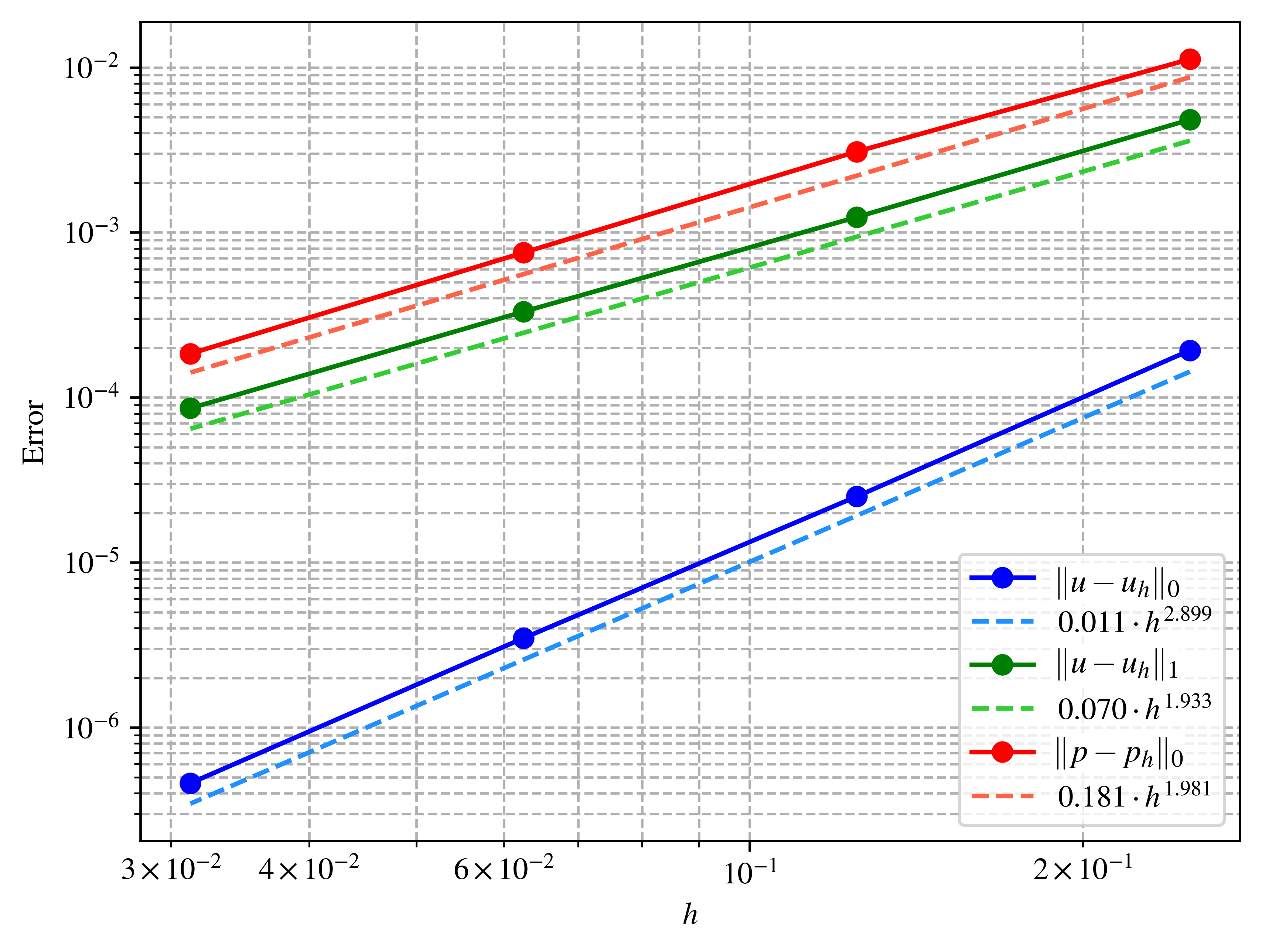}
    \caption{Convergence history of the $P^2$-$P^1$ element for the case of interface motion driven by the numerical solution.}
    \label{fig:EXP2_rate2}
\end{figure}

\section{Conclusion and future work}\label{sec:conclusion}

In this paper, we develop an arbitrary Lagrangian-Eulerian-finite
element method (ALE-FEM) for a class of Stokes moving interface
problems with jump coefficients, with the mesh motion suitably
prescribed. Both semi- and fully discrete schemes are constructed
based on an appropriate ALE mapping, and Taylor-Hood elements are
employed for the numerical approximation, while the theoretical
analysis can \textcolor{black}{also be adapted to MINI elements
under the corresponding approximation and inf-sup properties}. A key
ingredient of the analysis is the \textcolor{black}{use of a
specific} $H^1$-projection associated with the moving interface
problem, which takes into account the effect of mesh motion induced
by the ALE formulation. Optimal approximation properties of this
projection and its ALE-time derivative are established in both $H^1$
and $L^2$ norms. Based on these results, optimal \textcolor{black}{
$H^1$ and $L^2$} error estimates are derived for both semi- and
fully discrete ALE-finite element schemes. Compared with existing
ALE-based analyses, the present work provides a systematic treatment
of projection errors induced by moving meshes, including those of
the $H^1$-projection and its ALE-time derivative, and achieves
optimal convergence rates in both norms for Stokes interface
problems with jump coefficients. Numerical results confirm the
theoretical findings. The analytical framework developed here can be
further extended to more general two-phase flow problems as well as
fluid-structure interaction systems, which will be investigated in
our future work.

\section*{Acknowledgments}

C. Wang and Y. Liang were partially supported by National Natural
Science Foundation of China grant (No. 12171366), Y. Liang was
partially supported by International Exchange Program for Graduate
Students, Tongji University, P. Sun was partially supported by a
grant from the Simons Foundation (MPS-TSM-00706640).

\section*{Conflicts of Interest}
The authors declare that they have no conflicts of interest to this
work.

\section*{Data Availability Statement}
No data were used for the research described in the article.

\bibliographystyle{unsrt}
\bibliography{reference}

\end{document}